\documentclass[envcountsect,envcountsame]{svjour3}                     % onecolumn (standard format)
\usepackage{graphicx}
\usepackage{float}
\usepackage{cite}

\usepackage{amssymb,amsmath,amsthm}
\usepackage{mathtools}
\usepackage{fixmath}
\usepackage{braket}
\usepackage{caption}
\usepackage{subcaption}
\usepackage{enumitem}
\usepackage{units}
\usepackage{xcolor}
\usepackage{booktabs}
\usepackage[english]{babel}
\usepackage[T1]{fontenc}
\usepackage{algorithm}
\usepackage{algorithmic}

\usepackage{newtxtext,newtxmath}
\usepackage{fixmath}

\usepackage[colorlinks=true,linkcolor=black,urlcolor=black,citecolor=black,pdftex]{hyperref}
\hypersetup{final} %overwrite global draft mode

\usepackage{tikz}
\usepackage{pgfplots}
\usetikzlibrary{external}
\usepackage{todonotes}

\newcommand{\N}{\mathbb{N}}
\newcommand{\R}{\mathbb{R}}
\newcommand{\de}{\mathop{}\!\mathrm{d}}

\DeclareMathOperator{\supp}{supp}
\DeclareMathOperator{\Lap}{\Delta}
\DeclareMathOperator{\fish}{\mathcal{I}}
\DeclareMathOperator{\dom}{dom}
\DeclareMathOperator{\Tr}{Tr}
\DeclareMathOperator*{\argmin}{arg\,min}
\DeclareMathOperator*{\argmax}{arg\,max}

\newcommand{\sensv}{\partial S[\hat{q}]}
\newcommand{\sens}[1]{\partial_{#1} S[\hat{q}]}

\newcommand{\mnorm}[1]{\|#1\|_{M(\Om)}}
\newcommand{\tr}[1]{tr(C(1)^{-1})}
\newcommand{\optom}{\bar{\omega}}
\newcommand{\optomb}{\bar{\omega}_{\beta}}

\newcommand{\optombh}{\bar{\omega}_{\beta,h}}

\newcommand{\summ}[2]{\sum_{#2=1}^{#1}}
\newcommand{\maxnorm}[1]{\|#1\|_{C(\Om)}}
\newcommand{\Heins}{H_0^1(\Omega)}

\newcommand{\eps}{\varepsilon}
\newcommand{\om}{\omega}
\newcommand{\Om}{\Omega_o}

\newcommand{\vertiii}[1]{{\left\vert\kern-0.25ex\left\vert\kern-0.25ex\left\vert #1
    \right\vert\kern-0.25ex\right\vert\kern-0.25ex\right\vert}}

\newcommand{\pair}[1]{\langle #1 \rangle}
\newcommand{\norm}[1]{\lVert#1\rVert}

\newcommand*{\omparam}{\boldsymbol{\omega}}

\newcommand*{\changed}[1]{{\color{black} #1}}

\numberwithin{equation}{section}

\graphicspath{{./tikz/}{./figures/}}

\theoremstyle{definition}
\newtheorem{assumption}{Assumption}

\makeatletter
\def\namedlabel#1#2{\begingroup
    #2%
    \def\@currentlabel{#2}%
    \phantomsection\label{#1}\endgroup
}
\makeatother

\begin{document}

\title{%
  A sparse control approach to optimal sensor placement in PDE-constrained parameter estimation problems
}

\titlerunning{\changed{Sparse control for optimal sensor placement}}        % if too long for running head

\author{%
  Ira Neitzel
  \and Konstantin Pieper
  \thanks{
    K. Pieper acknowledges funding by the US
    Department of Energy Office of Science grant DE-SC0016591 and by the US
    Air Force Office of Scientific Research grant FA9550-15-1-0001.
    }
  \and Boris~Vexler
  \and Daniel Walter
  \thanks{
    D. Walter acknowledges support by the DFG through the International Research
    Training Group IGDK 1754 ``Optimization and Numerical Analysis for Partial Differential
    Equations with Nonsmooth Structures''. Furthermore, support from the TopMath Graduate
    Center of TUM Graduate School at Technische Universit{\"a}t M{\"u}nchen, Germany and
    from the TopMath Program at the Elite Network of Bavaria is gratefully acknowledged.
  }
}

%\authorrunning{} % if too long for running head

\institute{
  Ira Neitzel\at
  Institut f{\"u}r Numerische Simulation, Universit{\"a}t Bonn, 
  Wegelerstr. 6, 
  53115 Bonn, 
  Germany \\
  \email{neitzel@ins.uni-bonn.de}           %  \\
  % \emph{Present address:} of F. Author  %  if needed
  \and
  Konstantin Pieper \at
  Dept. Scientific Computing,
  Florida State University,
  400 Dirac Science Library,
  Tallahassee, FL 32306, USA  
  \\
  \email{kpieper@fsu.edu}   
  \and 
  Boris Vexler 
  \and
  Daniel Walter \at
  %Chair of Optimal Control,
  Center for Mathematical Sciences, Chair M17,
  Technische Universit{\"a}t M{\"u}nchen,
  Boltzmannstr. 3,
  85748 Garching bei M{\"u}nchen, Germany  \\
  \email{vexler@ma.tum.de, walter@ma.tum.de} \\
  Tel.: +49-89-289-18340
}

%\journalname{myjournal}

\date{Received: date / Accepted: date}

\maketitle

\begin{abstract}
  We present a systematic approach to the optimal placement of finitely many sensors in
  order to infer a finite-dimensional parameter from point evaluations of the solution of
  an associated parameter-dependent elliptic PDE. The quality of the corresponding least
  squares estimator is quantified by properties of the asymptotic covariance
  matrix depending on the distribution of the measurement sensors. We formulate a design
  problem where we minimize functionals related to the
  size of the corresponding confidence regions with respect to the position and number
  of pointwise measurements. The measurement setup is modeled by a positive Borel measure
  on the spatial experimental domain resulting in a convex optimization problem.
  For the algorithmic solution a class of accelerated conditional gradient methods in
  measure space is derived, which exploits the structural properties of the design problem to ensure convergence
  towards sparse solutions. Convergence properties are presented 
  %and our findings are compared to previous results. A
  %variational discretization of the continuous problem based on finite elements is investigated
  and the presented results are illustrated by numerical experiments.
%\PACS{PACS code1 \and PACS code2 \and more}
\end{abstract}

\section{Introduction}
In this paper we propose a measure-valued formulation for the optimal design of a
measurement setup for the identification of an unknown parameter vector entering a
system of partial differential equations. Many applications in physics, medicine  or
chemical engineering rely on complex mathematical models as a surrogate for
real-life processes. Typically the arising equations contain unknown (material) parameters
which have to be identified in order to obtain a realistic model for the simulation of the
underlying phenomenon. To illustrate the ideas, we consider a similar example as presented in
\cite{becker2005parameter}. Here, the combustion process of a single substance on a two
dimensional domain $\Omega$ is modeled by a non-linear convection-diffusion equation with
an Arrhenius-type reaction term, depending on four scalar parameters $D$, $E$, $d$, and $c$, representing
its material properties:
\begin{equation}\label{Combustion}
\begin{aligned}
 -\Lap y+\alpha\cdot \nabla y+D \operatorname{exp}\left\{-E/(d-y)\right\}y(c-y)&=0\quad \text{in}~\Omega,
\end{aligned}
\end{equation}
together with $y=\hat{y}$ on an inflow boundary $\Gamma_{\mathrm{in}}\subset \partial \Omega$ and $\partial_n y=0$ on $\partial \Omega \setminus \Gamma_{\mathrm{in}}$.
While $c$ and $d$ are known physical constants, the pre-exponential factor $D$ and the
activation energy $E$ are empirical and cannot be measured directly.
{\color{black}
Therefore one often has to rely on experimental data, for instance measurements of
the mole fraction \(y\).
}
An estimate for the true parameters is then obtained by finding a parameter vector
matching the collected data, which leads to a least-squares problem constrained by a partial
differential equation. However, due to errors in the measurement process the obtained
estimate is biased and could be far from the value which describes the physical process
most accurately.
This bias has to be quantified and the measurement procedure has to be adapted to mitigate the influence of
the  perturbed data.

In this manuscript, we consider a general PDE-model based on a parameter-dependent weak
formulation with an unknown parameter vector \(q\) in an admissible set
\(Q_{ad} \subset \R^d\) (for instance, \(q = (D,E) \in \R^2\) for~\eqref{Combustion}).
We refer to Section~\ref{sec:paraest} for the precise assumptions.
The parameter is estimated from point-wise observations of the solution \(y = S[q]\) of the
PDE-model at points \(\set{x_j}_j^m \subset \Om\), where \(\Om \subset \bar\Omega\) is a
closed set covering the possible observation locations.
We choose optimal designs according to criteria based on a
linearization of the model equation. To this purpose, we define the associated
sensitivities $\{\partial_k S[\hat{q}]\}^n_{k=1}$ of \(S[\hat{q}]\) with respect to
pertubations of each parameter \(q_k\), \(k = 1,\ldots,n\) at an initial guess $\hat{q}\in Q_{ad}$, stemming
either from prior knowledge or obtained from previous experiments.
We note that optimal design approaches based on first-order approximations have been studied for and
successfully applied to ordinary differential equations~\cite{avery2012experimental},
differential-algebraic equations~\cite{bauer2000numerical}, and also partial differential
equations~\cite{herzog2015sequentially}.
To each measurement location $x_j$ we assign a positive scalar $\lambda_j$ which is proportional to the
quality of the sensor at this location (or, alternatively corresponds to the number of
repeated measurements performed with an identical sensor).
Associated to the measurement setup is the design measure
\begin{equation}
\label{eq:design_measure}
\omparam(x,\lambda) = \summ{m}{j} \lambda_j \delta_{x_j},
\end{equation}
given by a weighted sum of Dirac delta functions.
To quantify the quality of a given measurement setup \(\omega\), we introduce the Fisher
information matrix $\fish(\om)$ with entries
\begin{equation}
\label{eq:fisher_in_intro}
\fish(\om)_{kl} = \int_{\Om} \sens{k}(x) \sens{l}(x) \de \om(x), \quad k,l \in \{1,\dots,n\}.
\end{equation}
%Clearly, $\left(\fish(\omparam(x,\lambda))\right)_{k,l}
%= \summ{m}{j} \sens{k}(x_j)\sens{l}(x_j)\lambda_j$, $k,l\in\{1,\dots,n\}$.
Furthermore, by $\Psi$ we denote a scalar quality criterion, which is a positive, smooth,
and convex functional acting on the symmetric, positive-definite matrices. 
Examples for possible choices of $\Psi$ can be found in, e.g.,
\cite{ucinski2004optimal,pukelsheim1993optimal}; see also Section~\ref{sec:existence_and_optimality}.
We consider optimal designs given by the solutions to the optimization problem
\begin{align}\label{def:pre_design}
 \min_{x_j \in \Om,\; \lambda_j \geq 0,\; j=1,\ldots,m}
  \quad\Psi(\fish(\omparam(x,\lambda)) + \fish_0) + \beta\summ{m}{j} \lambda_j,
\end{align}
where $\fish_0$ is a nonnegative-definite matrix (e.g., $\fish_0 = 0$).
It can be interpreted as a~priori knowledge on
the distribution of the estimator, which may be obtained from previously collected data, for
instance in the context of sequential optimal design;
cf.~\cite{korkel1999sequential}. Here, we would choose
$\fish_0=\fish(\om_{\mathrm{old}})$ where the design measure $\om_{\mathrm{old}}$
describes the previous experiments. Alternatively,
we may adopt a Bayesian viewpoint and consider $\fish_0$ as the
covariance matrix of a Gaussian prior.
The last term involving the cost
parameter $\beta > 0$ takes into account the overall cost of the measurement process. For
other optimal design approaches with sparsity promoting regularization we refer to, e.g.,
\cite{chung2012experimental,haber2008numerical, alexanderian2014optimal}.
We emphasize that we neither impose any restrictions on the number of measurements nor restrict the set of
candidate locations for the sensors to a finite set.

At first glance, problem~\eqref{def:pre_design} is a non-convex problem due to the
parameterization in terms of the points \(x_j\), and has a combinatorial aspect due to the
unknown number of measurements \(m\). However, we can bypass these difficulties by embedding the
problem into a more general abstract formulation:
introducing the set of positive Borel measures $M^+(\Om)$ on $\Om$ we
determine an optimal design measure from
\begin{align}
\label{def:designprop}\tag{\ensuremath{P_{\beta}}}
 \min_{\om \in M^+(\Om)} \Psi(\fish(\om)+\fish_0) + \beta \mnorm{\om},
\end{align}
where \(\mnorm{\om}\) is the canonical total variation norm.
While it is clear that~\eqref{def:designprop} is a more general formulation
than~\eqref{def:pre_design}, it can
be shown that it always admits solutions of the form $\om =\summ{m}{j} \lambda_{j}
\delta_{x_{j}}$ for some \(n \leq m \leq n(n+1)/2\), making both problems
essentially equivalent; see Section~\ref{sec:existence_and_optimality}. 
We give a derivation of~\eqref{def:pre_design} and its connection
to~\eqref{def:designprop} in Section~\ref{sec:fromparamtoopt}.

As an alternative to the penalization term $\beta \mnorm{\om}$ in \eqref{def:designprop}
it is possible to consider a fixed budget for the experiment leading to
\begin{align}
\label{def:constraineddesign}\tag{\ensuremath{P^K}}
 \min_{\om \in M^+(\Om)} \Psi(\fish(\om)+\fish_0)
  \quad\text{subject to}\quad
  \mnorm{\om} \leq K,
\end{align}
where $K>0$ denotes the overall maximal cost of the measurements. Under certain conditions
on $\Psi$ it can be shown that the inequality constraint in \eqref{def:constraineddesign}
is attained for every optimal design; see Proposition~\ref{prop:equivalencedesigns}.  
This relates \eqref{def:constraineddesign} closely to the
concept of approximate designs introduced by Kiefer and Wolfowitz in
\cite{kiefer1959optimum} for general linear-regression, where possible experiments are modeled by the
probability measures on~$\Om$. We refer also to
\cite{atkinson2007optimum, pukelsheim1993optimal, pazman1986foundations,
  fedorov1972theory, fedorov2013optimal} for the analysis of this kind of optimal design
formulations.
For the adaptation of this approach to parameter estimation in distributed systems we refer to
\cite{ucinski2004optimal,banks2014experimental}.
%Optimal solutions to \eqref{def:constraineddesign} and to our formulation
%\eqref{def:designprop} have to be interpreted differently.
%While the latter imposes a constraint on the cost of the overall measurement
%setup, which is defined as the sum of the quality of each sensor, the former adds this
%as a cost term to the minimization problem. 
Both formulations, \eqref{def:designprop} and \eqref{def:constraineddesign}, are closely
linked (see Section~\ref{sec:existence_and_optimality}): On the one hand, in the case of
no a~priori knowledge on the prior covariance, i.e.\ for \(\fish_0 = 0\), the
solutions of both problems coincide up to a scalar factor, depending on either \(K\) or
\(\beta\). On the other hand, incorporating a~priori knowledge, both problem formulations
parameterize the same solution manifold. The parameters~\(\beta\) and~\(K\), respectively, provide some indirect
control over the number of measurements, which is the cardinality of the support of the
optimal solution, in this case.

This paper is concerned with the analysis of
\eqref{def:designprop} and its efficient numerical solution. There exists
a large amount of literature on the solution of \eqref{def:constraineddesign} by
sequentially adding new Dirac delta functions to a sparse initial design measure. A description and
proofs of convergence for several variants of  these kind of methods can be found in,
e.g., \cite{fedorov1972theory,wynn1970} for the special case of
$\Psi(\cdot)=\operatorname{det}((\cdot)^{-1})$. These methods correspond to a conditional gradient,
or Frank-Wolfe~\cite{frank1956algorithm}, algorithm for minimizing the smooth functional
$\Psi(\fish(\cdot))$ over the ball with radius $K$ in $M^+(\Om)$.
Despite the ease of implementation the proposed methods suffer from some
drawbacks. On the one hand the speed of convergence is slow. Recently, in
\cite{boyd2015alternating} a sub-linear $\mathcal{O}(1/k)$ rate of convergence for the
error in the objective function in terms of the iteration number $k$ was proven by using an equivalent reformulation of
\eqref{def:constraineddesign} and results for the classical, finite dimensional conditional
gradient algorithm; see, e.g., \cite{jaggi2013revisiting}. Note, that without further
assumptions on $\Psi$ than convexity and for example
Lipschitz-continuity of its gradient, no better rate than $\mathcal{O}(1/k)$ can be expected in
general; see~\cite{dunn1979rates,dunn1980convergence}.
 
On the other hand, if only point insertion steps are considered, the support points of
the iterates tend to cluster around the optimal ones. To mitigate this effect and
accelerate the convergence, several modified variants of the sequential point insertion have
been proposed. In \cite{john1975review, atwood1973sequences} it is proposed to
alternate between point insertion steps and Wolfe's away steps (see \cite{Wolfe1970away}) to
remove mass from non-optimal points. Heuristically, adjacent support points may be lumped together;
see \cite{fedorov2012model}. More recently, several papers suggested to combine the
addition of a single Dirac-Delta in each iteration with the solution of a
finite-dimensional convex optimization problem and to apply point moving
\cite{boyd2015alternating} or vertex exchange methods \cite{yu2011cocktail}.  However, it
appears that there is no rigorous approach to guarantee the convergence of the
resulting algorithms towards a finitely supported optimal design on the function space
level.

In this paper we present a sequential point insertion algorithm for the (non-smooth)
optimal design problem \eqref{def:designprop} and prove convergence towards a
sparse minimizer of \eqref{def:designprop} comprising at most $n(n+1)/2$ support points.
To this purpose, we adapt the generalized conditional algorithm in measure space presented in
\cite{bredies2013inverse} for the minimization of a linear-quadratic Tikhonov-regularized
problem to our setting. Additionally we incorporate a post-processing step which ensures
that the support size of the generated iterates stays uniformly bounded. For further
sparsification and a practical acceleration of convergence we propose to alternate between
inserting several Dirac delta functions and point removal steps based on the (approximate) solution
of finite-dimensional $\ell_1$-regularized sub-problems, which are amenable for semi-smooth
Newton methods; see, e.g., \cite{ulbrich2002semismooth,milzarekfilter}.
A sublinear rate of convergence for the value of the objective function is proven for a
wide class of optimality criteria $\Psi$; see Theorem~\ref{thm:weakconvergencerate}.
%As a comparison, we report on a path-following approach based on a Hilbert space
%regularization of \eqref{def:designprop}; see~\cite{clason2011duality}.
\changed{Note that we do not employ acceleration strategies based on point
moving~\cite{bredies2013inverse,boyd2015alternating}, which are difficult to realize since
we will employ \(C^0\)-finite elements, which are not continuously differentiable, for the discretization of the underlying PDEs.}

%Finally, to solve \eqref{def:designprop} or~\eqref{def:constraineddesign} one has to compute the state
%$y=S[q]$ as well as the sensitivities of the state with respect to the parameters
%$\left\{ \partial_k S[q] \right\}^n_{k=1}$ for a given $q\in Q_{ad}$. In general, the
%state and sensitivity PDEs cannot be solved
%analytically, but only numerically. We present and analyze a discretization scheme for the
%optimal design problem by replacing the state (and therefore the corresponding
%sensitivities) by a sequence of finite-element approximations.
%To discretize the optimal design measure, we adapt the concept of variational discretization from
%\cite{casas2012,hinze2005variational}, and show the equivalence of a semidiscrete design
%problem with solutions in $M^+(\Om)$ to a fully discrete one with Dirac delta functions
%supported only in the grid nodes. Furthermore we prove convergence for a vanishing
%discretization parameter; see Section~\ref{sec:discretization}. To the best of the authors' knowledge this
%is the first work on a rigorous discretization concept and convergence analysis for an
%optimal design problem.

The paper is organized as follows: In Section~\ref{sec:fromparamtoopt} we present the
optimal design formulation under consideration. In Section~\ref{sec:analysis} we introduce
notation and state basic existence results for solutions to \eqref{def:designprop} as well
as first order optimality conditions. In Section~\ref{sec:algorithm} the generalized
conditional gradient algorithm for the
algorithmic solution of \eqref{def:designprop} is proposed and analyzed. Different
acceleration and sparsification strategies are presented and a (worst-case) sub-linear
convergence rate for the objective functional is proven.
%Section~\ref{sec:discretization}
%discusses the approximation of \eqref{def:designprop} by finite element methods.
The paper
is completed by a numerical example given in Section~\ref{sec:Numerics} to illustrate the
thory and show the practical efficiency of the
algorithms. In particular, we investigate the effect of the described
acceleration strategies.

\section{From Parameter estimation to optimal design}
\label{sec:fromparamtoopt}
In this section we derive the convex optimal design formulation
\eqref{def:designprop} and establish its connection to the non-convex problem
\eqref{def:pre_design}. We start by defining a least-squares estimator for
parameter estimation and the notion of the associated linearised
confidence domains.

\subsection{Parameter estimation}
\label{sec:paraest}

Within the scope of this work we consider the identification of a parameter $q$
entering a weak form $a(\cdot,\cdot)(\cdot)\colon~Q_{ad}\times \hat{Y}\times Y \rightarrow \R$,
which can be non-linear in its first two arguments but is linear in the last one. Here,
$Q_{ad} \subset \R^n$, $n \in \N$, denotes a set of admissible parameters, $Y$
denotes a suitable Hilbert space of functions, and $\hat{Y}=\hat{y}+Y$, where the function
$\hat{y}$ allows to include non-homogeneous (Dirichlet-type) boundary conditions in the model.
For every $q\in Q_{ad}$ a function $y=S[q]\in \hat{Y}$ is called
the state corresponding to $q$ if it is a solution to
\begin{equation}\label{stateequation}
\begin{aligned}
  a(q,y)(\varphi)=0\quad \forall \varphi \in Y.
\end{aligned}
\end{equation}
The operator $S \colon Q_{ad} \rightarrow \hat{Y}$ mapping a parameter $q$ to the associated state is called the parameter-to-state operator.
For instance, one might think of a Sobolev space defined on an open and
bounded Lipschitz domain $\Omega\subset \R^d$, $d\in \{1,2,3\}$ and as
$a(\cdot, \cdot)(\cdot)$ being the weak formulation of an elliptic partial
differential operator.
\begin{remark}
  Concretely, in the case of PDE~\eqref{Combustion}, we define
  \begin{align*}
    a(q,y)(\varphi)
    = \left(\nabla y, \nabla \varphi \right)_{L^2}
    + \left(\alpha \nabla y, \varphi \right)_{L^2}
    + \left(D \exp\left\{-E/(d-y)\right\}y(c-y), \varphi\right)_{L^2},
  \end{align*}
  and \(Y=\{\,\varphi \in H^1(\Omega)\;|\;\varphi\rvert_{\Gamma_{\mathrm{in}}}=0\,\}\).
  Here, the parameter vector is given by \(q = (D,E) \in \R^2\).
\end{remark}
{\color{black}
We define a closed set \(\Om \subset \bar\Omega\), on which it is possible to carry out
pointwise observations of the state.
}
We make the following general regularity assumption. 
\begin{assumption} \label{ass:Existenceofstate}
For every $q\in Q_{ad}$ there exists a unique solution $y\in \hat{Y} \cap C(\Om)$ to~\eqref{stateequation}.
The parameter-to-state mapping $S$ with
\begin{align*}
 S\colon Q_{ad}\rightarrow C(\Om) \quad\text{with}\quad q \mapsto S[q] = y,
\end{align*}
is continuously differentiable in a neighborhood of $Q_{ad}$ in $\R^n$.
We denote the directional derivative of $S$ in the
direction of the $k$-th unit vector by \( \partial_k S[q] \in C(\Om)\) and by $\partial S[q]\in C(\Om, \R^n)$ the vector of partial derivatives.
\end{assumption}
We emphasize that, under suitable differentiability assumptions on the form $a(\cdot, \cdot)$ and
Assumption~\ref{ass:Existenceofstate}, the $k$-th partial derivative $\delta y_k=\partial_k S[q]\in Y \cap C(\Om)$,
$k=1,\ldots,n,$ is the unique solution of the sensitivity equation
\begin{equation}\label{sensitivity}
\begin{aligned}
a'_y(q,y)(\delta y_k, \varphi)
= - a'_{q_k}(q,y)(\varphi), \quad \forall \varphi \in Y,
\end{aligned}
\end{equation}
where $y = S[q]$ and $a'_y$ and $a'_{q_k}$ denote the partial derivatives of the form $a$ with respect
to the state and the $k$-th parameter; see, e.g., \cite{troeltzsch2010optimale}.

In the following the exact value of the parameter vector $q \in \R^n$ appearing
in~\eqref{stateequation} is denoted by $q^*$. While, for
the purposes of analysis we can assume this value to be known, it will be replaced with an
appropriate a~priori guess in practice.
To estimate the parameter $q$ we consider measurement data $y_d$ collected at a set of $m$ disjoint
sensor locations $\{x_{j}\}^{m}_{j=1}\subset \Om$.
To take measurement errors into account we assume that the data $y^{j}_d \approx
S[q^*](x_{j})$ is given by the response of the model to the exact parameter values, which are additively perturbed by independently normally distributed
noise; cf., e.g., \cite{bates2007nonlinear}.
\changed{
%Taking into account that multiple measurements can be performed at the same location,
Thus, we obtain that
\begin{align*}
y^{j}_d = S[q^*](x_{j}) + \epsilon_{j},
~\epsilon_{j} \sim \mathcal{N}(0,1/\lambda_j),~\operatorname{Cov}(\epsilon_{j}, \epsilon_{i})=0,
\end{align*}
for all $i,j=1,\ldots, m$ and $j \neq i$, where the diligence factor
$\lambda_{j}$ denotes the inverse of the variance of the
measurement at the
$j$-th location. We assume that \(\lambda_j\) can be chosen arbitrarily in
\(\R_+\setminus\set{0}\) in the following.
%In this case the measurement weights \(\lambda_j > 0\)
%should be interpreted as diligence factors giving information on how carefully the data
%should be collected at the corresponding measurement point.
\begin{remark} \label{rem:onweights}
The scalar~$\lambda_j > 0$ corresponds to the reciprocal of the error variance of the
measurement taken at~$x_j$. Thus, it is part of the noise model. Since the diligence
factors~$\lambda_j$ are also subject to optimization, this interpretation requires some
additional discussion. First, assume that all measurements are performed with a given
sensor with unit error variance. Furthermore, suppose that taking~$N\in \N$
repeated measurements at the same location is possible.
By averaging the obtained measurement data and using the linearity of the variance,
$N$ measurements can be interpreted as a single one
with the improved error variance of~$1/N$. In this light, we can
interpret~\eqref{def:pre_design} as a convex relaxation of a mixed-integer optimization
problem for the overall number of different sensor sites~$m$, the positions~$x_j$, and the
associated number~$\lambda_j\in \N$ of repeated measurements at this point. Another point
of view is to simply assume that performing a single measurement with a given error
variance~$1/\lambda_j$ for any $\lambda_j > 0$ is possible by manufacturing or buying a
suitable sensor with precisely this variance.
\end{remark}}

%For ease of notation the positions $x_{j}$ are assumed to be distinct in the following.
To emphasize that the data $y_d$ is a random variable,
conditional on the measurement errors, we will write $y_d(\eps)$ in the following and
define the least squares functional
\begin{align}\label{Leastsquaresfunctional}
  J(q,\eps) = \frac{1}{2} \summ{m}{j} \lambda_{j}(S[q](x_{j})-y^{j}_{d}(\eps))^2
\end{align}
as well as the possibly multi-valued least squares estimator
\begin{align} \label{Leastsquareestimator}
  \tilde{q}\colon \mathbb{R}^m \rightarrow \mathcal{P}(\R^n),\;
  \tilde{q}(\eps)= \argmin_{q\in Q_{ad}} J(q,\eps),
\end{align}
where $\mathcal{P}(\R^n)$ denotes the power set of $\R^n$.
Note that this is the usual maximum likelihood estimator (MLE) using the assumption
on the distribution of the measurement errors $\eps_{j}$.

\subsection{Optimal design}
Since the measurement errors are modelled as random variables, the uncertainty in the data
is also propagated to the estimator. This means that $\tilde{q}$ should be interpreted as
a random vector. To quantify the bias in the estimation and to assess the quality of
computed realizations of the estimator, one considers the non-linear confidence domain of
$\tilde{q}$ defined as
\begin{align}\label{confidencedomain}
%D(\tilde{q},\alpha)(\cdot)\colon \mathbb{R}^m \rightarrow \mathcal{P}(\mathbb{R}^n),\;
D(\tilde{q},\alpha)(\epsilon)
= \left\{ p \in Q_{ad}\;\big|\;J(p,\epsilon)-\min_{q\in Q_{ad}} J(q,\epsilon)\leq \gamma_n^2(\alpha)\right\},
\end{align}
where $\gamma_n^2(\alpha)$ denotes the $(1-\alpha)$-quantile of the
$\chi^2$-distribution with $n$ degrees of freedom; see, e.g., \cite{bock1987randwertproblemmethoden, beal1960}. We
emphasize that the confidence domain is a function of the measurement errors and therefore
a random variable whose realizations are subsets of the parameter space. In this context,
the confidence level $\alpha\in(0,1)$ gives the probability that a certain realization of
$D(\tilde{q}(\epsilon),\alpha)(\epsilon)$ contains the true parameter vector $q^*$.

Consequently, a good indicator for the performance of the estimator $\tilde{q}$ is given by the
size of its associated confidence domains. The smaller their size, the closer
realizations of $\tilde{q}$ will be to $q^*$ with high probability. Given a realization
$D(\bar{q},\alpha)(\bar{\epsilon})$ of the non-linear confidence domain, its size only
depends on the position and the number of the measurements. To obtain a more reliable
estimate for the parameter vector, the experiment, i.e.\ positions $x_{j}$ and the measurement weights $\lambda_{j}$ should
be chosen a priori in such a way that confidence domains of the resulting estimator are
small. However, for general models and parameter-to-state mappings $S$ the estimator
$\tilde{q}$ cannot be given in closed form. Therefore it is generally not
possible to provide an exact expression for $D(\tilde{q},\alpha)$.

To circumvent this problem we follow the approach  proposed in,
e.g., \cite{pronzato2003removing, fedorov2013optimal} and consider a linearisation of the
original model around an a~priori guess $\hat{q}$ of $q^*$ which can stem from historical
data or previous experiments.
In the following, $\epsilon\in \R^m$ denotes an arbitrary vector of measurement errors,
and $x \in \R^{d\times m}$, \(x=(x_1,\dots, x_m)\), with \(x_j \in \R^d\), $j=1,\dots,m,$ stands for the
measurement locations. For abbreviation we write \(S[\hat{q}](x)\in \R^m\) for
the vector of observations with
\(S[\hat{q}](x)_j = S[\hat{q}](x_j)\), \(j=1,\ldots,m\).
Moreover the matrices $X \in \R^{m \times n}$ and $\Sigma^{-1} \in \R^{m \times m}$ are defined as
\begin{align*}
 X_{jk} = \partial_k S[\hat{q}](x_{j}), \quad \Sigma^{-1}_{ij} = \delta_{ij} \lambda_i, \quad i,j=1,\ldots,m,~k=1,\ldots,n,
\end{align*}
and are assumed to have full rank.
We arrive at the linearised least-squares functional
\begin{align*}
 J_{\mathrm{lin}}(q,\epsilon)
= \frac{1}{2} \summ{m}{j} \lambda_{j}(S[\hat{q}](x_{j})
  + \partial S[\hat{q}](x_{j})^\top (q-\hat{q}) - y^{j}_{d}(\epsilon))^2,
\end{align*}
which can be equivalently written as
\begin{align*}
 J_{\mathrm{lin}}(q,\epsilon)
= \frac{1}{2} \|X(q-\hat{q}) + S[\hat{q}](x)-y_d (\epsilon)\|^2_{\Sigma^{-1}},
\end{align*}
where $\|v\|_{\Sigma^{-1}} = v^\top \Sigma^{-1} v$ for $v \in \R^m$.
In contrast to the estimator $\tilde{q}$~\eqref{Leastsquareestimator}, the associated linearised estimator
\begin{align} \label{linearizedestimator}
  \tilde{q}_{\mathrm{lin}}\colon\mathbb{R}^m \to \mathbb{R}^n,
  \quad
  \tilde{q}_{\mathrm{lin}}(\epsilon) = \argmin_{q\in \R^n} J_{\mathrm{lin}}(q,\epsilon),
\end{align}
is single-valued and its realizations
can be calculated explicitly (see, e.g., \cite{tarantola2005inverse}), as
\begin{align} \label{explicitlinearized}
\tilde{q}_{\mathrm{lin}}(\epsilon)
  = \hat{q} + (X^\top\Sigma^{-1} X)^{-1}X^\top \Sigma^{-1}
  \left (y_d(\epsilon) - S[\hat{q}](x)\right).
\end{align}
Due to the assumptions on the noise $\epsilon$ the estimator $\tilde{q}_{\mathrm{lin}}$ is
a Gaussian random variable with
\(\tilde{q}_{\mathrm{lin}} \sim \mathcal{N}(\tilde{q}_{\mathrm{lin}}(0), (X^\top\Sigma^{-1} X)^{-1})\).
The associated realizations of its confidence domain (see, e.g.,
\cite{bock1987randwertproblemmethoden}) are thus given by
{\color{black}
\begin{align}\label{linearizedconfidence}
D(\tilde{q}_{\mathrm{lin}},\alpha)(\epsilon)
 = \left\{\tilde{q}_{\mathrm{lin}}(\varepsilon) + (X^\top \Sigma^{-1} X)^{-1}X^\top
  \Sigma^{-1/2} \xi,\;\big|\; \xi \in \R^m,\,
  \norm{\xi}_{\R^m} \leq \gamma_n(\alpha)\right\},
\end{align}
where \(\norm{\cdot}_{\R^m}\) denotes the Euclidean norm.
}
We point out that the linearised confidence domains are ellipsoids in the parameter space centered
around $\tilde{q}_{\mathrm{lin}}$. Their half axes are given by the eigenvectors of
the Fisher-information matrix $\fish = X^\top \Sigma^{-1} X$ with lengths proportional to the
associated eigenvalues. Their sizes depend only on the a~priori
guess $\hat{q}$ and the setup of the experiment, i.e.\ the position and total number of
measurements, but not on the concrete realization of the measurement noise. Consequently
we can improve the estimator by minimizing the linearised confidence domains as a function
of the measurement setup, which leads to \eqref{def:pre_design}. 

To establish the connection to the sparse optimal design approach we observe that
the entries of the Fisher-information matrix can be written alternatively as
\begin{equation} \label{def:Fisherinf}
(X^\top \Sigma^{-1} X)_{kl}
  = \summ{m}{j} \sens{k}(x_j)\sens{l}(x_j) \lambda_j
  = \int_{\Om} \sens{k} \sens{l}  \de \om
  = \fish(\omega)_{kl},
\end{equation}
with the design measure $\om = \summ{m}{j} \lambda_{j} \delta_{x_{j}}$.
Furthermore we note that for such a design measure there holds $\mnorm{\om} = \sum_{j=1}^m \lambda_j$.
%\cite{banks2014experimental, alexanderian2014optimal, alexanderian2014fast}.
Consequently, for some design criterion \(\Psi\) and prior knowledge \(\fish_0\),
the optimal design problem~\eqref{def:pre_design} can be equivalently expressed as
\begin{equation}\label{def:pre_design_measure}
%\begin{aligned}
% \min_{\om \in M^+(\Om)} &\quad\Psi(\fish(\om) + \fish_0) + \beta \mnorm{\om},\\
%   \text{subject to} &\quad\om \in \operatorname{cone}\{\,\delta_x\;|\;x \in \Om\,\},
%\end{aligned}
\min_{\om \in \operatorname{cone}\{\,\delta_x\;|\;x \in \Om\,\}} \quad\Psi(\fish(\om) + \fish_0) + \beta \mnorm{\om},
\end{equation}
where we minimize the objective functional over all non-negative linear combinations of
Dirac delta functions corresponding to points in the observational domain.
A~priori it is however unclear if this reformulation admits an optimal solution, since the
admissible set is not closed in the weak* topology on $M(\Om)$. For a rigorous analysis
one therefore has to pass to the closure
$\overline{\operatorname{cone}\{\,\delta_x\;|\;x \in \Om\,\}}^*=M^+(\Om)$;
{\color{black}
see, e.g., \cite[Problem~24.C]{Brezis:2010}.
}
As~\eqref{def:Fisherinf} suggests, the definition of $\fish$ can be extended to the set
of positive regular Borel measures \(M^+(\Om)\), resulting in the more general problem
formulation \eqref{def:designprop}.
\changed{\begin{remark}
In view of Remark~\ref{rem:onweights}, it may seem reasonable to incorporate upper bounds on
the coefficients~$\lambda_j$ into the formulation. This could be motivated either by
restricting the maximum number of repeated measurements at the same location (in case
the problem arises from a problem with identical sensors and integer \(\lambda_j\)
representing the number of measurements) or correspond to a restriction on the
variance provided by the best available sensor (e.g., due to manufacturing constraints).
Let us briefly discuss this issue.
Without restriction, we impose the restriction~$0 \leq \lambda_j \leq 1$, thus replacing the
cone of Dirac delta functions in~\eqref{def:pre_design_measure} by the set
\begin{align*}
M^+_{\text{const}}(\Om) = \left\{\,\om= \sum^m_{j=1} \lambda_j \delta_{x_j} \;\Big|\;
x_j \neq x_i\text{ for } i\neq j,\; 0 \leq \lambda_j \leq 1,\; m\in \N \,\right\} \subset M^+(\Om).
\end{align*}
We distinguish two cases:
First, let~$\Om$ be the closure of a bounded domain.
In this case,~$M^+_{\text{const}}(\Om)$ is not weak* closed.
Indeed, it is straightforward to argue
that~$\operatorname{cone}\{\,\delta_x\,|\,x \in \Om\,\} \subset \overline{M^+_{\text{const}}(\Om)}^*$ and
consequently~$\overline{M^+_{\text{const}}(\Om)}^*=M^+(\Om)$, i.e.\ we again arrive
at~\eqref{def:designprop}.
This stems back to the assumption that measurements at different
locations are pairwise uncorrelated. Thus, a measurement with arbitrarily small variance
at a point \(x\) can be approximated by a number of independent measurements with unit
variance at distinct points located in a small neighborhood of~\(x\).
Second, in the case that~$\Om$ is a collection of a finite number of isolated points,
replacing $M^+(\Om)$ by $M^+_{\text{const}}(\Om)$ is possible, since the latter is weak*
closed. However, for such $\Om$ the problem~\eqref{def:designprop} can be rewritten as a
simpler finite dimensional optimization problem 
%in terms of the coefficients associated to the finite number of points 
(cf.\ section~\ref{sec:acceleration}).
We do not specifically discuss this case in the following.
\end{remark}}

\section{Analysis of the optimal design problem} \label{sec:analysis}
%\TODO{KP: No text here? Remove next subsection or add text.}
%\subsection{Notation and assumptions}
In the following, we fix the general notation for the remainder of the paper. We consider an
observation set $\Om$ in which we allow measurements to be carried out.
It is assumed to be a closed subset of \(\bar\Omega\), which is
the closure of the bounded spatial domain $\Omega \subset \R^d$.
On $\Om$ we define the space of regular Borel measures \(M(\Om)\) as the topological dual
of $C(\Om)$, the space of continuous and bounded functions (see,
e.g., \cite{elstrodt2013mass}), with associated 
duality pairing $\pair{\cdot, \cdot}$. The norm on $M(\Om)$ is given by
\begin{align*}
\|\om\|_{M(\Om)}
%=\text{sup} \left\{ \sum_{n=1}^{\infty} |\om (B_n)|\bigg|~B_n \in
%\mathcal{B}(\Om)~\text{disjoint partition of}~\Om \right\}
=\sup_{y\in C(\Om),\; \maxnorm{y} \leq 1} \langle y,\om \rangle,
\end{align*} 
where $\maxnorm{\cdot}$ is the supremum norm on $C(\Om)$.
% and $\mathcal{B}(\Om)$ denotes the Borel sets on $\Om$.
By $M^+(\Om)$ we refer to the set of positive Borel measures on $\Om$ (see,
e.g., \cite[Def.~1.18]{Rudin}),
\begin{align*}
M^+(\Om)
= \left\{\,\om \in M(\Om) \;|\; \langle y, \om \rangle \geq 0,\; \forall y\in C(\Om)\; y \geq 0\,\right\},
%=\left\{\,\om \in M(\Om)\;|\;~\om(B)\in [0,\infty)~\forall B\in \mathcal{B}(\Om)\right\,\},
\end{align*}
with convex indicator function $I_{\om \geq 0}$.
For $\om \in M(\Om)$ the support is defined as usual by
\begin{align*}
\supp\om
 = \Om \backslash \left\{\,\bigcup B\in \mathcal{B}(\Om)\;|\;B \text{ open},\,\om(B)=0\,\right\},
\end{align*}
{\color{black}where \(\mathcal{B}(\Om)\) are the Borel subsets of \(\Om\).}
Note that the support is a closed set. \changed{In case the support is a finite
  set, we denote its cardinality (or counting measure) by \(\#\supp\om \in \N\).}

A sequence $\{\om_k\} \subset M(\Om)$ is called
convergent with respect to the weak*-topology with limit $\om\in M(\Om)$ if
$\pair{ y, \om_k} \to \pair{y, \om}$ for \(k \rightarrow\infty\) for all $y \in C(\Om)$ indicated by
$\om_k \rightharpoonup^* \om$.
Additionally we define the usual Lebesgue spaces of integrable and square integrable functions
$L^1(\Om)$ and $L^2(\Om)$, respectively, as well as the usual Sobolev space $H^1_0(\Om)$
 with associated (semi-)norm and inner product; see,
e.g., \cite{Adamssobolev}.
Furthermore we denote by $\operatorname{Sym}(n)$,
$\operatorname{NND}(n)$, and $\operatorname{PD}(n)$ the
sets of symmetric, symmetric non-negative definite {\color{black} 
(also, positive semi-definite)},
and symmetric positive definite matrices, respectively. On the set of symmetric matrices
we consider the inner product
\((A,B)_{\operatorname{Sym}(n)} = \Tr(AB^\top)\) for \(A,B \in \operatorname{Sym}(n)\),
where \(\Tr\) denotes the trace, and the L{\"o}wner partial order
\begin{align*}
0 \leq_L A \;\Leftrightarrow\; A \text{ is positive semidefinite}.
\end{align*} 
Last, for $\phi\colon M(\Om)\rightarrow \mathbb{R} \cup \{\,\infty\,\}$
and a convex set $M \subset M(\Om)$ we define the domain of $\phi$ over $M$ as
\begin{align*}
\dom_{M}\phi = \left\{\,\om \in M\;|\;\phi(\om)<\infty \,\right\},
\end{align*}
where the index is omitted when $M = M(\Om)$.

We consider design criteria of the form
\(\Psi(\cdot + \fish_0)\), where \(\fish_0 \in \operatorname{NND}(n)\)
(e.g. \(\fish_0 = 0\)) incorporates prior knowledge, as described in the introduction.  
Concerning the function $\Psi$ the following assumptions are made.
\begin{assumption}\label{ass:designcrit}
The function $\Psi\colon \operatorname{Sym}(n)\rightarrow \R \cup \{+\infty\}$ satisfies:
\begin{itemize}
\item [$\textbf{A1}$] There holds $\dom\Psi=\operatorname{PD}(n)$.
\item [$\textbf{A2}$] $\Psi$ is continuously differentiable for every $N\in \operatorname{PD}(n)$.
\item [$\textbf{A3}$] $\Psi$ is non-negative on $\operatorname{NND}(n)$.
\item [$\textbf{A4}$] $\Psi$ is lower semi-continuous and convex on $\operatorname{NND}$.
\item  [$\textbf{A5}$] $\Psi$ is monotone with respect to the L{\"o}wner ordering on $\operatorname{NND}(n)$, i.e. there holds
\begin{align*}
N_1\leq_L N_2 \Rightarrow \Psi(N_1)\geq \Psi(N_2)\quad \forall N_1,~N_2 \in \operatorname{NND}(n).
\end{align*}
\end{itemize} 
\end{assumption}
While Assumptions $(\textbf{A1})$ to $(\textbf{A4})$ are important for the existence of
optimal designs and the derivation of first order optimality conditions, Assumption
$(\textbf{A5})$ is related to the size of the linearised confidential
domains~\eqref{linearizedconfidence}.
Given two design measures $\om_1,\om_2 \in M^+(\Om)$ with
$\fish(\om_1)$, $\fish(\om_1)\in \operatorname{PD}(n)$ and
$\fish(\om_1)\leq_L\fish(\om_2)$ it holds
\begin{align*}
\mathcal{E}_2 = \{\, \delta q\in \R^n \;|\;\delta q^\top \fish(\om_2) \delta q \leq r\,\}
  \subset \mathcal{E}_1 = \{\, \delta q\in \R^n\;|\;\delta q^\top \fish(\om_1) \delta q \leq r \,\}
\end{align*}
for any $r>0$. {\color{black} Thus, $(\textbf{A5})$ ensures that $\Psi$ is a
scalar criterion for the size of the
linearised confidence ellipsoids \eqref{linearizedconfidence} that is compatible with
the inclusion of sets}. For a similar set of
conditions we refer to~\cite[p.~41]{ucinski2004optimal}. The given assumptions can be verified
for a large class of classical optimality criteria, among them the A and D criterion
\begin{align*}
\Psi_A(N) = 
  \begin{cases}
    \Tr(N^{-1}), & N \in \operatorname{PD}(n), \\
    \infty , & \text{else},
  \end{cases}
\quad
\Psi_D(N)=
   \begin{cases}
     \operatorname{det}(N^{-1}), & N \in \operatorname{PD}(n), \\
     \infty , & \text{else},
  \end{cases}
\end{align*}
corresponding to the combined length of the half axis and the volume of the confidence
ellipsoids. Additionally, one may also use weighted versions of the design criteria:
for instance $\Psi^w_A(N)=\Tr(WN^{-1}W)$ allows to put special emphasis
on particular parameters by virtue of the weight matrix
$W\in \operatorname{NND}(n)$. However,
we emphasize that the results presented in this paper cannot be applied to other
non-differentiable popular criteria such as the $E$ criterion defined by
\begin{align*}
\Psi_E(N)=\begin{cases} \max_{i}\left\{\lambda_i(N^{-1})\right\}, & N \in \operatorname{PD}(n), \\
         \infty , & \text{else}.\end{cases}
\end{align*}
describing the length of the longest half axis and the length of the longest side of the
smallest box containing the confidence ellipsoid. In this case, one can
for instance resort to smooth approximations of the design criteria.
%For the example of the $E$-criterion a suitable functional is given by
%\begin{align*}
%\Psi^{\delta}_E=\begin{cases}
% \delta \log\left(\Tr(\exp(N^{-1}/\delta))\right), & N \in \operatorname{PD}(n), \\
%\infty , & \text{else},\end{cases}
%\end{align*}
%for $0 < \delta \ll 1$; see, e.g., \cite{ucinski2004optimal}.
%\TODO{KP: Remove E-criterion to shorten?}

\subsection{Existence of optimal solutions to \eqref{def:designprop} and optimality conditions }
\label{sec:existence_and_optimality}
In this section we prove the existence of solutions as well as first order necessary and
sufficient optimality conditions for the optimal design problem
\eqref{def:designprop}. Additionally, results on the sparsity pattern of optimal designs
are derived. First, 
{\color{black}as canonical extension of~\eqref{def:Fisherinf}},
we introduce the linear and continuous Fisher-operator $\fish$ by
\begin{align*}
\fish\colon M(\Om)\rightarrow \operatorname{Sym}(n),
  \quad\text{with}\quad
  \fish(\om)_{k,l} = \pair{\sens{k} \sens{l},\, \om}
  \quad \forall k,l\in \{\,1, \ldots, n\,\}.
\end{align*}
It is readily verified that it is the Banach space adjoint of the operator
\[
\fish^*\colon \operatorname{Sym}(n) \to C(\Om),
\quad\text{with}\quad
\fish^*(A) = \varphi_A,
\]
where \(\varphi_A \in C(\Om)\) is the continuous function given for \(A \in \operatorname{Sym}(n)\) by
\begin{align}
\label{eq:adjoint_fisher}
\varphi_A(x)
= \Tr\left(\sensv(x) \sensv(x)^\top A\right)
= \sensv(x)^\top A\, \sensv(x)
\quad \forall x \in \Om.
\end{align}
Now, we formulate the reduced design problem~\eqref{def:designprop} as
\begin{align*}
 \min_{\om \in M^+(\Om)} F(\om) = \psi(\om) + \beta\mnorm{\om},
\end{align*}
where $\psi(\om) = \Psi(\fish(\om)+\mathcal{I}_0)$.
In the following proposition we collect some properties of the reduced functional.
\begin{proposition} \label{prop:properties}
Let Assumptions $(\textbf{A1})$--$(\textbf{A5})$ be fulfilled and let $\fish_0\in \operatorname{NND}(n)$ be given.
The operator $\fish$ and the functional $\psi$ satisfy:
\begin{itemize}
\item [1.] For every $\om \in M^+(\Om)$ there holds $\mathcal{I}(\om)\in \operatorname{NND}(n)$.
\item [2.] There holds $\dom_{M^+(\Om)}\psi
  = \left\{\,\om \in M^+(\Om) \;|\; \mathcal{I}(\om)+\mathcal{I}_0 \in \operatorname{PD}(n)\,\right\}$.
  %and
  %\begin{align} \label{form:conicdom}
  %  \om \in \dom_{M^+(\Om)} \psi \Rightarrow~ \lambda \om + v \in \dom_{M^+(\Om)}~\psi, 
  %\end{align}
  %for all $1 \geq \lambda > 0$, $v \in M^+(\Om)$.
\item [3.] $\psi$ is differentiable with derivative
  $\psi'(\om) = \fish^*\left(\Psi'(\fish(\om)+\fish_0)\right)\in C(\Om)$ for every
  $\om \in \dom_{M^+(\Om)} \psi$. The derivative can be identified with the
  continuous function
  \begin{equation}\label{eq:gradient_psi}
    \left[\psi'(\om)\right](x)
    = \partial S[q](x)^\top \Psi'(\fish(\om)+\fish_0)\, \partial S[q](x)
    \leq \changed{0}
    \quad\forall x\in\Om.
  \end{equation}
  %\textcolor{red}{For all~$\om \in \dom_{M^+(\Om)}\psi$ and~$x\in\Om$ there holds~$\psi'(\om)(x)\leq 0$.}
  Moreover the gradient $\psi'\colon\dom_{M^+(\Om)}\psi \to C(\Om)$ is
  weak*-to-strong continuous.
\item [4.] $\psi$ is non-negative on $\dom_{M^+(\Om)}\psi$.
\item [5.] $\psi$ is weak* lower semi-continuous and convex on $M^+(\Om)$.
\item [6.] $\psi$ is monotone in the sense that
\begin{align*}
\fish(\om_1)\leq_L \fish(\om_2) \Rightarrow \psi(\om_1)\geq \psi(\om_2)\quad \forall \om_1,~\om_2 \in M^+(\Om).
\end{align*}
\end{itemize}
%\TODO{KP: potentially drop 4.\ and 6.\ since really obvious?}
\end{proposition}
\begin{proof}
To prove the first claim we observe that there holds
\begin{align}
z^\top \fish(\om)z = \pair{(\partial S[\hat{q}]^\top z)^2, \om} \geq 0~\quad \forall z\in \R^n
\end{align}
for an arbitrary $\om \in M^+(\Om)$, thus $\fish(\om)\in \operatorname{NND}(n)$.
Statement $2.$ follows directly 
%\begin{align*}
%\om \in \dom_{M^+(\Om)} \psi \Leftrightarrow \psi(\om)< \infty \Leftrightarrow \fish(\om)+ \fish_0 \in \operatorname{PD}(n)
%\end{align*}
with $(\textbf{A1})$.
%Let $\om \in \dom_{M^+(\Om)} \psi$, $\lambda \in (0,1]$, $v\in M^+(\Om)$ be given. Observe that
%\begin{align*}
%0 < \lambda(z^\top\fish(\om)z + z^\top\fish_0z)
%  \leq z^\top\fish(\lambda \om)z  + z^\top \fish_0 z
%  \leq z^\top\fish(\lambda \om + v)z + z^\top \fish_0 z
%\end{align*}
%for all \(z \in \R^n\). Consequently $\fish(\lambda \om +v)+\fish_0 \in
%\operatorname{PD}(n)$ and $\lambda \om +v \in \dom_{M^+(\Om)} \psi$.
For $\om \in \dom_{M^+(\Om)}\psi$ the differentiability of $\psi$ follows from
assumption $(\textbf{A2})$ using the chain rule. We obtain the derivative
$\psi'(\om) \in M(\Om)^*$ characterized by
\begin{align*}
\langle \psi'(\om), \delta \om \rangle_{M^*,M} 
 = \Tr(\Psi'(\fish(\om)+\fish_0) \fish(\delta \om))
 = \langle \fish^*\left(\Psi'(\fish(\om)+\fish_0)\right), \delta \om \rangle_{M^*,M},
\end{align*} 
for every $\delta \om \in M(\Om)$, where $\langle \cdot, \cdot \rangle_{M^*,M}$ denotes the
duality pairing between $M(\Om)$ and its topological dual space.
%By straightforward computations we arrive at
%\begin{align*}
%\langle \psi'(\om), \delta \om \rangle_{M^*,M}
%= \Tr(\Psi'(\fish(\om)+\fish_0)\fish(\delta\om))\\
%= \summ{n}{i}\summ{n}{j} \Psi'(\fish(\om)+\fish_0)_{i,j} \pair{\partial_i S[\hat{q}] \partial_j S[\hat{q}],\delta\om}.
%\end{align*}
Using the adjoint expression for \(\fish\) given in~\eqref{eq:adjoint_fisher} we can
identify \(\psi'(\om)\) with the continuous function~\eqref{eq:gradient_psi}.
\changed{Due to the monotonicity of~$\Psi$ there holds 
\begin{multline*}
\psi'(\om)(x) = \pair{\psi'(\om),\delta_x}
= \lim_{\tau \to 0^+} \frac{1}{\tau}\left[\psi(\om + \tau \delta_x) - \psi(\om)\right]\\
= \lim_{\tau \to 0^+} \frac{1}{\tau}\left[\Psi(\fish(\om) + \fish_0 + \tau\,\partial S[\hat{q}](x) \partial S[\hat{q}](x)^\top)
- \Psi(\fish(\om)+\fish_0)\right] \leq 0\quad \forall x\in \Om.
\end{multline*}
For a sequence \(\{\,\om_k\,\}\subset \dom_{M^+(\Om)}\psi\) with
\(\om_k \rightharpoonup^* \om\) it follows from the definition of \(\fish\) that
\(\fish(\om_k) \to \fish(\om)\) for \(k\to\infty\). Using~\eqref{eq:gradient_psi}, it now
follows \(\psi'(\om_k) \to \psi'(\om)\) in \(C(\Om)\), which shows the continuity of \(\psi'\). 
}
Statements $4.$, $5.$, and $6.$ can be derived directly from Assumptions
$(\textbf{A2})$, $(\textbf{A4})$, and $(\textbf{A5})$ using
\changed{again the continuity of \(\fish\)}.
\end{proof}

\begin{proposition} \label{prop:existence}
Assume that $\dom_{M^+(\Om)}\psi\neq \emptyset$ and $\beta>0$. Then there exists at least one optimal solution $\optomb$  to \eqref{def:designprop}. Moreover the set of optimal solutions is bounded. If $\,\Psi$ is strictly convex on $\operatorname{PD}(n)$ then the optimal Fisher-information matrix $\fish(\optomb)$ is unique.
\end{proposition}
\begin{proof}
% The function $\psi$ is proper and non-negative by assumption. Consequently, there exists a sequence $\om_k \in \dom_{M^+(\Om)}\psi$ with 
% \begin{align*}
% F(\om_k)\rightarrow \inf_{\om\in M(\Om)}\;F(\om)=\hat{F}\in \R,
% \end{align*}
% and
% \begin{align*}
% \mnorm{\om_k} \leq F(\om_k)/ \beta \leq F(\tilde{\om})/\beta+1,
% \end{align*}
% for a fixed $\tilde{\om} \in\dom_{M^+(\Om)}\psi$.
% Therefore $\mnorm{\om_k}$ is bounded. By the Banach-Alaoglu Theorem there exists a weak*
% convergent subsequence of $\om_k$, denoted by the same symbol, with limit $\optomb$. Due
% to the weak* lower semi-continuity of $\psi$ and the norm we obtain
% \begin{align*}
% F(\optomb)\leq \liminf_{k \rightarrow \infty} F(\om_k)=\hat{F},
% \end{align*} 
% which proves the optimality of $\optomb$.
The proof follows standard arguments, using the direct method in variational calculus,
using the estimate \(\mnorm{\om} \leq F(\om)/ \beta\), the sequential version of the
Banach-Alaoglu theorem, and the facts that \(F\) is proper and weak* lower-semicontinuous.
The boundedness of the set of optimal solutions is another direct consequence.
Additionally, uniqueness of the optimal Fisher information matrix can be deduced from
strict convexity of \(\Psi\) by a direct contradiction argument.
% Assuming $\om_1,~\om_2 \in M^+(\Om)$ are optimal solutions with $\om_1\neq\om_2,~\fish(\om_1)\neq \fish(\om_2)$, every convex combination $\om_{\lambda}=(1-\lambda)\om_1 +\lambda \om_2$, $\lambda\in (0,1)$, is also optimal. Due to the strict convexity of $\Psi$ there holds
%\begin{align*}
%\psi(\om_{\lambda})<(1-\lambda)\psi(\om_1)+ \lambda \psi(\om_2)
%\end{align*}
%and thus
%\begin{align*}
%F(\om_1)=F(\om_{\lambda})=F((1-\lambda)\om_1 +\lambda \om_2)<(1-\lambda)F(\om_1) +\lambda F( \om_2)=F(\om_1),
%\end{align*}
%yielding a contradiction. Thus the optimal Fisher-information matrix is unique.
\end{proof}
\changed{
\begin{remark}
The A and D criterion introduced above are strictly convex.
\end{remark}
}
Next we give conditions for the domain of $\psi$ to be non-empty. 

\begin{proposition} \label{prop:existenceexample}
Assume that $\beta>0$ and
\begin{align*}
\R^n
= \operatorname{span} \left(\operatorname{Ran}\fish_0 \cup \left\{\,\partial S[\hat{q}](x)\;\big|\;x \in \Om\,\right\}\right ).
\end{align*}
 Then there exists at least one optimal solution of \eqref{def:designprop}. Furthermore, every design measure $\om \in \dom_{M^+(\Om)}\psi$ consists of at least $n_0=n- \operatorname{rank} \fish_0$ support points.
\end{proposition}
\begin{proof}
According to Proposition \ref{prop:existence} we have to show that there exists an admissible design measure. By assumption we can choose a set of $n- \operatorname{rank} \fish_0$ distinct  points $x_j\in \Om$ such that 
\begin{align*}
\R^n=\operatorname{span}\left (\operatorname{Ran}\fish_0 \cup \left\{\,\partial S[\hat{q}](x_j)\;\big|\;j=1,\dots, n- \operatorname{rank} \fish_0 \right\}\right ).
\end{align*}
Consequently, setting $\om=\sum^{n_0} _{j=1} \delta_{x_j}\in M^+(\Om)$, we obtain 
\begin{align*}
\fish(\om)+\fish_0=\sum_{j=1}^{n_0}  \partial S[\hat{q}](x_j)\partial S[\hat{q}](x_j)^\top+ \fish_0\in \operatorname{PD}(n),
\end{align*}
by straightforward arguments.
 For the last statement we simply observe that for a measure $\om$ with less than $n_0 = n- \operatorname{rank} \fish_0$ support points, the associated information matrix $\fish(\om)+\fish_0$ has a non-trivial kernel. 
\end{proof}
By standard results from convex analysis the following necessary and sufficient optimality conditions can be obtained.
\begin{proposition} \label{prop:firstordernec}
Let $\optomb \in \dom_{M^+(\Om)}\psi$ be given. Then $\optomb$ is an optimal solution to
\eqref{def:designprop} if and only if holds:
\begin{align}\label{subdifferentialinequ}
\langle -\psi'(\optomb), \om-\optomb \rangle + \beta \mnorm{\optomb} \leq \beta \mnorm{\om} \quad \forall \om \in M^+(\Om).
\end{align}
\end{proposition}
\begin{proof}
Since $F$ is convex, a given $\optomb$ is optimal if and only if
\begin{align*}
0 \in \partial \left( F(\optomb)+ I_{\om\geq 0}(\optomb) \right),
\end{align*}
where the expression on the right denotes the subdifferential of $F+ I_{\om\geq 0}$ at $\optomb$ in $M(\Om)^*$. \changed{Due to the convexity of~$\beta \mnorm{\cdot}+ I_{\om \geq 0}$ and since $\psi$ is convex and differentiable at~$\optomb$} there holds
\begin{align*}
0 \in \partial F(\optomb)
= \psi'(\optomb) + \partial \left(\beta \mnorm{\cdot}+ I_{\om\geq 0}(\cdot)\right)(\optomb),
\end{align*}
which is equivalent to~\eqref{subdifferentialinequ}.
%\TODO{DW\&KP: omit proof?}
\end{proof}
Since the norm as well as the indicator function are positively homogeneous, the
subdifferential of~$\beta \mnorm{\cdot}+I_{\om \geq 0}$ can be characterized further. This yields an equivalent characterization
of optimality relating the support points of an optimal design to the set of minimizers
of the gradient of $\psi$ in the optimum.
\begin{lemma}
\label{lem:optcond}
Let $\optomb$ be an optimal solution to \eqref{def:designprop}.
Condition \eqref{subdifferentialinequ} is equivalent to
\begin{align}\label{suppcongeneral}
-\psi'(\optomb) \leq \beta, \quad
\operatorname{supp} \optomb \subset \left\{ x\in \Om\;|\;\psi'(\optomb)(x)=-\beta \right\}.
\end{align}
\end{lemma}
\begin{proof}
\changed{
We only give a brief sketch of the proof. Set~$g=\beta \mnorm{\cdot}+I_{\om \geq
  0}$. Clearly, there holds~$g(\lambda \om)=\lambda g(\om)$ for all~$\om \in M(\Om)$
and~$\lambda \geq 0$. As a consequence, we obtain from
\(-\psi'(\optomb)\in \partial g(\optomb)\) that
\begin{align*}
-\psi'(\optomb)\in \partial g(0) \quad\text{and}\quad \beta \mnorm{\optomb}= \langle-\psi'(\optomb),\optomb \rangle.
\end{align*}
Due to the non-negativity of~$-\psi(\optomb)$, the first condition can be equivalently expressed as
\begin{align*}
\pair{-\psi'(\optomb),\om} \leq \beta \mnorm{\om}\; \forall \om \in M^+(\Om)
\;\Leftrightarrow\;
-\psi'(\optomb)(x) \leq \beta\;\forall x\in \Om.
\end{align*}
The condition on the support of $\optomb$ in~\eqref{suppcongeneral} now follows with similar arguments as in~\cite[Proposition~3]{bredies2013inverse}.}
\end{proof}
\begin{remark}
For \eqref{def:constraineddesign} a similar optimality condition can be derived. A measure $\optom^K\in \dom_{M^+(\Om)}\psi$ is an optimal solution of \eqref{def:constraineddesign} if and only if
\begin{align} \label{eq:equ1}
  \supp\optom^K \subset 
  \left\{x\in \Om \;\Big|\; \psi'(\optom^K)(x) = \argmin_{x\in \Om} \psi'(\optom^K)(x) \right\},
\end{align}
where the condition on the support of $\optom^K$ is equivalent to 
\begin{align} \label{eq:equ2}
-\langle \psi'(\optom^K),\optom^K \rangle + \min_{x\in \Om} \psi'(\optom^K)(x) \mnorm{\optom^K}=0.
\end{align}
\changed{As for the norm regularized case, we give a short proof of this result. We only derive~\eqref{eq:equ2}. The equivalence to~\eqref{eq:equ1} then again follows as in \cite[Proposition~3]{bredies2013inverse}.
The measure~$\bar{\om}^K \in \dom_{M^+(\Om)}\psi$ is optimal for~\eqref{def:constraineddesign} if and only if~$\mnorm{\bar{\om}^K}\leq K$ and
\begin{align*}
\langle \psi'(\bar{\om}^K), \bar{\om}^K \rangle \leq \langle \psi'(\bar{\om}^K), \om \rangle \quad \forall \om \in M^+(\Om),~\mnorm{\om} \leq K.
\end{align*}
Clearly, since~$\psi'(\bar{\om}^K)$ is non-positive, this holds if and only if
\begin{align*}
\pair{\psi'(\bar{\om}^K), \bar{\om}^K}
= \inf_{\substack{\om \in M^+(\Om) \\ \mnorm{\om}\leq K}} \pair{\psi'(\bar{\om}^K), \om}
= \min_{x \in \Om} \psi'(\bar{\om}^K)(x) \mnorm{\bar{\om}^K}.
\end{align*}
This finishes the proof. Moreover, if~$ \min_{x\in \Om} \psi'(\optom^K)(x)\neq 0$, we have \(\mnorm{\bar{\om}^K} = K\) and optimality of~$\bar{\om}^K$ is equivalent to 
\begin{align} \label{eq:equ3}
0=\langle \psi'(\bar{\om}^K),\bar{\om}^K \rangle - \min_{x\in\Om} \psi'(\bar{\om}^K)(x) K \leq \langle \psi'(\om),\om \rangle - \min_{x\in\Om} \psi'(\om)(x) K 
\end{align}
for all~$\om \in \dom_{M^+(\Om)}\psi$ with~$\mnorm{\om}\leq K$.
For~$K=1$ the three statements in~\eqref{eq:equ1},~\eqref{eq:equ2} and~\eqref{eq:equ3} yield the well-known Kiefer-Wolfowitz equivalence theorem; see
\cite{kiefer1959optimum,kiefer1974general} and \cite[Theorem~3.2]{ucinski2004optimal}.}
\end{remark}
Since the Fisher-operator $\fish$ is a finite rank operator, uniqueness of the optimal
solution is usually not guaranteed.
However, the existence of at least one solution with the practically desired sparsity
structure is addressed in the following theorem.
\begin{theorem} \label{thm:existence of finite}
Let $\om \in M^+(\Om)$ be given. Then there exists $\tilde{\om} \in M^+(\Om)$ with
\begin{align*}
\fish(\om)=\fish(\tilde{\om}),\quad \mnorm{\tilde{\om}}\leq \mnorm{\om},
\quad \# \supp \tilde{\om} \leq n(n+1)/2.
\end{align*}
Additionally, if there exists an optimal solution to \eqref{def:designprop}, then there exists an optimal solution $\optomb$ with $\# \supp \optomb \leq n(n+1)/2$.
\end{theorem} 
%\begin{proof}
%%The proof is based on a refined version of Caratheodory's theorem, \cite{Steinitz1914},
%%but will be omitted here. For details we refer to the proof of
%%\cite[Proposition~B.5]{walter2017Helmholtz}.
%\end{proof}
In order to prove this statement, we first provide some auxiliary results. For $m\in \N$ define the cone of measures supported on at most~$m$ points as
\begin{align*}
M_m^+(\Om) =
%  \left \{\,\sum^m_{j=1} \lambda_j \delta_{x_j}\;\Big|\;\,\lambda_j \in \R_+,\; x_j \in\Om,\; j=1,\dots,m\right\}
\changed{\left\{ \omega \in M^+(\Om) \;\big|\; \#\supp \omega \leq m \right\}}.
\end{align*}
\begin{lemma} \label{lem:weakclos}
Let~$m\in \N$ be given. The set~$M^+_m(\Om)$ is weak* closed.
\end{lemma}
\begin{proof}
Let a weak* convergent subsequence~$\{\om_k\}\subset M^+_m(\Om)$ with~$\om_k
\rightharpoonup^* \bar{\om} \in M^+(\Om)$ be given. For each~$k\in \N$ there
exist~$\lambda^k_j \in \R_+$ and~$x^k_j \in \Om$, \(j =1,\ldots,m\) with
\begin{align*}
\om_k= \sum^m_{j=1} \lambda^k_j \delta_{x^k_j} \quad \text{and} \quad \mnorm{\om_k}= \sum_{j=1} \lambda^k_j \leq c
\end{align*}
with~$c>0$ independent of~$k\in \N$. Introducing~$\lambda^k=(\lambda^k_1, \dots, \lambda^k_m)^\top \in \R^m_+$ and~$\mathbf{x}^k=(x_1^k, \dots,x_m^k)^\top \in \Om^m$ there exist a convergent subsequence of $\{(\mathbf{x}^k, \lambda^k)\}$, denoted by the same symbol, as well as elements with~$(\mathbf{x}^k,\lambda^k)\rightarrow (\mathbf{x},\lambda)$. Define the measure
\begin{align*}
\om= \sum^m_{j=1} \lambda_j \delta_{x_j} \in M^+_m(\Om) \quad \text{where}\quad \mathbf{x}=(x_1,\dots,x_m)^\top,~\lambda=(\lambda_1,\dots, \lambda_m)^\top.
\end{align*}
Given~$\varphi \in C(\Om)$ there holds~$\varphi(x^k_j)\rightarrow \varphi(x_j)$ as well
as~$\lambda_j^k \rightarrow \lambda_j$,~$j=1,\dots,m$. Since \(\varphi\) is arbitrary, we 
conclude~$\pair{\varphi, \om^k} \rightarrow \pair{\varphi, \om}$.
Thus there holds~$\om = \bar{\om}$ since the weak* limit is unique. This proves the statement.
\end{proof}
\changed{We require the following lemma, which is a variant of the
  Carath{\'e}odory lemma.}
\begin{lemma} \label{lem:repres}
Let~$\om \in M^+_m(\Om)$ for some~$m \in \N$ be given. Furthermore assume that the set~$\{\,\fish(x)\;|\;x \in \supp \om\,\}$ is linearly dependent. Then there exists $\widetilde{\om} \in M^+(\Om)$ with
\begin{align} \label{eq:sparseonce}
\fish(\om)=\fish(\widetilde{\om}),\quad \mnorm{\widetilde{\om}}\leq \mnorm{\om},
\quad \# \supp \widetilde{\om} \leq \# \supp \om -1 .
\end{align}
In particular, given any measure~$\om \in M^+_m(\Om)$,~$m\in \N$, there is~$\widetilde{\om}\in M^+(\Om)$ fulfilling~$\fish(\om)=\fish(\widetilde{\om}),~\mnorm{\widetilde{\om}}\leq \mnorm{\om}$ and~$\#\supp \om \leq n(n+1)/2$.
\end{lemma}
\begin{proof}
Let~$\om=\sum^m_{j=1} \lambda_j\delta_{x_j}$ be given.
\changed{Without restriction, assume that~$\lambda_j>0$ for~$j=1,\dots,m$.
Define~$\fish_j=\fish(\delta_{x_j}) \in \operatorname{Sym}(n)$.
By assumption, the set~$\{\fish_j\}_{j=1}^m$ is linearly dependent.
Thus, we find a nontrivial solution~$\gamma$ of the system of equations
\(\sum_{j=1,\ldots,m} \gamma_j \fish_j = 0\). By possibly taking the negative of
\(\gamma\) we can ensure that \(\sum_{j=1,\ldots,m} \gamma_j \geq 0\).}
Set
\[
\mu = \max_{n=1,\ldots,m} \frac{\gamma_j}{{\lambda_j}} > 0.
\]
We define
\[
\widetilde{\om} = \om - \frac{1}{\mu} \sum_{j=1}^{m} \gamma_j \delta_{x_j}
= \sum_{j=1}^{m} \left(1-\frac{\gamma_j}{\mu {\lambda_j}}\right) \lambda_j\delta_{x_j}.
\]
The coefficients of the new measure
\(\widetilde{\om}=\sum^m_{j=1}\widetilde{\lambda}_j \delta_{x_j}\) are given as \(\widetilde{\lambda}_j =
[1-\gamma_j/(\mu {\lambda_j})]\lambda_j \in \R_+\)
since $\gamma_j/\mu \leq {\lambda_j}$. Moreover, we have $\fish(\om)=\fish(\widetilde{\om})$ as well as
\begin{align*}
\mnorm{\widetilde{\om}}
 = \mnorm{\om} - \sum_{j=1,\ldots,m} \gamma_j/\mu
 \leq \mnorm{\om}.
\end{align*}
The proof of~\eqref{eq:sparseonce} is finished with the observation that
\[
  \widetilde{\lambda}_{\hat{\jmath}} = 0 \quad\text{for }
  \hat{\jmath} \in \argmax_{j=1,\ldots,m} \frac{\gamma_j}{{\lambda_j}}.
\]
For the last statement, we recall that for any~$\om \in M(\Om)$ it holds $\fish(\om)\in \operatorname{Sym}(n) \simeq \R^{n(n+1)/2}$. 
Thus, if~$\#\supp \om >n (n+1)/2$, the set~$\{\,\fish(x)\;|\;x \in \supp \om\,\}$ is linearly dependent.
The result can now be proven by induction over the number of support points.
\end{proof}
\begin{proof}[Proof of Theorem~\ref{thm:existence of finite}]
Let~$\om \in M^+(\Om)$ be given. There exist sequences~$\{\om_k\}\subset M^+(\Om)$
with~$\#\supp \om_k < \infty$, $\om_k \rightharpoonup^* \om$, and~$\mnorm{\om_k} \leq
\mnorm{\om}$. Invoking Lemma~\ref{lem:repres} now yields the existence of a
sequence~$\{\widetilde{\om}^k\}\subset M^+_{\widetilde{m}}(\Om)$,
where~$\widetilde{m}=n(n+1)/2$, with $\fish(\widetilde{\om}^k)=\fish(\om^k)$ and~$\mnorm{\widetilde{\om}^k}\leq \mnorm{\om^k}$ for all~$k\in\N$. Consequently~$\widetilde{\om}^k$ admits a weak* convergent subsequence, denoted by the same symbol, with limit~$\widetilde{\omega}$. Moreover,
\begin{align*}
\mnorm{\widetilde{\om}_k}= \langle 1, \widetilde{\om}^k \rangle \rightarrow \langle 1, \widetilde{\om} \rangle= \mnorm{\widetilde{\om}}.
\end{align*}
Similarly there holds~$\lim_{k \rightarrow \infty} \mnorm{\om_k}=\mnorm{\om}$.
Combining these observations we obtain
\begin{align*}
\fish(\om)= \lim_{k \rightarrow \infty} \fish(\om^k)= \lim_{k \rightarrow \infty} \fish(\widetilde{\om}_k)= \fish(\widetilde{\om}), \quad \mnorm{\widetilde{\om}} \leq  \mnorm{\om}.
\end{align*}
Since~$\widetilde{\omega}\in M^+_{\widetilde{m}}(\Om)$ (see Lemma~\ref{lem:weakclos}) and~$F(\widetilde{\om})\leq F(\om)$, this finishes the proof.
\end{proof}
In the last part of this section we will further discuss structural properties of
solutions to \eqref{def:designprop}, mainly focusing on their connection to
\eqref{def:constraineddesign} and their behaviour for $\beta \rightarrow \infty$.
\changed{
In the following, we call a criterion \(\Psi\) strictly monotone with respect to the L{\"o}wner ordering, if
\begin{align*}
 N_2 - N_1 \in \operatorname{PD}(n) \;\Rightarrow\; \Psi(N_1) > \Psi(N_2) \quad\forall N_1,N_2 \in \operatorname{PD}(n).
\end{align*}
In particular, this is true for the A and D criterion introduced above.
}
\begin{proposition} \label{prop:equivalencedesigns}
The problems \eqref{def:constraineddesign} and \eqref{def:designprop} are equivalent in
the following sense:
\changed{Given $K>0$, there exists a $\beta(K) \geq 0$ (not necessarily unique),
  such that any optimal solution to~\eqref{def:constraineddesign} is
  an optimal solution to~\eqref{def:designprop} and vice-versa.}
 
Furthermore, assuming that $\Psi$ is strictly monotone with respect to the L{\"o}wner ordering,
we additionally obtain the following:
\begin{enumerate}
\item We have $\mnorm{\bar{\om}^K}=K$ for each optimal solution $\bar{\om}^K$ to \eqref{def:constraineddesign}.
\item There exists a uniquely defined function 
\begin{align*}
\beta\colon \R_+ \setminus \{0\} \to \R_+ \setminus \{0\}, 
\quad K \mapsto \beta(K),
\end{align*} 
such that each optimal solution $\bar{\om}^K$ to \eqref{def:constraineddesign} is a minimizer of $(P_{\beta(K)})$.
\end{enumerate}
\end{proposition}
\begin{proof}
\changed{
We will derive the first result as a consequence of Lagrange duality. Define the
Lagrangian as
\begin{align*}
L \colon M(\Om) \times \R \rightarrow \R \cup \{\,+\infty\,\},
  \quad L(\om, \beta) = \psi(\om) + I_{\om\geq 0}(\om) + \beta \left(\mnorm{\om} - K\right).
\end{align*}
Since a Slater condition holds for~\eqref{def:constraineddesign}  -- there exists a
\(\om \geq 0\) with \(\psi(\om) < +\infty\) and \(\mnorm{\om} < K\) -- the following
strong duality holds (see \cite[Theorem~2.165]{bonnans2000perturbation}):
\[
\min \eqref{def:constraineddesign}
 = \min_{u \in M(\Om)} \sup_{\beta \geq 0} L(u,\beta)
 = \max_{\beta \geq 0} \inf_{u \in M(\Om)} L(u,\beta)
 = \max_{\beta \geq 0} \left[ \inf \eqref{def:designprop} - \beta K \right]
\]
Therefore, by Lagrange duality (see, e.g., \cite[Section~2.5.2]{bonnans2000perturbation}),
the set of saddle points of the Lagrangian is given precisely by \((\optomb^K,\beta(K))\),
where \(\optomb^k\) solves~\eqref{def:constraineddesign} and \(\beta(K)\) solves the dual
problem given above. Clearly, saddle points of \(L\) give solutions
of~\eqref{def:designprop} with \(\beta = \beta(K)\).
%A given measure $\bar{\om}^K\in M^+(\Om)$ is optimal for \eqref{def:constraineddesign} if
%and only if there exists a Lagrange multiplier $\beta \geq 0$ with
%\begin{align}\label{lagrangeminimizer}
%\bar{\om}^K \in \argmin_{\om \in M(\Om)} L(\om,\beta),
%\quad \beta(\mnorm{\bar{\om}^K} - K) = 0.
%\end{align}
%The set of Lagrange multipliers is independent of the choice of the optimizer
%$\bar{\om}^K$, i.e.\ given two arbitrary optimal solutions
%$\bar{\om}^K_1,\bar{\om}^K_2 \in M^+(\Om)$ to \eqref{def:constraineddesign} and
%$\beta \geq 0$ such that the pair $(\bar{\om}^K_1, \beta )$ fulfills
%\eqref{lagrangeminimizer}, then so does $(\bar{\om}^K_2, \beta )$.
}
Together, this proves the first statement.

Assume that $\Psi$ is strictly monotone. Let $\optom^K$ be an arbitrary optimal solution
to \eqref{def:constraineddesign} with $\mnorm{\optom^K}<K$. Using the strict monotonicity
of $\Psi$ we deduce that $\optom^K \neq 0$.
Defining $\widetilde{\om}=(K/\mnorm{\optom^K})\optom^K$ there holds
$\psi(\widetilde{\om})< \psi(\optom^K)$ since $(K/ \mnorm{\optom^K})>1$.
This gives a contradiction and $\mnorm{\optom^K}=K$.
 It remains to show that for a given $K$ the associated Lagrange multiplier denoted by $\beta(K)$ is positive, unique, and $\beta(K_1)\leq\beta(K_2)$ if $K_2 >K_1$. To prove the positivity, assume that $\beta(K)=0$. Then we obtain
\begin{align*}
L(\bar{\om}^K, \beta(K))=\inf_{\om \in M^+(\Om)}L(\om, \beta(K))= \inf_{\om \in M^+(\Om)} \psi(\om).
\end{align*}
Given $\om \in \dom_{M^+(\Om)}\psi$, we have $\psi(2 \om)< \psi(\om)$ and consequently the infimum in the equality above is not attained, yielding a contradiction. 
Assume that $\beta(K)$ is not unique, i.e. there exist $\beta_1(K), \beta_2(K) >0$ such that each optimal solution $\bar{\om}^K$ of \eqref{def:constraineddesign} is also a minimizer of $L(\cdot,\beta_1(K))$ and $L(\cdot,\beta_2(K))$ over $M^+(\Om)$. First we note again that $0\in M^+(\Om)$ is not an optimal solution to \eqref{def:constraineddesign} due to the strict monotonicity of $\Psi$. Additionally it holds $\mnorm{\bar{\om}^K}=K$. Without loss of generality assume that $\beta_1(K)<\beta_2(K)$. From the necessary optimality conditions for $(P_{\beta_1(K)})$ and $(P_{\beta_2(K)})$, see \eqref{suppcongeneral}, we then obtain
\begin{align*}
-\psi'(\bar{\om}^K) \leq \beta_1(K) < \beta_2(K),
\quad
\operatorname{supp} \optomb \subset \left\{\, x\in \Om \;|\; -\psi'(\optomb)(x)=\beta_2(K) \,\right\},
\end{align*}
implying $\bar{\omega}^K=0$ which gives a contradiction.
%\TODO{KP: Proof to appendix?}
\end{proof}
%i.e. normalizing an optimizer $\optomb$ to the latter one yields an optimal solution to the first one, in the case of no a priori information.
%The fact that both formulations of the optimal design problem are inherently equivalent was briefly mentioned in \cite{haber2008numerical} for the case of $\Om$ consisting of finitely many points.

Many commonly used optimality criteria $\Psi$ are positively homogeneous in the sense
that there exists a convex, strictly decreasing, and positive function $\gamma$ fulfilling 
\begin{align} \label{positivehomogen}
  \Psi(r N) = \gamma(r)\,\Psi(N) \quad \forall r>0,\; N\in \operatorname{PD}(n);
\end{align}
cf.\ also \cite[p.~26]{fedorov2012model}.
For example, both the A and the D-criterion fulfill this homogeneity with $\gamma_A$ and $\gamma_D$ given by
\begin{align*}
\gamma_A(r) = r^{-1}, \quad \gamma_D(r) = r^{-n}.
\end{align*}
The following lemma illustrates the findings of the previous result,
provided that $\fish_0=0$. It turns out that solutions to \eqref{def:constraineddesign}
can be readily obtained by scaling optimal solutions to \eqref{def:designprop}.
\begin{proposition} \label{prop:equivalence}
Assume that $\fish_0 = 0$ and $\,\Psi$ is positive homogeneous in the sense of~\eqref{positivehomogen}.
%with $\gamma \in C^1((0,\infty))$ fulfilling $\lim_{r\rightarrow 0}\gamma(r)= \infty$
%and $\lim_{r\rightarrow \infty}\gamma(r)= 0$.
%Furthermore let $\gamma'$ be strictly decreasing.
Let $\optomb$ be a solution to \eqref{def:designprop} for some fixed $\beta>0$. Then
\begin{align}\label{scalingfactor}
  K \, \optomb / \norm{\optomb}_{M(\Omega)} \quad\text{solves}\quad \eqref{def:constraineddesign}.
\end{align}
%Conversely, there is a function \(K\colon \R_+ \to \R_+\), such that every measure
%\(\om_\beta = K(\beta) \optom^1\) for some optimal solution $\optom^1$ of~$(P^1)$
%is optimal for~\eqref{def:designprop}.
\end{proposition} 
\begin{proof}
First we note that under the stated assumptions every optimal solution $\optom^K$ to
\eqref{def:constraineddesign} fulfills $\mnorm{\optom^K}=K$.
Clearly, we have
\[
\min \eqref{def:constraineddesign}
= \min_{\substack{\om \in M^+(\Om), \\ \|\om\|=K}} \psi(\om)
= \min_{\substack{\om' \in M^+(\Om), \\ \|\om'\|=1}} \psi(K \om')
= \gamma(K) \min (P^1),
\]
by using the positive homogeneity of \(\Psi\).
Thus, the solutions of \eqref{def:constraineddesign} are given by \(K\om^1\), where
\(\om^1\) are solutions of $(P^1)$. Now, using the fact that
\begin{align*}
\min \eqref{def:designprop}
= \min_{K \geq 0} \left\lbrack \min_{\om' \in M^+(\Om),\; \|\om'\|=1} \psi(K \om') + \beta K \right\rbrack
= \min_{K \geq 0} \left\lbrack \gamma(K) \min(P^1) + \beta K \right\rbrack
\end{align*}
the solutions \(\optomb\) of \eqref{def:designprop} can be computed as
\(\optomb = K \om^1\), where \(K\) minimizes the above expression and \(\om^1 \in \argmin (P^1)\).
Together, this directly implies~\eqref{scalingfactor}. 
% First we note that under the stated assumptions every optimal solution $\optom^K$ to
% \eqref{def:constraineddesign} fulfills $\mnorm{\optom^K}=K$.
% We observe that \eqref{def:designprop} is equivalent to
% \begin{align*}
% \min_{K \geq 0} \Bigg\lbrack
%   \min_{\om \in M^+(\Om),\; \|\om\|=1} \psi(K \om)+ \beta K
%   \Bigg \rbrack.
% \end{align*}
% Using the positive homogeneity of $\Psi$ we obtain
% \begin{align*}
% \argmin_{\substack{\om \in M^+(\Om), \\ \|\om\|=1}} \;\psi(K \om)+ \beta K
%   = \argmin_{ \substack{\om \in M^+(\Om), \\ \|\om\|=1}} \gamma(K) \psi(\om)
%   = \left\{\, \om^1 \;\big|\;\om^1~\text{solves~($P^1$)} \,\right\}.
% \end{align*}
% Consequently, given any arbitrary optimal solution $\optomb$ to \eqref{def:designprop}
% there exists an optimal solution $\optom^1$ of ($P^1$)
% with $\optomb = \bar{K}(\beta)\optom^1$.
% Moreover, \(\bar{K}(\beta)\) can be computed as
% \begin{align}\label{eq:rproblem}
% K(\beta) \in \argmin_{K \geq 0}\ \gamma(K) \psi(\optom^1) + \beta K.
% \end{align}
% Now,~\eqref{scalingfactor} follows from directly from the fact that the ...
\end{proof}

As we have shown in the case $\fish_0=0$, i.e.\ in the absence of a priori knowledge, the
optimal locations of the sensors \(x\) are independent of the cost parameter $\beta$ (resp,\ \(K\)), which only
affects the scaling of the coefficients \(\lambda\). However for $\fish_0 \neq 0$ this is
generally not the case.
Loosely speaking, if the a~priori information is relatively good
(i.e.\ \(\fish_0 \in \operatorname{PD}(n)\)) and the cost per measurement is too high, the optimal design is
given by the zero function, i.e.\ the experiment should not be carried out at all.
\begin{proposition}
Let $\fish_0 \in \operatorname{PD}(n)$. Then the zero function
$\optom = 0$ is an optimal solution to~\eqref{def:designprop} if and only if
$\beta > \beta_0 = -\min_{x\in\Om}\psi'(0)$.
\end{proposition} 
\begin{proof}
We first note that \(0 \in \dom \psi\) and \(\beta_0 = -\min_{x\in\Om} \psi'(0) <
\infty\). Clearly, for \(\beta \geq \beta_0\), the zero function fulfills the optimality
conditions from Lemma~\ref{lem:optcond}. Thus, it is a solution to~\eqref{def:designprop}.
Conversely, for \(\beta < \beta_0\), the optimality conditions are violated.
% For an arbitrary $\beta>0$ define the non-empty set of optimal solutions to~\eqref{def:designprop} by
% \begin{align*}
% \mathbf{\Omega}(\beta)=\left\{\, \optomb \in M^+(\Om)\;| \;\optomb ~\text{fulfills}~\eqref{suppcongeneral}\,\right\}.
% \end{align*}
% Let now $\left\{\optom_{\beta_n} \right \}_{n \in \N}$ denote an arbitrary sequence of optimal solutions to \eqref{def:designprop} for $\beta_n >0$, $\beta_n \rightarrow \infty$ for $n \rightarrow \infty$. Since $\fish_0\in \operatorname{PD}(n)$ we obtain
% \begin{align} \label{norm}
% \mnorm{\optom_{\beta_n}}\leq F(\optom_{\beta_n})/ \beta_n \leq F(0)/ \beta_n
% \end{align} 
% and consequently $\optom_{\beta_n} \rightarrow 0$ for $n \rightarrow \infty$. 
% Because of \eqref{norm}, given a fixed $\tilde{\beta}>0$, the set
% $\bigcup_{\beta\geq \tilde{\beta}} \mathbf{\Omega}(\beta)$ is bounded. Due to its weak*-to-strong continuity, the gradient $\psi'$ maps bounded sets to bounded sets, i.e. there exists $c>0$ with $\|\psi'(\om)\|_{C(\Om)}\leq c$ for $\om \in \mathbf{\Omega}(\beta) $. Choosing $\hat{\beta}=\max \left\{\tilde{\beta}, 2c \right\}$ and an arbitrary $\beta \geq \hat{\beta}$, there holds 
% \begin{align*}
% -\psi'(\om) < \hat{\beta} \leq \beta \quad \forall \om \in \mathbf{\Omega}(\beta).
% \end{align*}
% Consequently we arrive at 
% \begin{align*}
% \left\{\,x\in \Om\;|\;-\psi'(\optomb)(x)=\beta \,\right\}=\emptyset,
% \end{align*}
% yielding $\mathbf{\Omega}(\beta)=\{0\}$.
\end{proof}

\section{Algorithmic solution} \label{sec:algorithm}
In this section we will elaborate on the solution of \eqref{def:designprop}.
We consider two different approaches.
First, we present an algorithm relying on finitely supported iterates and the sequential
insertion of single Dirac Delta functions based on results for a linear-quadratic optimization
problem in \cite{bredies2013inverse} and \cite{bredies2009generalized}. We derive all
necessary results to prove convergence of the generated sequence of measures towards a
minimizer of \eqref{def:designprop} together with a sub-linear convergence rate of the
objective function value. Additionally we propose to alternate between point insertion and
point deletion steps to benefit the sparsity of the iterates and to speed
up the convergence of the algorithm in practice. These sparsification steps are based on
the approximate solution of finite dimensional optimization problems in every
iteration. As an example we give two explicit realizations for the point removal and
discuss the additional computational effort in comparison to an algorithm solely based on
point insertion steps. \changed{Moreover, we propose a sparsification procedure based on the proof of Theorem~\ref{thm:existence of finite}, which ensures that the support size of all iterates is uniformly bounded and guarantees a sparse structure of the computed optimal design.}

%Moreover the resulting algorithms can be combined with
%Algorithm~\ref{alg:postprocessing} in a straightforward manner, guaranteeing a sparse
%structure of the computed optimal design.
%
%Secondly, we adapt an approach based on a Hilbert space regularization of the original sparse
%optimization problem. Here, the optimal design problem \eqref{def:designprop} is
%replaced by a sequence of regularized optimization problems, which are amenable to
%proximal point or semismooth Newton methods (which converge locally superlinearly). Algorithmic
%approaches for the solution of non-smooth optimization problems based on Hilbert space
%regularizations have recently increased in interest in the context of PDE-constrained
%optimization; see, e.g., \cite{stadler2009elliptic,clason2011duality}.
%We briefly describe this approach for the sake of comparison at the end of this section.

\subsection{A generalized conditional gradient method}
\label{subsec:condgrad}
For the direct solution of \eqref{def:designprop} on the admissible set $M^+(\Om)$ we
adapt the numerical procedure presented in \cite{bredies2013inverse}, which relies on
finitely supported iterates.  A general description of the method is given in
Algorithm~\ref{alg:conditionalgradient}. For convenience of the reader
we give a detailed description of the individual steps and their derivation.
\begin{algorithm}
\begin{algorithmic}
\STATE 1. Choose $\om^1 \in \dom_{M^+(\Om)}~\psi$, $\# \supp\om^1 \leq n(n+1)/2$. Set $M_0= F(\om^1)/ \beta$.
 \WHILE {$\Phi(\om^k)\geq \mathrm{TOL}$}
 \STATE 2.
 Compute $\psi'_k=\psi'(\om^k)$.
 Determine $\hat{x}^k \in \argmin_{x\in\Om}~\psi'_k(x)$.
 \STATE 3.
 Set $v^k = \theta^k \delta_{\hat{x}^k}$ with
 $\theta^k = \begin{cases} 0, & \psi'_k(\hat{x}^k) \geq -\beta, \\
   -(M_0/\beta)\psi'_k(\hat{x}^k) , & \text{else} \end{cases}$
 \STATE 4.
 Select a step size \(s^k \in (0,1]\) and set
 \(\om^{k+1/2} = (1-s^k)\om^k + s^k v^k\).
 \STATE 5.
 Find $\om^{k+1}$ with $\supp\om^{k+1} \subseteq \supp\om^{k+1/2}$ and $F(\om^{k+1})\leq F(\om^{k+1/2})$.
 \ENDWHILE
\end{algorithmic}
\caption{Successive point insertion}\label{alg:conditionalgradient}
\end{algorithm}
The basic idea behind the procedure relies on a point insertion process (steps~2.--4.\ in
Algorithm~\ref{alg:conditionalgradient}) related to a generalized conditional gradient
method.
More precisely, they consist of conditional gradient steps for a surrogate optimization
problem with the same optimal solutions, in which the sublinear total variation norm is replaced by
a coercive cost term for designs of very large norm.
Additionally, we consider the minimization of the finite dimensional subproblem that
arises from restriction of the design measure to the 
active support of the current iterate (in step~5.). This is motivated on the one hand by the
desire to potentially remove non-optimal support points
by setting the corresponding coefficient to zero, and on the other hand by the desire to
obtain an accelerated convergence behavior in practice.
%In a special case, we obtain the
%simplified procedure given in Algorithm~\ref{alg:primaldualactivepoints}
%\begin{algorithm}
%\begin{algorithmic}
%\STATE 1. Choose $\om^1 \in \dom_{M^+(\Om)}~\psi$, $\# \supp\om^1 \leq n(n+1)/2$.
% \WHILE {$\Phi(\om^k)\geq \mathrm{TOL}$}
% \STATE 2.
% Compute $\psi'_k=\psi'(\om^k)$, $\hat{x}^k \in \argmin_{x\in\Om}~\psi'_k(x)$
% and set \(\mathcal{A}_{k+1} = \supp\om^k \cup \{\,\hat{x}_k\,\}\)
% \STATE 3.
% Find $\om^{k+1}$ with $\supp\om^{k+1} \subseteq \mathcal{A}_{k+1}$ that solves
% $\min_{\om \in M^+(\Om),\; \supp\om \subseteq \mathcal{A}_{k+1}} F(\om)$.
% \ENDWHILE
%\end{algorithmic}
%\caption{Primal-dual active point method}\label{alg:primaldualactivepoints}
%\end{algorithm}

This section is structured as follows: First, we focus on the point insertion step and its descent properties.
By a suitable choice of the step 
size $s^k$ in each step of the procedure we are able to prove a sub-linear convergence
rate for the objective functional value. Secondly, we consider concrete examples for
the point removal step 5.
%\1 and discuss the applicability of
%Algorithm~\ref{alg:postprocessing} in the context of the successive point insertion
%algorithm.

\subsubsection{Convergence analysis}
As already pointed out, Algorithm~\ref{alg:conditionalgradient}
relies on a coercive surrogate design problem which admits the same optimal solutions as
\eqref{def:designprop}.
Given a constant $M_0>0$, we start by introducing the auxiliary function $\varphi_{M_0}
\colon \R_+ \to \R$ as
\begin{align*}
\varphi_{M_0}(t)
  = \begin{cases} t, & t \leq M_0, \\
    (1/(2M_0)) \left[t^2 + M_0^2\right], & \text{else}, \end{cases}
\end{align*}
and consider the modified problem
\begin{align}
\label{def:designpropaux}\tag{\ensuremath{P^{M_0}_{\beta}}}
\min_{\om \in M^+(\Om)} F_{M_0}(\om) = \psi(\om) + \beta \varphi_{M_0} (\mnorm{\om})
\end{align}
for the special choice of $M_0 = F(\om^1)/\beta$, with arbitrary but fixed $\om^1 \in
\dom_{M^+(\Om)} \psi$. Note that, for all $\om \in M^+(\Om)$ with $F(\om)\leq
F(\om^1)$, there holds $\mnorm{\om} \leq M_0$ and
consequently $F(\om)=F_{M_0}(\om)$ . We additionally point out that
\begin{equation}\label{eq:optomnorm}
\varphi_{M_0}(\mnorm{\optomb}) = \mnorm{\optomb}
\end{equation}
for every optimal solution $\optomb$ of \eqref{def:designpropaux}.
Connected to this auxiliary problem we additionally define the primal-dual gap
$\Phi\colon \dom \psi \to [0, \infty)$ by 
\begin{align*}
\Phi(\om) = \sup_{v \in M^+(\Om)} \left\lbrack\pair{\psi'(\om),\om-v} + \beta\mnorm{\om} - \beta\varphi_{M_0}(\mnorm{v})\right\rbrack.
\end{align*}
Note that the value of \(\Phi\) is finite for every \(v \in \dom\psi\),
which follows with the coercivity of \(\varphi_{M_0}(\cdot)\).
In the next proposition we collect several results to establish the connection between the optimal design problems \eqref{def:designprop} and \eqref{def:designpropaux}.
\begin{proposition}
Let $\om^1 \in \dom_{M^+(\Om)} \psi$ be arbitrary but fixed and set
$M_0=F(\om^1)/ \beta$. Given $\optomb \in \dom_{M^+(\Om)} \psi$ the
following three statements are equivalent:
\begin{itemize}
\item [1.] The measure $\optomb$ is a minimizer of \eqref{def:designprop}.
\item [2.] The measure $\optomb$ is a minimizer of \eqref{def:designpropaux}.
\item [3.] The measure $\optomb$ fulfils $\Phi(\optomb)=0$.
\end{itemize}
Furthermore there holds
\begin{align}\label{est:residual}
\Phi(\om) \geq F(\om)- F(\optomb) =: r_F(\om),
\end{align}
for all $\om \in \dom_{M^+(\Om)}\psi,\mnorm{\om}\leq M_0$ and all minimizers $\optomb$ of \eqref{def:designpropaux}.
\end{proposition} 
\begin{proof}
The equivalence between the first two statements can be proven as in \cite{bredies2013inverse}. We only prove the third one. Similar to the proof of \eqref{subdifferentialinequ} (see Proposition \ref{prop:firstordernec}) a given $\optomb \in \dom_{M^+(\Om)}\psi$ is a minimizer of \eqref{def:designpropaux} if and only if it fulfills
\begin{align*}
- \pair{\psi'(\optomb), \om-\optomb} + \beta\varphi_{M_0} (\mnorm{\optomb})
 \leq \beta\varphi_{M_0} (\mnorm{\om}) \quad \forall \om \in M^+(\Om).
\end{align*}
By reordering and taking the minimum over all $\om \in M^+(\Om)$ this can be equivalently written as
\begin{align*}
\sup_{\om\in M^+(\Om)} \left\lbrack\pair{\psi'(\optomb),\optomb-\om}
 + \beta\varphi_{M_0}(\mnorm{\optomb}) - \beta\varphi_{M_0}(\mnorm{\om}) \right\rbrack
  = 0.
\end{align*}
Utilizing \eqref{eq:optomnorm} we find $\Phi(\optomb)=0$ if and only if $\optomb$ is a
minimizer of $F_{M_0}$. It remains to prove \eqref{est:residual}.
Given $\om \in \dom_{M^+(\Om)} \psi$ with $\mnorm{\om}\leq M_0$
and a minimizer $\optomb$ we obtain
\begin{align}\label{eq:Fdiff}
  F(\om)-F(\optomb)\leq \beta \mnorm{\om} - \beta\mnorm{\optomb} + \pair{\psi'(\om),\om-\optomb},
\end{align}
by convexity of $\psi$. Noting that
\begin{align*}
 -\lbrack \beta \mnorm{\optomb} + \pair{\psi'(\om),\optomb} \rbrack
  \leq -\inf_{v\in M^+(\Om)} \lbrack\pair{\psi'(\om),v} + \beta\varphi_{M_0}(\mnorm{v})\rbrack,
\end{align*}
the right-hand side in \eqref{eq:Fdiff} is estimated by $\Phi(\om)$, which concludes the proof.
\end{proof}
With the result of the last proposition we may consider a minimization algorithm
for~\eqref{def:designpropaux} in order to compute optimal solutions to
\eqref{def:designprop}.
Additionally, the result sug\-gests the use of $\Phi$ as a convergence criterion, since it gives an upper bound for the
residual error in the objective function value. As can be seen below, the evaluation of
$\Phi$ can be easily computed as a by-product of steps 2.--3. in
Algorithm~\ref{alg:conditionalgradient}.

The algorithm operates on finitely supported iterates $\om^k= \sum_{i=1}^{m_k}
\lambda^k _i \delta_{x^k_i}$ with distinct support points $x^k_i \in \Om$ and positive
coefficients $\lambda^k_i$, $i \in \{\,1, \ldots, m_k\,\}, m_k \in \N$.
% A decrease of the
%objective function value in every iteration will be ensured, which implies
%\begin{align*}
%\beta \mnorm{\om^{k+1}}\leq F_{M_0}(\om^{k+1})\leq {F}_{M_0}(\om^{k})\leq F_{M_0}(\om^1 ),
%\end{align*}
%and consequently $F_{M_0}(\om^{k}) = F(\om^{k})$ for all iterates $\om^k$.
In steps 2.--4.\ the intermediate iterate $\om^{k+1/2}$ is obtained as a convex combination
between the previous iterate $\om^k$ and a scaled Dirac delta function $\theta^k
\delta_{\hat{x}^k}$ inserted at the global minimum of the gradient $\psi'(\om^k)$. The
initial coefficient $\theta^k$ is determined by the current maximal violation of the lower
bound on the gradient of $\psi$; see \eqref{suppcongeneral}. In the following lemma we
relate this definition to the computation of a descent direction in the context of a
generalized conditional gradient method
(cf.~\cite{rakotomamonjy2015generalized,bredies2013inverse, bredies2009generalized})
for the auxiliary problem \eqref{def:designpropaux}.
\begin{lemma}\label{lem:minimizeroflin}
Let $\om^k \in \dom_{M^+(\Om)}\psi$ be given. Then the measure
$v^k = \theta^k \delta_{\hat{x}^k}$ with $\hat{x}^k\in \Om$ and $\theta^k \geq 0$ as defined in
steps 2.--3.\ of Algorithm~\ref{alg:conditionalgradient} is a minimizer of
\begin{align*} \label{def:linearizeddesign}
\min_{v \in M^+(\Om)} \pair{\psi'(\om^k), v} + \beta\varphi_{M_0}(\mnorm{v}).
\tag{\ensuremath{P^{\mathrm{lin}}_{\beta}}}
\end{align*}
Moreover, \(v^k\) realizes the supremum in the definition of the primal-dual gap: it holds
\(\Phi(\om^k) = \pair{\psi'(\om^k),\om^k-v^k} + \beta\mnorm{\om^k} - \beta\varphi_{M_0}(\mnorm{v^k})\).
\end{lemma}
\begin{proof}
\changed{We note that \eqref{def:linearizeddesign} can be equivalently expressed as
\begin{align} \label{eq:linref} 
\min_{\theta\in [0, \infty)}
  \theta  \min_{\substack{\tilde{v} \in M^+(\Om), \\ \mnorm{\tilde{v}}=1}} \pair{\psi'(\om^k), \tilde{v}} + \beta\varphi_{M_0}(\theta)
\end{align}
Due to~$\psi'(\om^k)\leq 0$ and~$\theta \geq 0$, a solution to the inner minimization problem is given by~$\tilde{v}^k= \delta_{\hat{x}^k}$ with~$\hat{x}^k \in \argmin_{x\in \Om} \psi'(\om^k)(x)$. In fact we have
\begin{align*}
\langle \psi'(\om^k), \tilde{v} \rangle \geq \langle \psi'(\om^k), \tilde{v}^k \rangle=\min_{x\in \Om} \psi'(\om^k)(x) \quad \forall \tilde{v} \in M^+(\Om),~\mnorm{\tilde{v}}=1.
\end{align*}
Thus problem~\eqref{eq:linref} reduces to
\begin{align*}
\min_{\theta\in [0, \infty)} \theta \min_{x\in\Om} \psi'(\om^k)(x)+ \beta\varphi_{M_0}(\theta).
\end{align*}
By straightforward calculations, we verify that~$\theta^k$ as defined in step 2. of Algorithm~\ref{alg:conditionalgradient} is a minimizer of this problem. We conclude that $v^k=\theta^k \tilde{v}^k$ is a solution of~\eqref{def:linearizeddesign}. This finishes the proof of the first statement. Moreover, the second statement follows due to
\begin{align*}
\Phi(\om^k) &= \sup_{v \in M^+(\Om)} \left\lbrack\pair{\psi'(\om^k),\om-v} + \beta\mnorm{\om^k} - \beta\varphi_{M_0}(\mnorm{v})\right\rbrack \\&=\pair{\psi'(\om^k),\om}+\beta\mnorm{\om^k}-\min_{v \in M^+(\Om)} \pair{\psi'(\om^k), v} + \beta\varphi_{M_0}(\mnorm{v})\\
&= \pair{\psi'(\om^k),\om^k-v^k} + \beta\mnorm{\om^k} - \beta\varphi_{M_0}(\mnorm{v^k}).
\qedhere
\end{align*}
%The concrete expression of \(v^k\) follows now by a straightforward computation using the
%positive homogeneity of the total variation norm and the definition of
%\(\varphi_{M_0}(\cdot)\). Clearly, \(\Phi(v^k)\)
%agrees to \(- \min \eqref{def:linearizeddesign}\) up to a constant value.
}
\end{proof}
\begin{remark}
At this point, replacing \eqref{def:designprop} by the equivalent formulation
\eqref{def:designpropaux} is crucial. In fact, the partially linearized problem
corresponding to the original problem
\begin{align*}
\min_{v \in M^+(\Om)} \pair{\psi'(\om), v} + \beta \mnorm{v},
\end{align*}
is either unbounded or has an unbounded solution set in the case that 
$\min_{x \in \Om}\psi'(\om) \leq -\beta$.
\end{remark}
Note that, as a by-product of the last result, the convergence criterion $\Phi(\om^k)$ can be
evaluated cheaply, once the current gradient $\psi'(\om^k)$ and its minimum point are calculated.

We form the intermediate iterate as convex combination between the old iterate and the new sensor i.e., ${\om^{k+1/2} = (1-s^k)\om^k + s^k v^k}$, where $s^k \in (0, 1]$ is suitably
chosen. This ensures $\om^{k+1/2}\in M^+(\Om)$. The step size $s^k$ will be chosen by the
following generalization of the well-known Armijo-Goldstein condition; see, e.g.,
\cite{bredies2009generalized}. This choice of the step size ensures a sufficient decrease of the objective
function value in every iteration of Algorithm~\ref{alg:conditionalgradient} and the
overall convergence of the presented method. More precisely, for fixed $\gamma \in (0,1)$,
$\alpha \in (0,1/2]$, the step size is set to $s^k=\gamma^{n_k}$, where $n_k$ is the
smallest non-negative integer with
 \begin{align}\label{Def:Armijo}
\alpha s^k \Phi(\om^k)\leq F_{M_0}(\om^k)-F_{M_0}(\om^{k}+s^k(v^k-\om^k)).
\end{align}
Note that given an arbitrary non-optimal $\om^k \in \dom_{M^+(\Om)}\psi$ with $\mnorm{\om^k}\leq M_0$ this choice of the step size $s^k$ is always possible since the function $W\colon [0,1]\rightarrow \R \cup \{-\infty\}$
\begin{align} \label{def:helpfunction}
W(s)=\frac{F_{M_0}(\om^k)-F_{M_0}(\om^k +s(v^k-\om^k))}{s\Phi(\om^k)},
\end{align}
fulfills $\lim_{s\rightarrow 0} W(s)\geq 1$, similarly to
\cite[Remark~2]{bredies2009generalized}.
\changed{Note that the left-hand side of~\eqref{Def:Armijo} is positive if~$\om^k$
is not optimal. Thus, the quasi-Armijo-Goldstein stepsize rule ensures a decrease of the
objective function value in each iteration. In particular, we get
\begin{align*}
\beta \mnorm{\om^{k+1}}\leq F_{M_0}(\om^{k+1})\leq F_{M_0}(\om^{k+1/2})\leq {F}_{M_0}(\om^{k})\leq F_{M_0}(\om^1 ),
\end{align*}
and consequently $F_{M_0}(\om^{k}) = F(\om^{k})$ for all iterates $\om^k$.}
To obtain quantifiable estimates for the descent in the objective function value we impose additional regularity assumptions on $\Psi'$ until the end of this section.
\begin{assumption} \label{ass:locallipschitz}
Assume that $\Psi'$ is Lipschitz-continuous on compact sets: Given a compact set $\mathcal{N}\subset \dom \Psi$ there exists $L_{\mathcal{N}}>0$ with
\begin{align}\label{lipschitzPsi}
\sup_{N_1, N_2\in \mathcal{N}} \frac{\norm{\Psi'(N_1)-\Psi'(N_2)}}{\norm{N_1-N_2}} \leq L_{\mathcal{N}},
\end{align}
where $\norm{A} = \norm{A}_{\operatorname{Sym}(n)} = \sqrt{\Tr(AA^\top)}$ is the Frobenius norm.
\end{assumption}
Note that this additional assumption is fulfilled if the design criterion $\Psi$ is
two-times continuously differentiable on its domain. This is the case for, e.g., the already mentioned A and
D-criterion
% as well as for the smoothed E-criterion as introduced in 
, see Section
\ref{sec:existence_and_optimality}. We immediately arrive at the following proposition.
\begin{proposition} \label{prop:locallipschitz}  
Let Assumption \ref{ass:locallipschitz} hold and let $\om_1 \in \dom_{M^+(\Om)} \psi$ be given. Define the associated sub-level set $E_{\om^1}$ as
\begin{align*}
E_{\om^1}=\left\{\,\om \in M^+(\Om)\;|\;F(\om)\leq F(\om^1)\,\right\}.
\end{align*}
 Then there exists $L_{\om^1}$ such that
\begin{align}\label{lipschitz}
\sup_{\om_1, \om_2\in E_{\om^1}} \frac{\|\psi'(\om_1)-\psi'(\om_2)\|_{C(\Om)}}{\mnorm{\om_1-\om_2}} \leq L_{\om^1}.
\end{align}
\end{proposition}
\begin{proof}
First we observe that $E_{\om^1}$ is convex, bounded, and weak* closed. Consequently the set of associated information matrices
\begin{align*}
\fish(E_{\om^1})=\left\{\,\fish(\om)+\fish_0\;|\;\om \in E_{\om^1}\,\right\}
\end{align*}
is compact.
For $\om_1, \om_2 \in E_{\om^1}$ we  obtain
\begin{multline*}
\|\psi'(\om_1)-\psi'(\om_2)\|_{C(\Om)}=\|\fish^*\Psi'(\fish(\om_1)+\fish_0)-\fish^*\Psi'(\fish(\om_2)+\fish_0)\| \\ \leq 
\|\fish^*\|_{\operatorname{Sym}(n)\rightarrow C(\Om)}
\|\Psi'(\fish(\om_1)+\fish_0)-\Psi'(\fish(\om_2)+\fish_0)\|
\\ \leq L_{\mathcal{I}(E_{\om^1})}
\|\fish^*\|_{\operatorname{Sym}(n)\rightarrow C(\Om)}
\|\fish(\om_1)-\fish(\om_2)\|
\\
\leq L_{\mathcal{I}(E_{\om^1})}
\|\fish^*\|_{\operatorname{Sym}(n)\rightarrow C(\Om)}
\|\fish\|_{M^+(\Om)\rightarrow \operatorname{Sym}(n)}
\mnorm{\om_1-\om_2},
\end{multline*}
completing the proof.
\end{proof}
Using this additional local regularity we obtain the following estimate on the growth
behavior of the function \(F\) at \(\omega^k\) in the search direction.
\begin{lemma} \label{lemma:descent}
Let Assumption~\ref{ass:locallipschitz} hold.
Let $\om^k \in \dom_{M^+(\Om)}\psi$ with $\mnorm{\om^k} \leq M_0$ and $v^k$
as in Lemma~\ref{lem:minimizeroflin} be given. Moreover, let $\om^{k+1/2}_s=(1-s)\om^k
+s v^k$ with $s\in [0,1]$ and~$\om^{k+1/2}_s \in E_{\om^k}$ be given. Then there holds
\begin{align*}
F_{M_0}(\om^{k+1/2}_s)-F_{M_0}(\om^{k})\leq -s \Phi(\om^k)+\frac{L_{\om^k}}{2} s^2  \mnorm{v^k -\om^k}^2,
\end{align*}
where $L_{\om^k}$ denotes the Lipschitz constant of $\psi'$ on $E_{\om^k}$.
\begin{proof}
By assumption there holds $F_{M_0}(\om^{k+1/2}_s)\leq F_{M_0}(\om^{k})$ and consequently $\om^{k+1/2}_s \in E_{\om^k}$.
Therefore we obtain 
\begin{multline*}
F_{M_0}(\om^{k+1/2}_s)-F_{M_0}(\om^{k}) 
= -s \pair{\psi'(\om^k), \om^k -v^k} \\
+ \beta\varphi_{M_0}(\mnorm{\om^{k+1/2}_s})
- \beta\varphi_{M_0}(\mnorm{\om^k})
+ \int_0^{s} \pair{\psi'(\om_{\sigma})-\psi'(\om^k),v^k-\om^k} \de \sigma ,
\end{multline*}
with $\om_{\sigma} = \om^k + \sigma (v^k-\om^k)$ for $\sigma\in[0,1]$.
Using the convexity of $\varphi_{M_0}(\mnorm{\cdot})$ we obtain
\begin{multline*}
-s \pair{\psi'(\om^k), \om^k-v^k}
 + \beta\varphi_{M_0}(\mnorm{\om^{k+1/2}_s}) - \beta\varphi_{M_0}(\mnorm{\om^k})
\\ \leq 
-s \left(\langle \psi'(\om^k), \om^k -v^k \rangle +\beta \varphi_{M_0}(\mnorm{\om^k})-\beta\varphi_{M_0}(\mnorm{v^k})\right),
\end{multline*}
where the right-hand side simplifies to $-s \Phi(\om^k)$.
Due to the Lipschitz continuity of $\psi'$ on $E_{\om^k}$ we get 
\begin{multline*}
\int_0 ^{s} \pair{\psi'(\om_{\sigma})-\psi'(\om^k),v^k-\om^k} \de \sigma \leq \mnorm{v^k-\om^k} 
\int_0 ^{s} \norm{\psi'(\om_{\sigma})-\psi'(\om^k)}_{C(\Om)} \de \sigma \\
  \leq  L_{\om^k} \mnorm{v^k-\om^k}^2
\int_0 ^{s} \! \sigma \mathrm{d} \sigma= \frac{L_{\om^k} s^2}{2} \mnorm{v^k -\om^k}^2.
\end{multline*}
Combining both estimates yields the result.
\end{proof}
\end{lemma}
In order to prove the main result we additionally need the following technical lemma.
\begin{lemma} \label{lem:auxisalpha}
Let $\om^k \in \dom_{M^+(\Om)}\psi$ with $\Phi(\om^k)>0$ be given. The
function 
\begin{align*}
W\colon (0,1]\rightarrow~ \R \cup \{-\infty \}
\end{align*}
 from \eqref{def:helpfunction} is
continuous on $(0,1)$. Furthermore, denoting by $s^k$ the step size from
\eqref{Def:Armijo}, there exists $\hat{s}^k \in [s^k, s^k/ \gamma]$ with
$W(\hat{s}^k)=\alpha$ if $s^k < 1$.
\end{lemma}
\begin{proof}
First, note that for $s\in [0,1)$ we have $\om_s = (1-s)\om^k + sv^k \in
\dom_{M^+(\Om)} \psi$ due to
\(\fish(\om_s) + \fish_0 = (1-s)\fish(\om^k)
+ s \theta_k \, \sensv(\hat{x}^k)\sensv(\hat{x}^k)^\top + \fish_0 \in \operatorname{PD}(n)\).
Furthermore, using Assumption~\ref{ass:designcrit} it can be verified that
\begin{align*}
W(s) = (F_{M_0}(\om_0)-F_{M_0}(\om_s))/(s\,\Phi(\om_0))
\end{align*}
is continuous on $s \in (0,1)$. Additionally, with lower semi-continuity of \(\Psi\), we verify that
\(W(s) \to -\infty\) for \(s \to 1\) in case that \(\fish(v^k) \not\in\dom \Psi\).
We conclude the proof by applying the mean value theorem on $[s^k, s^k/ \gamma] \subset
(0, 1]$, taking into account that $W(s^k) \geq \alpha > W(s^k/ \gamma)$.
\end{proof}
Combining the previous results we are able to prove sub-linear convergence of the presented algorithm.
\begin{theorem}\label{thm:weakconvergencerate}
Let the sequence $\om^k $ be generated by Algorithm~\ref{alg:conditionalgradient} using
the quasi-Armijo-Goldstein condition \eqref{Def:Armijo}. Then there exists at least one
weak* accumulation point $\optomb$ of $\om^k$ and every such point is an optimal solution
to \eqref{def:designprop}. Additionally there holds
\begin{align} \label{rate}
r_F(\om^k)\leq \frac{r_F(\om^1)}{1+q(k-1)}
\end{align}
with
\begin{align}
q=\alpha \min \left \{\, \frac{c_1}{L_{\om^1} (M_0 +c_2)^2},\; 1 \,\right\},
\end{align}
where $L_{\om^1}$ is the Lipschitz-constant of $\psi'$ on $E_{\om^1}$, $M_0=F(\om^1)/\beta$, $c_1=2\gamma(1-\alpha)r_F(\om_1)$ and a constant $c_2> 0$ with $\mnorm{v^k}\leq c_2$ for all $k$. 
\end{theorem}
\begin{proof}
Assume without restriction that $\Phi(\om^k)>0$, i.e.\ the algorithm
does not terminate after finitely many steps. By construction and the choice of $s^k$
there holds $\om^k\in E_{\om^1}$ and consequently $\mnorm{\om^k} \leq M_0$,
$F_{M_0}(\om^k)=F(\om^k)$ for all $k$. The same can be proven for
$\om^{k+1/2}$. Note that $\om^k$ is bounded and $\psi'$ is weak*-to-strong
continuous. Therefore, there exists $c_2>0$ with $\mnorm{v^k}\leq c_2$ for all  $k$ s.
 
By the definition of the step size $s^k$ as well as
\eqref{est:residual} there holds
\begin{align*}
\alpha s^k r_F(\om^k)\leq \alpha s^k \Phi(\om^k)\leq r_F(\om^k)-r_F(\om^{k+1^/2}),
\end{align*}
which yields
\begin{align} \label{descentinresidual}
r_F(\om^{k+1/2})\leq(1-\alpha s^k) r_F(\om^k).
\end{align}
Since $\Phi(\om^k)>0$ we obtain $s^k \neq 0$ for all $k$. Two cases have to be distinguished. If $s^k$ is equal to one we immediately arrive at
\begin{align*}
r_F({\om}^{k+1/2})\leq(1-\alpha) r_F(\om^k)\leq r_F(\om^k) -\alpha \frac{r_F(\om^k)^2}{r_F(\om^1)}.
\end{align*}
In the second case, if $s^k< 1$, there exists $\hat{s}^k \in [s^k, s^k / \gamma]$ with 
\begin{align*}
\alpha=\frac{F(\om^k)-F(\om^k +\hat{s}^k(v^k-\om^k))}{\hat{s}^k \Phi(\om^k)},
\end{align*}
using Lemma \ref{lem:auxisalpha}. Consequently $\om^k +s(v^k-\om^k)\in E_{\om^1}$ for all $0\leq s\leq \hat{s}^k$ due to the convexity of $F$.
Because of the Lipschitz-continuity of $\psi'$ on $E_{\om_1}$, Lemma~\ref{lemma:descent}
can be applied and, defining $\delta \om^k =v^k-\om^k$, there holds
\begin{align*}
\alpha = \frac{F(\om^k)-F(\om^{k}+\hat{s}^k\delta\om^k)}{\hat{s}^k\Phi(\om^k)}
\geq 1- \frac{L_{\om^1} \hat{s}^k}{2}  \frac{\mnorm{\delta \om^k}^2}{\Phi(\om^k)}
\geq 1- \frac{L_{\om^1} s^k}{2\gamma}  \frac{\mnorm{\delta \om^k}^2}{\Phi(\om^k)}.
\end{align*}
The last estimate is true because of $\hat{s}^k \leq s^k/\gamma$.  
Reordering and using~\eqref{est:residual} yields
\begin{align*}
1\geq s^k \geq 2 \gamma (1-\alpha) \frac{\Phi(\om^k)}{L_{\om^1} \mnorm{v^k -\om^k}^2}\geq 2 \gamma (1-\alpha) \frac{r(\om^k)}{L_{\om^1} \mnorm{v^k -\om^k}^2}.
\end{align*}
Combining the estimates in both cases and using $r_F(\om^{k+1})\geq r_F(\om^{k+1/2})$, the inequality
\begin{align}\label{recursion}
0\leq \frac{r_F(\om^{k+1})}{r_F(\om^1)}\leq \frac{r_F(\om^{k+1/2})}{r_F(\om^1)}\leq \frac{r_F(\om^{k})}{r_F(\om^1)}-q_k \left ( \frac{r_F(\om^{k})}{r_F(\om^1)}\right )^2\quad \forall k\in \mathbb{N}
\end{align}
holds, where the constant $q_k$ 
is given by
\begin{align*}
q_k=r_F(\om^1) \alpha \min \left \{ \frac{2 \gamma (1-\alpha)}{L_{\om^1} \mnorm{v^k -\om^k}^2} ,\frac{1}{r_F(\om^k)}\right \} ,
\end{align*}
if $s^k<1$ and $q_k=\alpha$ otherwise. We estimate
\begin{align*}
q_k\geq \alpha \min \left \{ \frac{2 \gamma (1-\alpha)r_F(\om^1)}{L_{\om^1} (M_0+c_2)^2} ,1\right \}=:q.
\end{align*}
The claimed convergence rate \eqref{rate} now follows directly from the recursion
formula~\eqref{recursion}; see~\cite[Lemma~3.1]{dunn1980convergence}. \changed{ Consequently, each
subsequence of $\om^k$ is a minimizing sequence. Since~$\om^k$ is bounded, it admits at least one weak* accumulation point. Due to the derived convergence rate and the weak* lower semi-continuity
of $F$ each weak* accumulation point
$\optomb$ is a minimizer of~\eqref{def:designprop}}.
\end{proof}

\subsection{Acceleration and sparsification strategies}
\label{sec:acceleration}
As we have seen in the previous section, an iterative application of steps~2.--4.\ in
Algorithm~\ref{alg:conditionalgradient} is sufficient to obtain weak* convergence of the
iterates $\om^k$, as well as a sublinear convergence rate for the objective function. However, it is
obvious that the support size of the iterates $\om^k$ grows monotonically in every
iteration unless the current gradient is bounded from below by $-\beta$ or, more unlikely,
the step size $s^k$ is chosen as $1$. Therefore, while the implementation of
steps~2.--4.\ is fairly easy, an algorithm only consisting of point insertion steps will likely yield iterates
with undesirable sparsity properties, e.g., a clusterization of the intermediate support
points around the support points of a minimizer to \eqref{def:designprop}. 
In the following we mitigate those effects by augmenting the point insertion steps by
point removal steps, where we incorporate ideas from \cite{bredies2013inverse, boyd2015alternating}.
\changed{For \(\{x_j\}_{j=1}^{m_k} = \supp \om^{k+1/2}\), we define the parameterization:
\begin{equation}
\label{eq:measure_on_active_set}
 \omparam(\lambda) := \sum_{j=1}^{m_k} \lambda_j \delta_{x_j} 
 \quad\forall \lambda \in \R^{m_k}.
\end{equation}
}
Now, we set $\om^{k+1} = \omparam(\lambda^{k+1})$,
where the improved vector $\lambda^{k+1} \in \R^{m_k}$ is chosen as an approximate
solution to the (finite dimensional) coefficient optimization problem
\begin{equation} \label{def:subproblem}
 \min_{\lambda\in \R^{m_k},~\lambda \geq 0} F(\omparam(\lambda))
 = \psi(\omparam(\lambda)) + \beta\|\lambda\|_1,
\end{equation}
that fulfills $F(\om^{k+1}) \leq F(\om^{k+1/2})$.
In this manuscript, we focus on two special instances of this removal step, which are
detailed below.

In the first strategy, the new coefficient vector
$\lambda^{k+1}=\lambda^{k+1}({\sigma_k})$ is obtained by
\begin{align} \label{projectedgradient}
\lambda^{k+1}({\sigma_k})_j
  = \max \left\{\,\lambda^{{k+1/2}}_i
  - \sigma_k \left[\psi'(\om^{k+1/2})(x_{j})+\beta\right],\;0\,\right\}
\quad \forall j\in \{1,\dots, m_k\}, 
\end{align}
where $\sigma_k >0$ is a suitably chosen step size that avoids ascent in the objective
function value. This corresponds to performing one step of a projected gradient method on
\eqref{def:subproblem} using the previous coefficient vector $\lambda^{k+1/2}$ as a starting
point. Thus, step 5.\ in Algorithm~\ref{alg:conditionalgradient}
subtracts or adds mass at support point $x_j$ for
$-\psi'(\om^{k+1/2})(x_j) < \beta$ or $-\psi'(\om^{k+1/2})(x_j) > \beta$, respectively.
Furthermore, the new coefficient $\lambda^{k+1}_j$ of the Dirac delta function
$\delta_{x_j}$ is set to zero if
\begin{align*}
 \lambda^{{k+1/2}}_j - \sigma_k \left[\psi'(\om^{k+1/2})(x_{j})+\beta\right] \leq 0,
\end{align*}
removing the point measure from the iterate.

Secondly, we suppose that the finite-dimensional sub-problems \eqref{def:subproblem} can be solved exactly and choose
\begin{align} \label{PDAP}
 \lambda^{k+1}\in \argmin_{\lambda\in \R^{m_k},~\lambda \geq 0} F(\omparam(\lambda)).
\end{align}
In this case, the conditions
\begin{align*}
\supp\om^{k+1} \subset \supp\om^{k+1/2}, \quad F(\om^{k+1}) \leq F(\om^{k+1/2})
\end{align*}
are trivially fulfilled.
If all finite dimensional sub-problems are solved exactly, the method can be interpreted
as a method operating on a set of active points \(\mathcal{A}_k = \supp \omega^k\);
cf.~\cite{walter2017Helmholtz}: In each iteration, the minimizer \(\hat{x}^k\) of the
current gradient \(\psi'_k\) is added to the support set to obtain
\(\mathcal{A}_{k+1/2} = \mathcal{A}_{k} \cup \set{\hat{x}^k}\). Then, the
problem~\eqref{PDAP} is solved on the new support set (i.e.\ with $\supp \omega^{k+1/2}$
replaced by \(\mathcal{A}_{k+1/2}\) in the definition of~\eqref{eq:measure_on_active_set})
to obtain the next iterate \(\omega^{k+1}\). Note that the next active
set is given by \(\mathcal{A}_{k+1} = \supp \omega^{k+1}\), which automatically removes support points
corresponding to zero coefficients in each iteration.

%Finally, to be able to guarantee the a~priori bound \(\#\supp \omega^k \leq n(n+1)/2\)
%for the algorithmic solutions, we can apply
%Algorithm~\ref{alg:postprocessing} to the intermediate iterate $\om^{k+1/2}$ in step 5.\ of
%Algorithm~\ref{alg:conditionalgradient}. This ensures the convergence of the presented
%procedure towards a sparse minimizer of \eqref{def:designprop}.
\changed{Finally, the proof of Lemma \ref{lem:repres} leads to an implementable sparsifying procedure which, given an arbitrary finitely supported positive measure, finds a sparse measure choosing a subset of at most $n(n+1)/2$ support points and yielding the same information matrix at a smaller cost. The procedure is summarized in Algorithm \ref{alg:postprocessing}. Applying this method to the intermediate iterate $\om^{k+1/2}$ in step 5.\ of
Algorithm~\ref{alg:conditionalgradient} guarantees the a~priori bound \(\#\supp \omega^k \leq n(n+1)/2\) as well as the convergence of the presented
procedure towards a sparse minimizer of \eqref{def:designprop}.}
 \begin{proposition} \label{prop:postprocconepoint}
 Let $\om=\sum_{j=1}^m \lambda_j  \delta_{x_j}$ be given and assume that $\{\fish(\delta_{x_j})\}^m_{j=1}$ is linearly dependent.
 Denote by $\om_{\mathrm{new}}=\sum_{\{\, j\;|\;\lambda_{\mathrm{new},j}>0\,\}} \lambda_{\mathrm{new},j} \delta_{x_j}$ the measure that is obtained after one execution of the loop in Algorithm \ref{alg:postprocessing}. Then there holds 
 \begin{align*}
 F(\om_{\mathrm{new}})\leq F(\om),\quad \# \supp \om_{\mathrm{new}} \leq \# \supp \om -1.
 \end{align*}
 \end{proposition}
\begin{proof}
\changed{ We point to the proof of Lemma~\ref{lem:repres} which gives
\begin{align*}
\fish(\om_{\mathrm{new}})=\fish(\om), \quad \mnorm{\om_{\mathrm{new}}}\leq \mnorm{\om},
  \quad \# \supp \om_{\mathrm{new}} \leq \# \supp \om -1.
\qquad\qedhere
\end{align*}}
\end{proof}
\begin{algorithm}
\caption{Support-point removal}
\label{alg:postprocessing}
\begin{algorithmic}
 \STATE 1. Let $\om=\sum_{j=1}^m \lambda_j \delta_{x_j} $ be given.
 %\delta_{x_i}$  $\fish(\delta_{x_1}),\dots,\fish(\delta_{x_m})$ be given.
 \WHILE {$\left\{\fish(\delta_{x_j})\right\}^m_{j=1}$ linearly dependent}
 
 \STATE 2. Find $0\neq\gamma$ with $0=\sum^m_{j=1}\gamma_j \fish(\delta_{x_j})$
 and \changed{$\sum^m_{j=1} \gamma_j \geq 0$ (see section~\ref{sec:impl_sparsification}).}
         
 \STATE 3. Set $\mu=\max_{j}\{\,\gamma_j/\lambda_j \,\} $, $\lambda_{\mathrm{new},j}=\lambda_j-\gamma_j/\mu$.
 
 \STATE 4. Update $\om_{\mathrm{new}}=\sum_{\{\, j\;|\;\lambda_{\mathrm{new},j}>0\,\}} \lambda_{\mathrm{new},j} \delta_{x_j}$.
 \ENDWHILE
\end{algorithmic}
\end{algorithm}
\begin{proposition} \label{prop:sparsifying}
Assume that $\#\supp \om^1 \leq n(n+1)/2$ and let $\om^{k+1}$ be obtained by applying
Algorithm~\ref{alg:postprocessing} to $\om^{k+1/2}$ in each iteration of
Algorithm~\ref{alg:conditionalgradient}. Then the results of
Theorem~\ref{thm:weakconvergencerate} hold. Furthermore we obtain $\#\supp \om^{k}\leq n(n+1)/2$
for all $k \in \N$ and consequently $\#\supp \optomb\leq n(n+1)/2$ for every weak*
accumulation point $\optomb$ of $\om^k$.
\end{proposition}  
\begin{proof}
\changed{The statement for the support of $\om^k$ readily follows from an inductive application of
Proposition~\ref{prop:postprocconepoint} by noting that
\(\fish(\delta_{x_j}) \in \operatorname{Sym}(n)\) and $\operatorname{dim}\operatorname{Sym}(n) = n(n+1)/2$.
The sparsity statement for every accumulation point $\optom$ follows then directly from Lemma~\ref{lem:weakclos}.}
\end{proof}
We emphasize that the sparsifying procedure from Algorithm~\ref{alg:postprocessing} can be
readily combined with the previously presented point removal steps in a straightforward
fashion. In practical computations we optimize the coefficients of the Dirac delta functions in the
current support either by \eqref{projectedgradient} or \eqref{PDAP} obtaining an
intermediate iterate $\om^{k+3/4}$. Subsequently we apply
Algorithm~\ref{alg:postprocessing}. Since in both cases, the number of support points
cannot increase, the statements of the last proposition remain true.
%\TODO{KP: remark superfluous, since Multi-PDAP removed from paper. I would suggest to
%  remove it.}
%\begin{remark}  \label{rem:multiplepoints}
%Note that Algorithm~\ref{alg:conditionalgradient} can be easily generalized to allow for
%the insertion of more than one point in every iteration, which may yield an additional
%practical speed up of the method. In detail, the
%results of Theorem~\ref{thm:weakconvergencerate} and Proposition~\ref{prop:sparsifying}
%hold true if the search direction $v^k\in M^+(\Om)$ from Lemma~\ref{lem:minimizeroflin}
%is more generally chosen as
%\begin{align*}
%v^k = \sum_{i=1}^m \lambda_i \delta_{x_i},
%\;\; \left\{x_i\right\}^m_{i=1} \subset \argmin_{x\in \Om}\psi'(\om^k),
%\;\; \mnorm{v^k} = -M_0 \min_{x\in \Om} \psi'(\om^k)/\beta
%\end{align*}
%if $\min_{x\in\Om}\psi'(\om^k)\leq -\beta$. Moreover, in the case that all
%finite-dimensional sub-problems are solved exactly, all results remain valid if we compute
%\(\om^{k+1}\) as the solution of the coefficient minimization problem~\eqref{PDAP} with
%\(\supp\om^{k+1/2}\) replaced by some finite point set \(\mathcal{A}_{k+1/2}\) which
%contains \(\supp\om^k\cup\set{\hat{x}_k}\).
%\end{remark}

\subsection{Computation of the sparsification steps}
\label{sec:impl_sparsification}
It remains to comment on the computational \changed{aspects of the} point removal
steps presented in this section. 
First, we address the approximate
solution of the finite dimensional subproblems.
%Computing the new coefficient vector $\lambda^k$ from \eqref{projectedgradient} requires
%the computation of the pointwise evaluation of $\psi'(\om^{k+1/2})$ at the current support
%points once. In our numerical experiments a suitable step size $\sigma_k$ is found by a
%simple backtracking line search to avoid ascent. Consequently, for each trial step size,
%the $\max$-operator in \eqref{projectedgradient} as well as the objective function is
%evaluated once. 
%This can be done efficiently with cost scaling linearly with the
%current support size $m_k$.
If $\lambda^k$ is determined from \eqref{PDAP}, we have to solve a finite-dimensional
convex optimization problem in every iteration. Since the most common choices for the
optimal design criterion $\Psi$ are twice continuously differentiable, we choose to implement a
semi-smooth Newton method; \changed{see, e.g., \cite{milzarekfilter}}. To benefit from the fast local convergence behavior for this
class of methods we warm-start the algorithm using the coefficient vector $\lambda^{k+1/2}$
of the intermediate iterate $\om^{k+1/2}$. This choice of the starting point often gives a
good initial guess for $\lambda^{k+1}$. However, we note that essentially any
algorithm for smooth convex problems with non-negativity constraints on the optimization
variables can be employed instead.
% In particular, interior point methods
%provide complexity bounds for the solution up to machine precision in terms
%of the support size $m_k$; see, e.g., \cite[Section~11.5]{boyd2004convex}. In light of
%this fact, the computational cost for the point removal steps can be regarded as a
%constant, assuming that $m_k$ is uniformly bounded through the iterations, e.g., by
%employing Algorithm~\ref{alg:postprocessing}.
%However, interior point methods cannot be warm-started in general, which is why we prefer
%semi-smooth Newton methods in practice.

Finally, we consider the application of
Algorithm~\ref{alg:postprocessing}, given a sparse input measure $\om$ with 
$\supp \om = \{x_i\}_{i=1}^m$. Step~1.\ amounts to the computation of the symmetric rank
one matrices ${\{\fish(\delta_{x_i})\}_{i=1}^m\subset\operatorname{NND}(n)}$, which we
identify with vectors $\{\boldsymbol{I}(\delta_{x_i})\}_{i=1}^m \subset \R^{n(n+1)/2}$. Additionally, in
each execution of the loop step~2.\ has to be executed, which requires to compute
a vector $\gamma$ in the kernel of the matrix $\boldsymbol{I}(\omega) \in \R^{n(n+1)/2\times m}$, defined by
\begin{align*}%\label{step2matrix}
 [\boldsymbol{I}(\omega)]_{i,j} = \boldsymbol{I}(\delta_{x_j})_i, \quad
  i=1,\ldots,n(n+1)/2, \; j=1,\ldots,m.
\end{align*}
This can be done efficiently employing either a SVD-decomposition or a rank-revealing
QR-decomposition. \changed{Since \(\gamma\) is only determined up to a scalar
  multiple, it can be chosen with \(\sum_{j=1}^m \gamma_j \geq 0\). }
Furthermore, assuming that Algorithm~\ref{alg:postprocessing} is applied to $\om^{k+1/2}$
for every $k$, this loop will run at most once in each iteration.
\changed{This follows since each iteration starts with a support set such that
  $\boldsymbol{I}(\omega^k)$ is of full rank, and the point insertion step either
  maintains full rank, or adds a
  linearly dependent vector to $\boldsymbol{I}(\omega^{k+1/2})$. In the latter case the
  removal of at least one support point
  yields again full rank in the next iteration.}

\section{Numerical example} \label{sec:Numerics}
%We end this paper with the study of a numerical example. In the following we
%consider the unit square $\bar{\Omega}=\Om=[0,1]^2$ and a sequence $\mathcal{T}_{h_k}$, $k
%\in \set{1,2,\ldots,9}$, of uniform triangulations of $\Om$ with $h_k=\sqrt{2}/2^k$.
\changed{We end this paper with the study of a numerical example. In the following, we
consider the unit square $\bar{\Omega}=\Om=[0,1]^2$ and a family $\{\mathcal{T}_h\}_{h>0}$
of uniform triangulations of $\Om$, where~$h$ denotes the maximal diameter of a cell~$K
\in \mathcal{T}_{h}$. The set of associated grid nodes is called~$\mathcal{N}_h$. Concretely, we
consider a sequence of successively refined grids with $h_k=\sqrt{2}/2^k$,~$k
\in \set{1,2,\ldots,9}$. The state and sensitivity equations, respectively, are
discretized by linear finite elements on~$\mathcal{T}_h$ and the solutions to the discretized sensitivity equations are denoted by~$\{\partial_k S^h[\hat{q}]\}^n_{k=1}$.
Moreover, $M^+(\Om)$ is replaced by positive linear combinations of nodal Dirac delta function
\begin{align*}
M^+_h:=\set{\om_h \in M^+(\Om) | \supp \om_h\subset \mathcal{N}_h} = M^+(\mathcal{N}_h).
\end{align*}
The discrete design problem is now stated as 
\begin{align} \label{def:fulldisc}
\min_{\om \in M^+_h} &\Psi(\fish_h(\om)+\fish_0) + \beta \mnorm{\om}, \\\quad
  \text{where }&\fish_h(\om) = \int_{\Om} \partial S^h[\hat{q}](x) \partial S^h[\hat{q}](x)^\top \de \om(x).\nonumber
\end{align}
A solution~$\optombh \in M^+_h$ to~\eqref{def:fulldisc} is computed by the different variants of Algorithm~\ref{alg:conditionalgradient} where the search for the new position~$\hat{x}^k$ in step 2. is restricted to~$N_h$. For abbreviation we again define the reduced design criterion~$\psi_h(\om)= \Psi(\fish_h(\om))$.
}

\changed{
Our aim in this
section is twofold. First, we want to numerically illustrate the theoretical results. Secondly, we
want to study the practical performance of the proposed algorithms according to various criteria including
the computational time, the evolution of the sparsity pattern throughout the iterations
and the influence of the fineness of the triangulation.
Concretely, we consider the A-optimal design problem,
i.e.~$\Psi(N) = \Tr(N^{-1})$ and the discrete state and the associated
sensitivities \(\partial S^h[\hat{q}]\) are computed for a fixed \(\hat{q}\) once at the beginning. During the
execution of the different variants of Algorithms~\ref{alg:conditionalgradient} no additional PDEs need to be solved. Moreover, the gradient of
the reduced cost functional is given by
\[
\left[\psi'_h(\omega)\right](x)
= - \Tr(\fish_h(\omega)^{-1}\fish_h(\delta_x)\fish_h(\omega)^{-1})
= - \norm{\fish_h(\omega)^{-1}\partial S^h[\hat{q}](x)}^2_{\R^n}
\quad \forall x \in \Omega,
\]
which relates the pointwise value of the gradient directly to the corresponding
sensitivity vector \(\partial S^h[\hat{q}](x) \in \R^n\). A corresponding
computation on the discrete level allows for an efficient implementation based
on a single Cholesky-decomposition of \(\fish^h(\omega)\) in each iteration. Moreover an expression for the Hessian-vector-product
\(\left[\psi_h''(\omega)(\delta\omega)\right](x)\) for \(\delta\omega \in M(\Omega)\) can be
derived by differentiating the above expression.
%In both examples, the assumptions on the continuous and discrete state equation, see Assumption \ref{ass:Existenceofstate} and Assumption \ref{ass:discrete}, respectively, can be easily verified.
\begin{remark}
It is possible to show that every solution~$\optombh \in M^+(\Om)\cap M_h$ to~\eqref{def:fulldisc} is also a mininimizer of the semi-discrete problem
\begin{align} \label{def:semidisc}
\min_{\om \in M^+(\Om)} \psi_h(\om) + \beta \mnorm{\om}
\end{align}
where the space of possible designs is not discretized. This corresponds to the variational discretization paradigm; cf~\cite{casas2012}. In particular, proceeding as for the fully continuous problem, a measure~$\optombh \in M^+(\Om)$ is an optimal solution to~\eqref{def:fulldisc} if and only if
\begin{align*}
-\psi'_h(\optombh) \leq \beta, \quad
\operatorname{supp} \optombh \subset \left\{ x\in \Om\;|\;\psi'_h(\optombh)(x)=-\beta \right\}. 
\end{align*}
Since the main focus of the present paper lies on the description of the sparse sensor placement problem and its efficient solution, we postpone a detailed discussion of the discretization to a follow-up paper. 
\end{remark}}
\subsection{Estimation of convection and diffusion parameters} \label{ex:3param}
As an example for the state equation~\eqref{stateequation}, we take a convection-diffusion
process where for a given $q \in Q_{ad} = \{\,q \in \R^3 \;|\; 0.25 \leq q_1 \leq 5\,\}$ the associated
state $y = S[q] \in H^1_0(\Omega)\cap C(\Om)$ is the unique solution to
\begin{align} \label{3param}
a(q,y)(\varphi)
  = \int_{\Omega} \left\lbrack q_1  \nabla y \cdot \nabla \varphi
 + q_2 \varphi \frac{\partial y}{\partial x_1}
 + q_3 \varphi \frac{\partial y}{\partial x_2} \right\rbrack \de x
  = \int_{\Omega} f \varphi \de x,
\end{align}
for all \(\varphi \in \Heins\).
The forcing term \(f\) is chosen as $\operatorname{exp}(3(x^2_1+x_2^3))$. This corresponds to the linear elliptic equation
\begin{align*}
-q_1 \Lap y +\left(\begin{array}{c}q_2\\q_3\end{array}\right)\cdot \nabla y=f \quad \text{in}~ \Omega,
\end{align*}
together with homogeneous Dirichlet boundary conditions on $\partial \Omega$. Here, the
parameter $q$ contains the scalar diffusion and convection coefficients of the
elliptic operator. As a~priori guess for the parameter we choose
$\hat{q}=(3,0.5,0.25)^\top$. Note that while~\eqref{3param} is a linear
equation, the state $y \in \Heins \cap C(\Om)$ depends non-linearly but differentiably on
$q$. For each $k \in \set{1,2,3}$ the sensitivity $\delta{y}_k = \partial_k S[\hat{q}]\in \Heins \cap C(\Om)$
can be computed from~\eqref{sensitivity}.
Due to the tri-linearity of the form \(a(\cdot,\cdot)(\cdot)\) it fulfills
\begin{align*}
a(\hat{q}, \delta{y}_k)(\varphi) = -a(\textbf{e}_k,\hat{y})(\varphi)
\quad\forall\varphi\in \Heins,
\end{align*}
where \(\hat{y} = S[\hat{q}]\) and $\textbf{e}_k \in \R^3$ denotes the $k$-th canonical unit vector.
%\subsubsection{Verification of the main assumptions}
%First, we will briefly comment on the major assumptions  on the mapping $q \mapsto y_q$ to ensure wellposedness of \eqref{def:designprop} \eqref{discretizedcost}, \eqref{def:regularizedproblem} and \eqref{discregcost}. The following Proposition can be directly derived by applying the Lax-Milgram Lemma, elliptic regularity results from \cite[Chapter 4]{gilbarg2001elliptic}, \cite[Proposition 2.2.1]{vexler04} and \eqref{sensitivity}.
%\begin{proposition}
%There holds:
%\begin{itemize}
%\item[1.] For every $q\in Q_{ad}$ there exists an unique solutions $y_q \in \Heins \cap \Hzwo $ to \eqref{diffequationweak}. Moreover there holds $y_q\in C(\bar{\Om})\cap C^{2, \nu}(\Om)$.
%\item[2.] The mapping $S:~Q_{ad}\rightarrow C_0(\Om) $ defined by $S[q]=y_q$ is at least continuously Fr\'echet differentiable. Its partial derivatives $\{\partial_i S[q]\}_{i=1}^3 \subset H^1_0(\Om) \cap \Hzwo \cap C(\bar{\Om}) \cap C^{2,\nu}(\Om)$ fulfill the following sensitivity equations
%\begin{align*}
%a(q)\left(\partial_1 S[q] , \varphi\right)&= \int_{\Om}\Delta S[q] \varphi \de x \quad \forall \varphi \in H^1_0(\Om) 
%\end{align*}
%\end{itemize}
%\end{proposition}
\subsubsection{First order optimality condition} \label{subsec:firstorder}
\changed{In this section we numerically illustrate the
first-order necessary and sufficient optimality conditions from Proposition~\ref{prop:equivalence}. Therefore
we compute an A-optimal design for Example~1 on grid level nine $\mathcal{T}_{h_9}$ for $\beta=1$ and
$\fish_0=0$. For the computation we use Algorithm~\ref{alg:conditionalgradient} (together
with Algorithm~\ref{alg:postprocessing} and a full resolution of the arising finite-dimensional
subproblems), until the residual is zero (up to machine precision).
We obtain a discrete optimal design $\optombh$ in $M^+(\Om)\cap M_h$
with five support points.
By closer inspection we observe that two of the computed support points are
located in adjacent nodes of the triangulation. For a better visualization of the computed
result, the corresponding Dirac delta functions are replaced by a single one placed at the
center of mass. The coefficient of this new Dirac delta function is given by the combined
mass of the original ones; see Figure~\ref{fig:optdesign}.}
\begin{figure}[htb]
\centering
\begin{subfigure}[t]{.45\linewidth}
\centering
\scalebox{.8}
{% This file was created by matlab2tikz.
\begin{tikzpicture}

\begin{axis}[%
width=2.2in,
height=1.8in,
scale only axis,
xmin=0,
xmax=1,
xlabel style={font=\color{white!15!black}},
xlabel={$x_\text{1}$},
ymin=0,
ymax=1,
ylabel style={font=\color{white!15!black},at={(axis description cs:0.07,0.5)}},
ylabel={$x_\text{2}$},
axis background/.style={fill=white},
%axis x line*=bottom,
%axis y line*=left
]
\addplot [color=blue, draw=none, mark=x, mark options={solid, blue}, forget plot]
  table[row sep=crcr]{%
0.320917408926917	0.686756692853834\\
};
\node[above, align=center]
at (axis cs:0.321,0.687) {219.068};
\addplot [color=blue, draw=none, mark=x, mark options={solid, blue}, forget plot]
  table[row sep=crcr]{%
0.84765625	0.890625\\
};
\node[above, align=center]
at (axis cs:0.848,0.891) {115.441};
\addplot [color=blue, draw=none, mark=x, mark options={solid, blue}, forget plot]
  table[row sep=crcr]{%
0.841796875	0.501953125\\
};
\node[above, align=center]
at (axis cs:0.842,0.502) {56.758};
\addplot [color=blue, draw=none, mark=x, mark options={solid, blue}, forget plot]
  table[row sep=crcr]{%
0.646484375	0.298828125\\
};
\node[above, align=center]
at (axis cs:0.646,0.299) {198.667};
\end{axis}
\end{tikzpicture}%
}\\
\caption{Optimal design $\optombh$.}\label{fig:optdesign}
\end{subfigure}
\begin{subfigure}[t]{.45\linewidth}
\centering
\scalebox{.8}
{\input{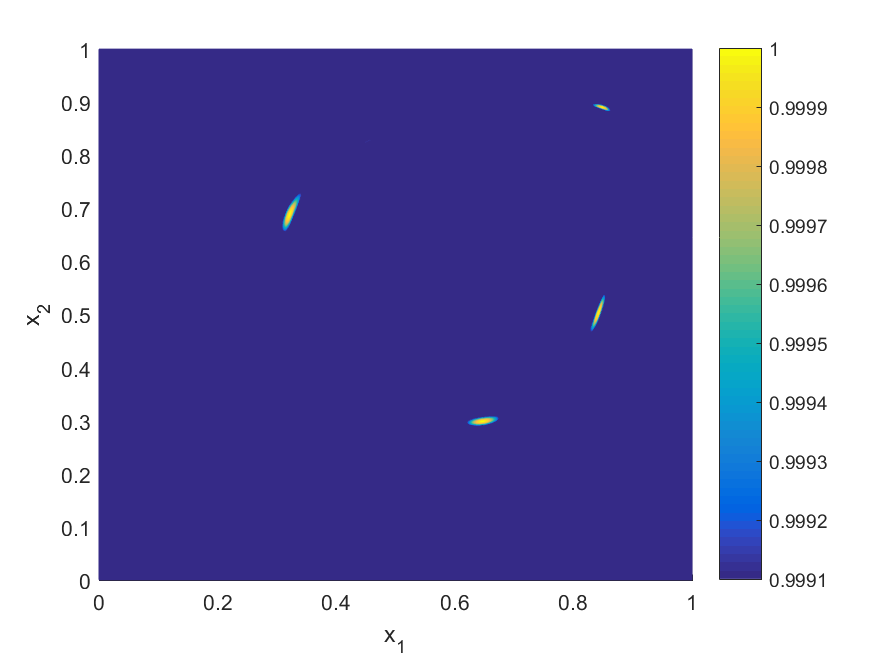}}\\
\caption{Isolines of $-\psi_h(\optombh)$.}\label{fig:neggrad}
\end{subfigure}
\caption{Optimal design and isolines of the gradient.}\label{fig:effectiveness1}
\end{figure}
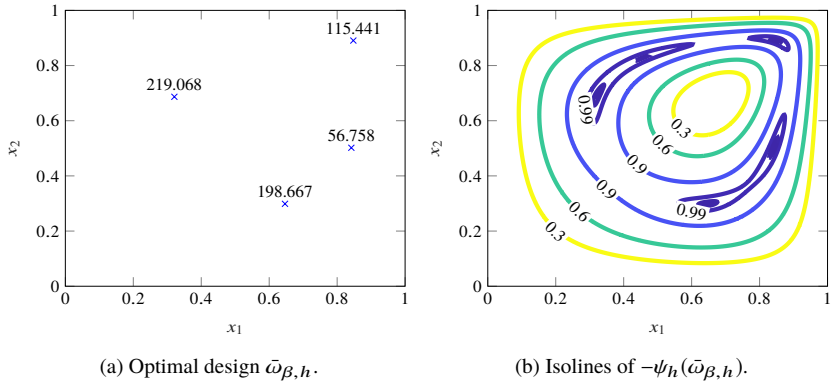
\changed{
Alongside we plot the isolines of the nodal interpolant of $-\psi_h '(\optombh)$. Note that the values of $-\psi_h '(\optombh)$ in~$N_h$ are bounded
from above by the cost parameter $\beta = 1$ and the support points of $\optombh$ align
themselves with those points in which this upper bound is achieved; see Figure~\ref{fig:neggrad}.
}
\subsubsection{Confidence domains of the optimal estimator} \label{subsec:confidence}
Given the optimal design $\bar{\omega}_h$ from Figure~\ref{fig:optdesign}, and $K>0$ we
note that the measure $\bar{\om}^K_h=(K/\mnorm{\optombh}) \optombh$ is an optimal solution
to
\begin{align*}
 \min_{\om \in M^+(\Om)} \Tr(\fish_h(\om_h)^{-1}) \quad\text{subject to~}\mnorm{\om_h} \leq K,
\end{align*}
since the A-optimal design criterion is positive homogeneous;
see Proposition~\ref{prop:equivalence}. In this section we compute the linearised confidence
domains~\eqref{linearizedconfidence} of the least-squares estimator $\tilde{q}$ from
\eqref{Leastsquareestimator} corresponding to $\bar{\om}^K_h$ for $K = 3 \cdot 10^{4}$.
\begin{figure}[htb]
\centering
\begin{subfigure}[t]{.45\linewidth}
\centering
\scalebox{.8}
{% This file was created by matlab2tikz.
\begin{tikzpicture}

\begin{axis}[%
width=2.2in,
height=1.8in,
scale only axis,
xmin=0,
xmax=1,
xlabel style={font=\color{white!15!black}},
xlabel={$x_\text{1}$},
ymin=0,
ymax=1,
ylabel style={font=\color{white!15!black},at={(axis description cs:0.07,0.5)}},
ylabel={$x_\text{2}$},
axis background/.style={fill=white},
%axis x line*=bottom,
%axis y line*=left
]
\addplot [color=blue, draw=none, mark=x, mark options={solid, blue}, forget plot]
  table[row sep=crcr]{%
0.25	0.25\\
};
\node[above, align=center]
at (axis cs:0.25,0.25) {10000};
\addplot [color=blue, draw=none, mark=x, mark options={solid, blue}, forget plot]
  table[row sep=crcr]{%
0.25	0.75\\
};
\node[above, align=center]
at (axis cs:0.25,0.75) {10000};
\addplot [color=blue, draw=none, mark=x, mark options={solid, blue}, forget plot]
  table[row sep=crcr]{%
0.75	0.5\\
};
\node[above, align=center]
at (axis cs:0.75,0.5) {10000};
\end{axis}
\end{tikzpicture}%
}\\
\end{subfigure}
\begin{subfigure}[t]{.45\linewidth}
\centering
\scalebox{.8}
{% This file was created by matlab2tikz.
\begin{tikzpicture}

\begin{axis}[%
width=2.2in,
height=1.8in,
scale only axis,
xmin=0,
xmax=1,
xlabel style={font=\color{white!15!black}},
xlabel={$x_\text{1}$},
ymin=0,
ymax=1,
ylabel style={font=\color{white!15!black},at={(axis description cs:0.07,0.5)}},
ylabel={$x_\text{2}$},
axis background/.style={fill=white},
%axis x line*=bottom,
%axis y line*=left
]
\addplot [color=blue, draw=none, mark=x, mark options={solid, blue}, forget plot]
  table[row sep=crcr]{%
0.641940255749232	0.298190255749232\\
};
\node[above, align=center]
at (axis cs:0.642,0.298) {16089.5};
\addplot [color=blue, draw=none, mark=x, mark options={solid, blue}, forget plot]
  table[row sep=crcr]{%
0.84375	0.892578125\\
};
\node[above, align=center]
at (axis cs:0.844,0.893) {5245.1};
\addplot [color=blue, draw=none, mark=x, mark options={solid, blue}, forget plot]
  table[row sep=crcr]{%
0.322265625	0.689453125\\
};
\node[above, align=center]
at (axis cs:0.322,0.689) {1337.4};
\addplot [color=blue, draw=none, mark=x, mark options={solid, blue}, forget plot]
  table[row sep=crcr]{%
0.4609375	0.830078125\\
};
\node[above, align=center]
at (axis cs:0.461,0.83) {7327.937};
\end{axis}
\end{tikzpicture}%
}\\
\end{subfigure}
\caption{Reference measures $\om_1$ (left) and $\bar{\om}_h^{K,W}$ (right).}
\label{fig:randdesign}
\end{figure}
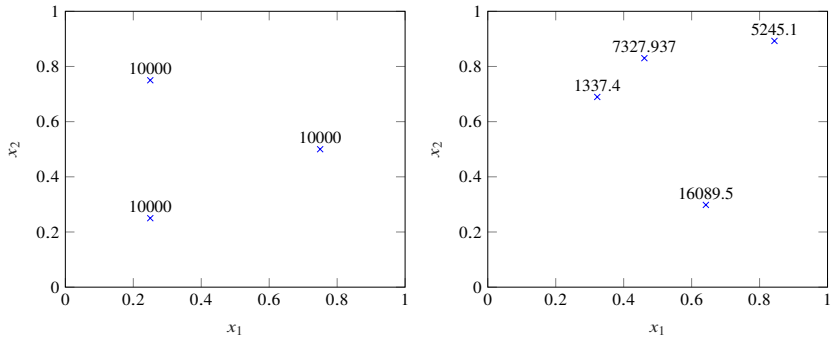

Note that, given a sparse design measure $\om$, and the associated linearised estimator
$\tilde{q}_{\mathrm{lin}}=(\tilde{q}^1_{\mathrm{lin}},\tilde{q}^2_{\mathrm{lin}},\tilde{q}^3_{\mathrm{lin}})^\top$, see
\eqref{explicitlinearized}, there holds
$\mathrm{Cov}[\tilde{q}_{\mathrm{lin}},\tilde{q}_{\mathrm{lin}}]=\fish_h(\om)^{-1}$; see the discussion in
Section~\ref{sec:fromparamtoopt}. Consequently we have
\begin{align*}
\fish_h(\om)^{-1}_{kk}=\mathrm{Var}[\tilde{q}^k_{\mathrm{lin}}],~k\in \set{1,2,3}
\quad\text{and}\quad \Tr(\fish_h(\om)^{-1})=\sum_{k=1}^3 \mathrm{Var}[\tilde{q}^k_{\mathrm{lin}}].
\end{align*}
As a comparison, we also consider the estimators corresponding to two
reference designs of the same norm. The  first measure $\om^1$ is chosen as a linear
combination of three Dirac delta functions with equal coefficients while the second measure
$\bar{\om}^{K,W}_{h}$ is a solution to
\begin{align} \label{ex:trace}
 \min_{\om \in M^+(\Om)} \Tr(W\fish_h(\om_h)^{-1}W)
  \quad\text{subject to~} \mnorm{\om_h} \leq K,
\end{align}
where $W=\operatorname{diag}(1,1,4)$, i.e. we place more weight on  the variance for the estimation of~$q_3$.  
The designs $\om_1$ and $\bar{\om}^{K,W}_{h}$ are depicted in Figure~\ref{fig:randdesign}.
%\TODO{KP: Weighting corresponds to rescaling of the coefficients.
%  Maybe, add physical motivation.}

For a better visualization we plot the $50\%$-linearised confidence domains of the
obtained estimators for the two dimensional parameter vectors $(q_1,q_2)^\top$, $(q_2,q_3)^\top$,
and $(q_3,q_1)^\top$ in Figure~\ref{fig:confidencedomains}. Additionally, for each design
we report $\Tr(\fish_h(\om)^{-1})$ as well as the diagonal entries of
$\fish_h(\om)^{-1}$ in Table~\ref{tab:diagonalentries}.
\begin{table}[tbhp]
\caption{Trace and diagonal entries of $\fish_h (\om)^{-1}$}
\label{tab:diagonalentries}
\centering
\begin{tabular}{ccccc} \toprule
$\om$ & $\fish_h (\om)_{11}^{-1}$& $\fish_h (\om)_{22}^{-1}$ & $\fish_h (\om)_{33}^{-1}$& $\Tr(\fish_h (\om)^{-1})$\\ \midrule
$\bar{\om}^K_h$& 0.019 & 5.627 & 5.955 & 11.601  \\
$\om_1$ & 0.091 & 7.388 & 20.678 & 28.157 \\
$\bar{\om}^{K,W}_{h}$ & 0.023 & 14.12 & 3.831& 17.974\\
\bottomrule
\end{tabular}
\end{table}
As expected, since $\optombh$ is chosen by the A-optimal design criterion, we observe that
\begin{align} \label{ex:weighttrace}
 \Tr(\fish_h(\bar{\om}^K_h)^{-1})
  \leq \Tr(\fish_h(\bar{\om}^{K,W}_{h})^{-1})
  \leq \Tr(\fish_h(\om_1)^{-1}).
\end{align}
Moreover we note that $\fish_h(\bar{\om}^K_h)^{-1}_{kk}<\fish_h(\om_1)^{-1}_{kk}$ for all
$k$, i.e. the optimal estimator estimates all unknown parameters with a smaller variance
than the estimator associated to the reference design $\om_1$. As a consequence, the
linearised confidence domains of the optimal estimator are contained in those of the one
corresponding to $\om_1$; see Figure \ref{fig:confidencedomains}.
In contrast, considering $\om_2$, we have
$\fish_h(\bar{\om}^{K,W}_{h})^{-1}_{33}<\fish_h(\bar{\om}^K_h)^{-1}_{33}$ and
$\fish_h(\bar{\om}^K_h)^{-1}_{kk}<\fish_h(\fish_h(\bar{\om}^{K,W}_{h})^{-1}_{kk}$ for $k=1,2$, i.e.\ the third
parameter is estimated more accurately by choosing the measurement locations and weights
according to $\om_2$ while the variance for the estimation of the other parameters is
larger. This is a consequence of the different weighting of the matrix entries in
\eqref{ex:weighttrace}.
On the one hand, the obtained results show the efficiency of an optimally chosen
measurement design at least for the linearised model. On the other hand, they also
highlight that the properties of the obtained optimal estimators crucially depend on the
choice of the optimal design criterion~$\Psi$.

\begin{figure}
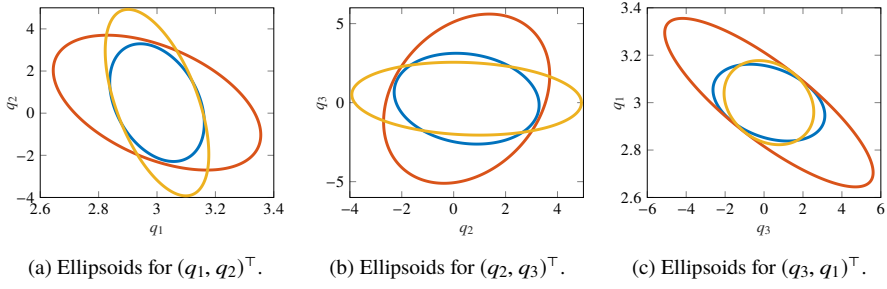

\centering
\begin{subfigure}[t]{.32\linewidth}
\centering
\scalebox{.55}
{\large\input{figures/12projection}}\\
\caption{Ellipsoids for $(q_1,q_2)^\top$.}
\end{subfigure}
\begin{subfigure}[t]{.32\linewidth}
\centering
\scalebox{.55}
{\large\input{figures/23projection}}\\
\caption{Ellipsoids for $(q_2,q_3)^\top$.}
\end{subfigure}
\begin{subfigure}[t]{0.32\linewidth}
\centering
\scalebox{.55}
{\large\input{figures/13projection}}\\
\caption{Ellipsoids for $(q_3,q_1)^\top$.}
\end{subfigure}
\caption{Confidence ellipsoids for the estimators associated to $\bar{\om}^K_h$ (blue), $\om_1$ (red) and $\bar{\om}^{K,W}_{h}$ (yellow).}\label{fig:confidencedomains}
\end{figure}

\subsubsection{Comparison of point insertion algorithms}
\label{subsubsec:pointcomparison}
In this section we investigate the performance
of the successive point insertion algorithm presented in Section~\ref{subsec:condgrad}. We
consider the same setup as in Section \ref{subsec:firstorder}, i.e.\ we solve the A-optimal
design problem for Example~1 on grid level nine with $\beta=1$ and $\fish_0=0$.
The step size parameters $\alpha$ and $\gamma$ in~\eqref{Def:Armijo} are both chosen as
$1/2$ throughout the experiments and the
iteration is terminated if either $\Phi(\om^k) \leq 10^{-9}$ or if the iteration number
$k$ exceeds $2 \cdot 10^4$. The aim of this section is to confirm the theoretical
convergence results for Algorithm~\ref{alg:conditionalgradient} and to demonstrate the
necessity of additional point removal steps. Additionally we want to highlight the
differences between the three presented choices of the new coefficient vector
$\lambda^{k+1}$ concerning the sparsity of the iterates and the practically achieved
acceleration of the convergence. Specifically, we consider the following
implementations of step~4.\ in Algorithm~\ref{alg:conditionalgradient}:
\begin{description}
\item[\namedlabel{alg:GCG}{GCG}]
  In the straightforward implementation of the GCG algorithm we set $\lambda^{k+1} =
  \lambda^{k+1/2}$, i.e.\ only steps~1.\ to 4. are performed.
\item[\namedlabel{alg:SPINAT}{SPINAT}]
  Here, we employ the procedure suggested in~\cite{bredies2013inverse}, termed
  ``Sequential Point Insertion and Thresholding''.
  In step~5., $\lambda^{k+1}$ is determined from a proximal gradient
  iteration~\eqref{projectedgradient}. The step size is chosen as
  $\sigma_k = (1/2)^{n}\sigma_{0,k}$, where $\sigma_{0,k}>0$ 
  for the smallest $n \in \N$ giving
  $F(\om(\lambda^{k+1}(\sigma_k)))\leq F(\om(\lambda^{k+1/2}))$.
  In particular, given $\om^{k+1/2} = \sum_{i} \lambda^{k+1/2}_i\delta_{x_i}$, we choose $\sigma_{0,k}$ as
  \begin{align*}
    \sigma_{0,k} =
    \max \left\{100, - 2 \min_{i} \left\{ \frac{\lambda_i}{-\psi'(\om^{k+1/2})(x_i)-\beta}\right\}\right \}.
  \end{align*}
  Note that by this choice of $\sigma_{0,k}$, the coefficients of all points
  $x \in \supp \om^{k+1/2}$ with $-\psi'(\om^{k+1/2})(x)< \beta$ are set to zero in the first trial
  step (i.e.\ for $n = 0$).
\item[\namedlabel{alg:PDAP}{PDAP}]
  Here, the coefficient vector $\lambda^{k+1}$ is chosen as in \eqref{PDAP} by solving the
  finite dimensional sub-problem \eqref{def:subproblem} up to machine precision in each
  iteration. For the solution we use a semi-smooth Newton method with a globalization
  strategy based on a backtracking line-search. The convergence criterion for
  the solution of the sub-problems is based on the norm of the Newton-residual.
  Since, this method can be interpreted as a method operating on a set of active points
  \(\mathcal{A}_k = \supp \omega^k\) (see section~\ref{sec:acceleration}), we reference it by
  the name: ``Primal-Dual Active Point''.
\end{description}
All three versions of the algorithm are also considered with an application of the
sparsification step Algorithm~\ref{alg:postprocessing} applied at the end of each
iteration of Algorithm~\ref{alg:conditionalgradient}. In the
following this will be denoted by an additional ``+PP''.

\begin{figure}[htb]
\centering
\begin{subfigure}[t]{.45\linewidth}
\centering
\scalebox{.8}
{\input{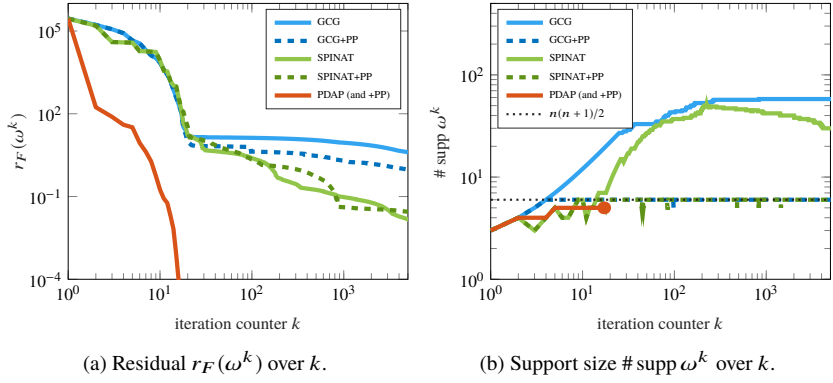}}\\
\caption{Residual $r_F(\om^k)$ over $k$.}\label{fig:residual}
\end{subfigure}
\begin{subfigure}[t]{.45\linewidth}
\centering
\scalebox{.8}
{% This file was created by matlab2tikz.
\definecolor{mycolor1}{rgb}{0.20000,0.64700,0.94100}%
\definecolor{mycolor2}{rgb}{0.00000,0.44700,0.74100}%
\definecolor{mycolor3}{rgb}{0.56600,0.77400,0.28800}%
\definecolor{mycolor4}{rgb}{0.36600,0.57400,0.08800}%
\definecolor{mycolor5}{rgb}{1.00000,0.52500,0.28800}%
\definecolor{mycolor6}{rgb}{0.85000,0.32500,0.09800}%
\begin{tikzpicture}

\begin{axis}[%
width=2.2in,
height=1.8in,
scale only axis,
xmode=log,
xmin=1,
xmax=5000,
xminorticks=true,
xlabel style={font=\color{white!15!black}},
xlabel={iteration counter $k$},
ymode=log,
ymin=1,
ymax=500,
yminorticks=true,
ylabel style={font=\color{white!15!black},at={(axis description cs:0.07,0.5)}},
ylabel={$\text{\# supp }\omega{}^{k}$},
axis background/.style={fill=white},
legend style={at={(0.02,0.98)}, anchor=north west, legend cell align=left, align=left,
  draw=white!15!black, font=\tiny}
]

\addplot [color=mycolor1, line width=2pt]
  table[row sep=crcr]{%
1	3\\
2	4\\
3	5\\
4	6\\
5	7\\
6	8\\
7	9.00000000000001\\
8	10\\
10	12\\
12	14\\
15	17\\
19	21\\
24	26\\
25	27\\
27	27\\
28	28\\
30	28\\
31	29\\
33	29\\
37	33\\
61	33\\
62.0000000000001	34\\
63.0000000000001	34\\
64.0000000000001	35\\
72.0000000000001	35\\
73	36\\
74	36\\
76	38\\
77.0000000000001	38\\
78.0000000000001	39\\
80	39\\
81	40\\
85.9999999999999	40\\
88.0000000000001	42\\
89.0000000000001	42\\
90.0000000000001	43.0000000000001\\
102	43.0000000000001\\
103	44\\
120	44\\
121	45\\
132	45\\
134	47.0000000000001\\
135	47.0000000000001\\
136	48\\
148	48\\
149	49\\
150	49\\
151	50\\
161	50\\
162	51\\
165	51\\
166	52\\
167	52\\
168	53\\
236	53\\
237	54.0000000000001\\
257	54.0000000000001\\
258	55\\
261	55\\
262	56\\
270	56\\
271	57\\
830	57\\
832.000000000001	58.0000000000001\\
20001	58.0000000000001\\
};
\addlegendentry{GCG}

\addplot [color=mycolor2, line width=2pt, dashed]
  table[row sep=crcr]{%
1	3\\
2	4\\
3	5\\
4	6\\
97.0000000000001	6\\
98.0000000000001	5\\
98.9999999999999	5\\
100	6\\
20001	6\\
};
\addlegendentry{GCG+PP}

\addplot [color=mycolor3, line width=2pt]
  table[row sep=crcr]{%
1	3\\
2	4\\
3	3\\
5	5\\
6	4\\
7	4\\
8	5\\
9.00000000000001	6\\
10	6\\
11	4\\
12	5\\
13	5\\
14	6\\
15	7\\
18	7\\
19	8\\
20	9.00000000000001\\
21	10\\
22	11\\
24	13\\
26	15\\
27	15\\
28	14\\
30	16\\
31	17\\
32	17\\
34	19\\
37	19\\
40	22\\
41	22\\
42	23\\
43.0000000000001	23\\
45	25\\
47.0000000000001	25\\
48	26\\
49	26\\
50	27\\
52	27\\
53	28\\
54.0000000000001	28\\
55	29\\
58.0000000000001	29\\
60	31\\
62.0000000000001	31\\
63.0000000000001	32\\
70	32\\
72.0000000000001	34\\
73	34\\
74	35\\
92.0000000000001	35\\
93.0000000000001	36\\
95	36\\
96.0000000000001	37\\
125	37\\
126	38\\
150	38\\
151	39\\
154	39\\
155	40\\
168	40\\
169	41\\
170	41\\
171	42\\
180	42\\
181	43.0000000000001\\
190	43.0000000000001\\
191	42\\
193	42\\
196	45\\
198	45\\
199	46\\
200	46\\
201	47.0000000000001\\
204	47.0000000000001\\
205	48\\
208	48\\
209	49\\
211	49\\
212	50\\
223	50\\
224	51\\
225	50\\
226	50\\
227	49\\
236	49\\
237	48\\
248	48\\
249	49\\
287	49\\
288	48\\
342	48\\
343	47.0000000000001\\
466	47.0000000000001\\
467.000000000001	46\\
603	46\\
605.000000000001	45\\
690.000000000001	45\\
692.000000000001	44\\
708	44\\
710	45\\
810	45\\
812	44\\
902.000000000001	44\\
904.000000000001	43.0000000000001\\
933.000000000001	43.0000000000001\\
935.000000000001	42\\
1406	42\\
1409	41\\
1422	41\\
1425	40\\
2001	40\\
2005	39\\
2015	39\\
2019	38\\
2310	38\\
2315	37\\
2880	37\\
2886	36\\
3269.00000000001	36\\
3275	35\\
3349	35\\
3355	34\\
3825	34\\
3832	33\\
3944	33\\
3951	32\\
4054.00000000001	32\\
4062	31\\
4163	31\\
4171	30\\
5360.00000000001	30\\
5370	29\\
5571	29\\
5581.00000000001	28\\
5783	28\\
5794	27\\
6320.00000000001	27\\
6332	26\\
7280.00000000001	26\\
7293	25\\
7765	25\\
7779	24\\
10442	24\\
10461	23\\
11281	23\\
11301	22\\
11581	22\\
11602	21\\
12094	21\\
12116	20\\
13666	20\\
13690	19\\
13777	19\\
13802	18\\
14701	18\\
14727	17\\
16716	17\\
16746	16\\
20001	16\\
};
\addlegendentry{SPINAT}

\addplot [color=mycolor4, line width=2pt, dashed]
  table[row sep=crcr]{%
1	3\\
2	4\\
3	3\\
5	5\\
6	4\\
7	4\\
8	5\\
9.00000000000001	6\\
10	6\\
11	4\\
12	5\\
13	5\\
14	6\\
17	6\\
18	5\\
19	6\\
43.0000000000001	6\\
44	5\\
45	6\\
46	5\\
47.0000000000001	6\\
83	6\\
84.0000000000001	5\\
85	5\\
85.9999999999999	6\\
480	6\\
481	5\\
482	6\\
828.000000000001	6\\
829	5\\
831	6\\
833.000000000001	5\\
835	6\\
1445	6\\
1446	5\\
1449	6\\
10109	6\\
10110	5\\
10128	6\\
20001	6\\
};
\addlegendentry{SPINAT+PP}

\addplot [color=mycolor6, line width=2pt]
  table[row sep=crcr]{%
1	3\\
2	4\\
4	4\\
5	5\\
17	5\\
};
\addlegendentry{PDAP (and +PP)}

% \addplot [color=mycolor6, line width=2pt, dashed]
%   table[row sep=crcr]{%
% 1	3\\
% 2	4\\
% 4	4\\
% 5	5\\
% 17	5\\
% };
% \addlegendentry{PDAP+PP}

\addplot [color=white!15!black, line width=1pt, dotted]
  table[row sep=crcr]{%
1	6\\
20001   6\\
};
\addlegendentry{$n(n+1)/2$}

\addplot [color=mycolor6, line width=6pt, dotted, line cap=round]
  table[row sep=crcr]{%
17      5\\
17      5\\
};

\end{axis}
\end{tikzpicture}%
}\\
\caption{Support size $\#\supp \om^k$ over $k$.} \label{fig:support}
\end{subfigure}
\caption{Residual and support size plotted over iteration number $k$. \changed{The
  results for PDAP and PDAP+PP are identical; a line with a dot denotes termination of the
  algorithm within machine tolerance.}}
\end{figure}
In Figure~\ref{fig:residual} we plot the residual $r_F(\om^k)$ for all considered algorithms over the
iteration counter $k$. For \ref{alg:GCG} as well as \ref{alg:SPINAT} we observe a rapid
decay of the computed residuals in the first few iterations. However, asymptotically both
admit a sub-linear convergence rate, suggesting that the convergence result
derived in Theorem~\ref{thm:weakconvergencerate} is sharp in this instance. The additional
application of Algorithm~\ref{alg:postprocessing} has no significant impact on the
convergence behavior. We additionally note that both \ref{alg:GCG} and \ref{alg:SPINAT}
terminate only since the maximum number of
iterations is exceeded while the computed residuals $r_F(\om^k)$ and thus also the primal-dual gap
$\Phi(\om^k)$ remain above $10^{-3}$. In contrast, \ref{alg:PDAP} terminates after few
iterations within the tolerance. The results clearly indicate a
better convergence rate than the one derived in Theorem~\ref{thm:weakconvergencerate}.

Next, we study the influence of the different point removal steps on the sparsity pattern
of the obtained iterates in Figure~\ref{fig:support}. For \ref{alg:GCG} we notice that the number of
support points increases monotonically up to approximately $60$. This suggests a strong
clusterization of the intermediate support points around those of $\optombh$ which is
possibly caused by the small curvature of $-\psi_h '(\optombh)$ (see
Figure~\ref{fig:neggrad}) in the vicinity of its global maxima.
\changed{A similar behavior can be
observed for the iterates obtained through \ref{alg:SPINAT}. However, compared to
\ref{alg:GCG} the support size for \ref{alg:SPINAT} grows slower due
to the additional projected gradient step in every iteration.}
%Additionally, after reaching
%a threshold at approximately $k = 110$, the support size decreases monotonically in the remaining iterations.
Concerning the application of Algorithm~\ref{alg:postprocessing}, we observe that the
support remains bounded \changed{for all implementations with ``+PP'' by $n(n+1)/2
= 6$ as predicted by
Proposition~\ref{prop:sparsifying}}. We note that this upper bound is achieved in almost
all but the first few iterations for \ref{alg:GCG} and \ref{alg:SPINAT}.
In contrast, \ref{alg:PDAP} yields iterates comprising less than six support points independently
of the additional post-processing. A closer inspection reveals that the loop in
Algorithm~\ref{alg:postprocessing} is not carried out in any iteration, i.e.\ the sparsity
of the iterates is fully provided by the exact solution of the finite-dimensional sub-problems.

\begin{figure}[htb]
\centering
\begin{subfigure}[t]{.45\linewidth}
\centering
\scalebox{.8}
{\input{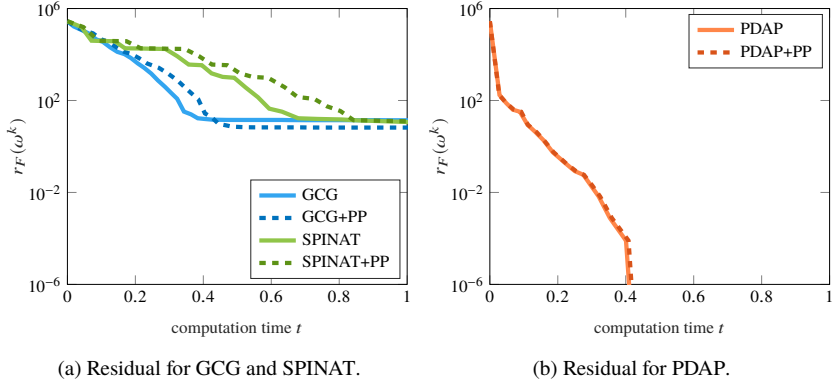}}\\
\caption{Residual for \ref{alg:GCG} and \ref{alg:SPINAT}.}\label{fig:timeresidualGCG}
\end{subfigure}
\begin{subfigure}[t]{.45\linewidth}
\centering
\scalebox{.8}
{% This file was created by matlab2tikz.
\definecolor{mycolor1}{rgb}{0.00000,0.44700,0.74100}%
\definecolor{mycolor2}{rgb}{0.85000,0.32500,0.09800}%
\begin{tikzpicture}

\begin{axis}[%
width=4.521in,
height=3.566in,
at={(0.758in,0.481in)},
scale only axis,
xmin=0,
xmax=1,
xlabel style={font=\color{white!15!black}},
xlabel={t},
ymode=log,
ymin=1e-06,
ymax=1000000,
yminorticks=true,
ylabel style={font=\color{white!15!black}},
ylabel={$\text{r}_\text{F}\text{(}\omega{}_\text{k}\text{)}$},
axis background/.style={fill=white},
legend style={legend cell align=left, align=left, draw=white!15!black}
]
\addplot [color=mycolor1]
  table[row sep=crcr]{%
0	280396.460334535\\
0.028488956359982	172.402160364329\\
0.0481532639141312	77.4077916686927\\
0.0677213538882593	39.9485675580459\\
0.0885163249160769	32.2455507986074\\
0.111061172999209	8.76060171108588\\
0.134416814556651	4.19388073633172\\
0.157672389830871	1.88095329659131\\
0.18517522290406	0.563801182415091\\
0.224296008039513	0.176466474949621\\
0.25026662960508	0.0843199148378062\\
0.275336226987962	0.0583570323374261\\
0.297596271034232	0.0198952803254997\\
0.32052085783318	0.00661927371697857\\
0.353179670294426	0.000779666356493181\\
0.400791122324617	7.98427711288241e-05\\
0.438617460283521	0\\
};
\addlegendentry{PDAP}

\addplot [color=mycolor2]
  table[row sep=crcr]{%
0	280396.460334535\\
0.0277072419765229	172.402160364329\\
0.0482259616412581	77.4077916686927\\
0.0705800850991041	39.9485675580459\\
0.092319271312194	32.2455507986074\\
0.112099467596057	8.76060171108588\\
0.137618935619748	4.19388073633172\\
0.159079732301311	1.88095329659131\\
0.187560991254891	0.563801182415091\\
0.225725159828091	0.176466474949621\\
0.253657337125019	0.0843199148378062\\
0.277064722358827	0.0583570323374261\\
0.301586948619812	0.0198952803254997\\
0.324651371635057	0.00661927371697857\\
0.358955718531506	0.000779666356493181\\
0.408205007590498	7.98427711288241e-05\\
0.444495712972268	0\\
};
\addlegendentry{PDAP+PP}

\end{axis}
\end{tikzpicture}%
}\\
\caption{Residual for \ref{alg:PDAP}.} \label{fig:timeresidualPDAP}
\end{subfigure}
\caption{Residual $r_F(\om^k)$ plotted over the first second of the running time.}
\end{figure}
Last, we report on the computational time for the
setup considered before, in order to account for the numerical effort of the additional point removal steps.
The evolution of the residuals in the first second of the running time
for \ref{alg:GCG} and \ref{alg:SPINAT} can be found in Figure~\ref{fig:timeresidualGCG}.
We observe that neither
the additional projected gradient steps nor the additional application of Algorithm~\ref{alg:postprocessing} lead to a significant increase of the computational time. For \ref{alg:PDAP}, the measurement
times and residuals for all iterations are shown in Figure~\ref{fig:timeresidualPDAP}. We point
out that \ref{alg:PDAP} converges after $12$ iterations computed in approximately $0.4$ seconds
in this example. This is comparable to the elapsed computation time for computing
$25$ iterations of the \ref{alg:GCG} method. The small average time for a single iteration of \ref{alg:PDAP} is on the one hand a consequence of the uniformly bounded, low dimension of
the sub-problem \eqref{PDAP}. On the other hand, using the intermediate iterate $\om^{k+1/2}$ to
warm-start the semi-smooth Newton method greatly benefits its convergence behavior,
restricting the additional numerical effort in of \ref{alg:PDAP} in comparison to \ref{alg:GCG} to the
solution of a few low-dimensional Newton systems in each iteration. These results again
underline the practical efficiency of the presented acceleration strategies.

\subsubsection{Mesh-independence} \label{subsubsec:meshindependence}
To finish our numerical studies on Example~1 we examine
the influence of the mesh-size $h$ on the performance of
Algorithm~\ref{alg:conditionalgradient}. We again consider
the A-optimal design problem for $\beta=1$ and $\fish_0=0$ on consecutively refined meshes $\mathcal{T}_{h_l}$ ,
$l=5,\dots,9$. On each refinement level \(l\) the optimal design problem is solved using \ref{alg:GCG}
and \ref{alg:PDAP}, respectively. The computed residuals are shown in Figure~\ref{fig:independenceAlg}. For both
versions we observe that the convergence rate of the objective function value is stable
with respect to {mesh-refinement}. We point out
that this indicates a better convergence behavior of \ref{alg:PDAP} also on the continuous
level. A theoretical investigation of this improved rate is beyond the scope of this work
but will be given in a future manuscript.
\begin{figure}[htb]
\centering
\begin{subfigure}[t]{.45\linewidth}
\centering
\scalebox{.8}
{\input{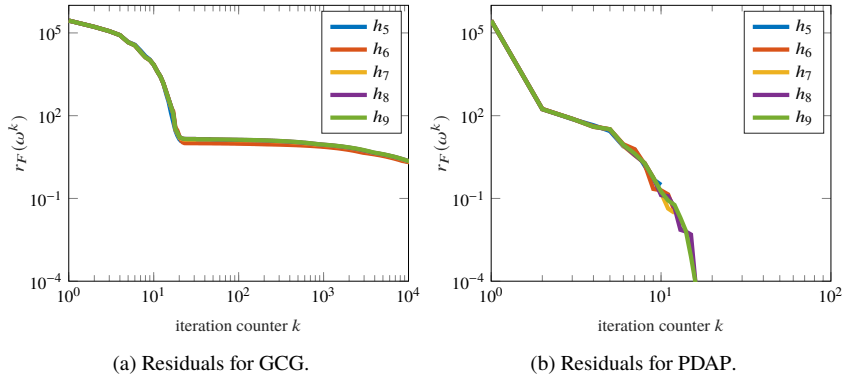}}\\
\caption{Residuals for \ref{alg:GCG}.}\label{fig:independenceGCG}
\end{subfigure}
\begin{subfigure}[t]{.45\linewidth}
\centering
\scalebox{.8}
{% This file was created by matlab2tikz.
\definecolor{mycolor1}{rgb}{0.00000,0.44700,0.74100}%
\definecolor{mycolor2}{rgb}{0.85000,0.32500,0.09800}%
\definecolor{mycolor3}{rgb}{0.92900,0.69400,0.12500}%
\definecolor{mycolor4}{rgb}{0.49400,0.18400,0.55600}%
\definecolor{mycolor5}{rgb}{0.46600,0.67400,0.18800}%
\begin{tikzpicture}

\begin{axis}[%
width=2.2in,
height=1.8in,
scale only axis,
xmode=log,
xmin=1,
xmax=100,
xminorticks=true,
xlabel style={font=\color{white!15!black}},
xlabel={iteration counter $k$},
ymode=log,
ymin=1e-04,
ymax=1000000,
yminorticks=true,
ylabel style={font=\color{white!15!black},at={(axis description cs:0.07,0.5)}},
ylabel={${r}_{F}{(}\omega^{k}{)}$},
axis background/.style={fill=white},
legend style={legend cell align=left, align=left, draw=white!15!black}
]
\addplot [color=mycolor1, line width=2pt]
  table[row sep=crcr]{%
1	279400.769073642\\
2	174.987298973513\\
3	74.9433053392409\\
4	45.0668130236877\\
5	27.7543999899349\\
6	7.77169524714964\\
7	5.80043882471318\\
8	1.54501568497221\\
9.00000000000001	0.495836814160157\\
10	0.310438267687914\\
};
\addlegendentry{${h}_{5}$}

\addplot [color=mycolor2, line width=2pt]
  table[row sep=crcr]{%
1	280148.3952623\\
2	176.817072817774\\
3	77.785015171622\\
4	41.4113211751264\\
5	32.1440574235667\\
6	9.22761908344001\\
7	6.08241292512594\\
8	1.78302326151243\\
9.00000000000001	0.221272545287775\\
10	0.196611820710587\\
11	0.140133131538732\\
12	0.0353571199780163\\
};
\addlegendentry{${h}_{6}$}

\addplot [color=mycolor3, line width=2pt]
  table[row sep=crcr]{%
1	280337.255976643\\
2	171.061305723443\\
3	76.2477079334453\\
4	39.0685983893177\\
5	32.0736057356462\\
6	8.0021963262393\\
7	3.64718603813071\\
8	1.95783708510498\\
9.00000000000001	0.564893475284862\\
10	0.161508521925043\\
11	0.0431134331652174\\
12	0.0300318815739047\\
};
\addlegendentry{${h}_{7}$}

\addplot [color=mycolor4, line width=2pt]
  table[row sep=crcr]{%
1	280384.608757359\\
2	170.924608016832\\
3	77.258603746674\\
4	39.7439568870178\\
5	31.6513303019369\\
6	8.51907166558204\\
7	3.74031609756503\\
8	1.88623197313859\\
9.00000000000001	0.559291200554298\\
10	0.132268659694091\\
11	0.119985496459322\\
12	0.0473737268566765\\
13	0.00728807906170914\\
14	0.00631128567147244\\
15	0.0049164653439675\\
16	6.89650421463738e-05\\
};
\addlegendentry{${h}_{8}$}

\addplot [color=mycolor5, line width=2pt]
  table[row sep=crcr]{%
1	280396.460334535\\
2	172.402160364329\\
3	77.4077916686927\\
4	39.9485675580459\\
5	32.2455507986074\\
6	8.76060171108588\\
7	4.19388073633172\\
8	1.88095329659132\\
9.00000000000001	0.563801182415091\\
10	0.176466474949621\\
11	0.0843199148378062\\
12	0.058357032337426\\
13	0.0198952803254997\\
14	0.00661927371697857\\
15	0.00077966635649318\\
16	7.98427711288242e-05\\
};
\addlegendentry{${h}_{9}$}

\end{axis}
\end{tikzpicture}%
}\\
\caption{Residuals for \ref{alg:PDAP}.}
\label{fig:independencePDAP}
\end{subfigure}
\caption{Evolution of residuals $r_F(\om^k)$ over iterations \(k\) on different refinement levels.}
\label{fig:independenceAlg}
\end{figure}
Additionally, in Figure~\ref{fig:independencesupport}, we plot the support size over the
iteration counter for each refinement level. For \ref{alg:GCG} we observe a
monotonic growth of the support size up to a certain threshold.
Note that the upper bound on the support size seems
to depend on the spatial discretization: the finer the grid, the more clusterization
around the true support points can be observed. In contrast, for \ref{alg:PDAP}, the
evolution of the support size admits a mesh-independent behavior in this example.
\begin{figure}[htb]
\centering
\begin{subfigure}[t]{.45\linewidth}
\centering
\scalebox{.8}
{% This file was created by matlab2tikz.
\definecolor{mycolor1}{rgb}{0.00000,0.44700,0.74100}%
\definecolor{mycolor2}{rgb}{0.85000,0.32500,0.09800}%
\definecolor{mycolor3}{rgb}{0.92900,0.69400,0.12500}%
\definecolor{mycolor4}{rgb}{0.49400,0.18400,0.55600}%
\definecolor{mycolor5}{rgb}{0.46600,0.67400,0.18800}%
\begin{tikzpicture}

\begin{axis}[%
width=2.2in,
height=1.8in,
scale only axis,
xmode=log,
xmin=1,
xmax=10000,
xminorticks=true,
xlabel style={font=\color{white!15!black}},
xlabel={iteration counter $k$},
ymode=log,
ymin=1,
ymax=100,
yminorticks=true,
ylabel style={font=\color{white!15!black},at={(axis description cs:0.07,0.5)}},
ylabel={$\text{\# supp }\omega{}^{k}$},
axis background/.style={fill=white},
legend style={at={(0.03,0.97)}, anchor=north west, legend cell align=left, align=left, draw=white!15!black}
]
\addplot [color=mycolor1, line width=2pt]
  table[row sep=crcr]{%
1	3\\
2	4\\
3	5\\
4	6\\
5	7\\
6	8\\
7	9.00000000000001\\
8	10\\
10	12\\
12	14\\
13	14\\
17	18\\
18	19\\
19	19\\
21	21\\
35	21\\
36	22\\
50	22\\
51	23\\
20001	23\\
};
\addlegendentry{$h_5$}

\addplot [color=mycolor2,line width=2pt]
  table[row sep=crcr]{%
1	3\\
2	4\\
3	5\\
4	6\\
5	7\\
6	8\\
7	9.00000000000001\\
8	10\\
10	12\\
12	14\\
15	17\\
19	21\\
20	22\\
21	22\\
25	26\\
26	26\\
27	27\\
28	27\\
29	28\\
37	28\\
38	29\\
125	29\\
126	30\\
177	30\\
178	31\\
209	31\\
210	32\\
20001	32\\
};
\addlegendentry{$h_{6}$}

\addplot [color=mycolor3, line width=2pt]
  table[row sep=crcr]{%
1	3\\
2	4\\
3	5\\
4	6\\
5	7\\
6	8\\
7	9.00000000000001\\
8	10\\
10	12\\
12	14\\
15	17\\
19	21\\
24	26\\
28	30\\
46	30\\
47.0000000000001	31\\
48	31\\
49	32\\
51	32\\
52	33\\
80	33\\
81	34\\
93.9999999999999	34\\
96.0000000000001	36\\
171	36\\
172	37\\
247	37\\
248	38\\
277	38\\
278	39\\
373	39\\
374	40\\
20001	40\\
};
\addlegendentry{$h_{7}$}

\addplot [color=mycolor4, line width=2pt]
  table[row sep=crcr]{%
1	3\\
2	4\\
3	5\\
4	6\\
5	7\\
6	8\\
7	9.00000000000001\\
8	10\\
10	12\\
12	14\\
15	17\\
19	21\\
24	26\\
25	26\\
26	27\\
27	27\\
28	28\\
49	28\\
52	31\\
59.0000000000001	31\\
60	32\\
61	32\\
62.0000000000001	33\\
68	33\\
70	35\\
71	35\\
72.0000000000001	36\\
85	36\\
85.9999999999999	37\\
100	37\\
101	38\\
121	38\\
122	39\\
135	39\\
136	40\\
209	40\\
210	41\\
211	41\\
212	42\\
217	42\\
218	43.0000000000001\\
299	43.0000000000001\\
301	45\\
20001	45\\
};
\addlegendentry{$h_{8}$}

\addplot [color=mycolor5, line width=2pt]
  table[row sep=crcr]{%
1	3\\
2	4\\
3	5\\
4	6\\
5	7\\
6	8\\
7	9.00000000000001\\
8	10\\
10	12\\
12	14\\
15	17\\
19	21\\
24	26\\
25	27\\
27	27\\
28	28\\
30	28\\
31	29\\
33	29\\
37	33\\
61	33\\
62.0000000000001	34\\
63.0000000000001	34\\
64.0000000000001	35\\
72.0000000000001	35\\
73	36\\
74	36\\
76	38\\
77.0000000000001	38\\
78.0000000000001	39\\
80	39\\
81	40\\
85.9999999999999	40\\
88.0000000000001	42\\
89.0000000000001	42\\
90.0000000000001	43.0000000000001\\
102	43.0000000000001\\
103	44\\
120	44\\
121	45\\
132	45\\
134	47.0000000000001\\
135	47.0000000000001\\
136	48\\
148	48\\
149	49\\
150	49\\
151	50\\
161	50\\
162	51\\
165	51\\
166	52\\
167	52\\
168	53\\
236	53\\
237	54.0000000000001\\
257	54.0000000000001\\
258	55\\
261	55\\
262	56\\
270	56\\
271	57\\
830	57\\
832.000000000001	58.0000000000001\\
20001	58.0000000000001\\
};
\addlegendentry{$h_{9}$}

\end{axis}
\end{tikzpicture}%
}\\
\caption{$\#\supp \om^k$ over $k$ for \ref{alg:GCG}.}\label{fig:independencesupp}
\end{subfigure}
\begin{subfigure}[t]{.45\linewidth}
\centering
\scalebox{.8}
{% This file was created by matlab2tikz.
\definecolor{mycolor1}{rgb}{0.00000,0.44700,0.74100}%
\definecolor{mycolor2}{rgb}{0.85000,0.32500,0.09800}%
\definecolor{mycolor3}{rgb}{0.92900,0.69400,0.12500}%
\definecolor{mycolor4}{rgb}{0.49400,0.18400,0.55600}%
\definecolor{mycolor5}{rgb}{0.46600,0.67400,0.18800}%
\begin{tikzpicture}

\begin{axis}[%
width=2.2in,
height=1.8in,
scale only axis,
xmode=log,
xmin=1,
xmax=100,
xminorticks=true,
ylabel style={font=\color{white!15!black},at={(axis description cs:0.07,0.5)}},
xlabel={iteration counter $k$},
ymode=log,
ymin=3,
ymax=100,
yminorticks=true,
ylabel style={font=\color{white!15!black}},
ylabel={$\text{\# supp }\omega{}^{k}$},
axis background/.style={fill=white},
legend style={at={(0.03,0.97)}, anchor=north west, legend cell align=left, align=left, draw=white!15!black}
]
\addplot [color=mycolor1, line width=2pt]
  table[row sep=crcr]{%
1	3\\
2	4\\
4	4\\
5	5\\
11	5\\
};
\addlegendentry{${h}_{5}$}

\addplot [color=mycolor2, line width=2pt]
  table[row sep=crcr]{%
1	3\\
2	4\\
4	4\\
5	5\\
13	5\\
};
\addlegendentry{${h}_{6}$}

\addplot [color=mycolor3, line width=2pt]
  table[row sep=crcr]{%
1	3\\
2	4\\
4	4\\
5	5\\
13	5\\
};
\addlegendentry{${h}_{7}$}

\addplot [color=mycolor4, line width=2pt]
  table[row sep=crcr]{%
1	3\\
2	4\\
4	4\\
5	5\\
17	5\\
};
\addlegendentry{${h}_{8}$}

\addplot [color=mycolor5,line width=2pt]
  table[row sep=crcr]{%
1	3\\
2	4\\
4	4\\
5	5\\
17	5\\
};
\addlegendentry{$h_{9}$}

\end{axis}
\end{tikzpicture}%
}\\
\caption{$\#\supp \om^k$ over $k$ for \ref{alg:PDAP}.}
\end{subfigure}
\caption{Evolution of the support size on different refinement levels.}
\label{fig:independencesupport}
\end{figure}
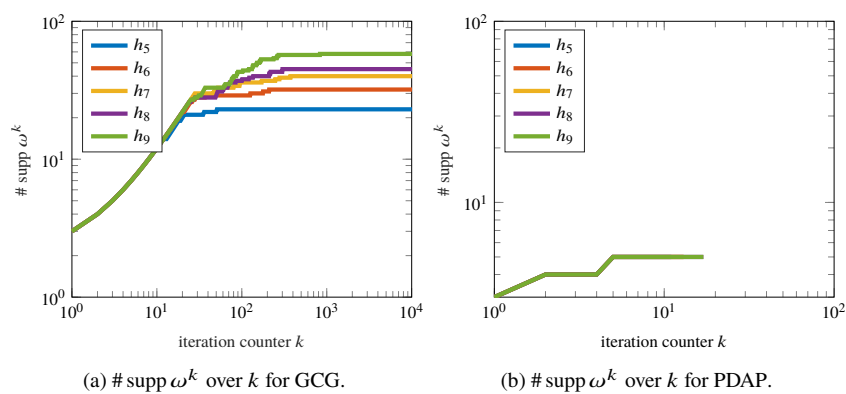

\bibliographystyle{siam2}
\bibliography{SparseDesign3}

\begin{thebibliography}{10}

\bibitem{Adamssobolev}
{\sc R.~A. Adams}, {\em Sobolev spaces}, Pure and applied mathematics, Academic
  Press, New York, 1978.

\bibitem{alexanderian2014optimal}
{\sc A.~Alexanderian, N.~Petra, G.~Stadler, and O.~Ghattas}, {\em A-optimal
  design of experiments for infinite-dimensional {B}ayesian linear inverse
  problems with regularized {$\ell_0$}-sparsification}, SIAM J. Sci. Comput.,
  36 (2014), pp.~A2122--A2148.

\bibitem{atkinson2007optimum}
{\sc A.~C. Atkinson, A.~N. Donev, and R.~D. Tobias}, {\em Optimum experimental
  designs, with {SAS}}, vol.~34 of Oxford Statistical Science Series, Oxford
  University Press, Oxford, 2007.

\bibitem{atwood1973sequences}
{\sc C.~L. Atwood}, {\em Sequences converging to {$D$}-optimal designs of
  experiments}, Ann. Statist., 1 (1973), pp.~342--352.

\bibitem{avery2012experimental}
{\sc M.~Avery, H.~T. Banks, K.~Basu, Y.~Cheng, E.~Eager, S.~Khasawinah,
  L.~Potter, and K.~L. Rehm}, {\em Experimental design and inverse problems in
  plant biological modeling}, J. Inverse Ill-Posed Probl., 20 (2012),
  pp.~169--191.

\bibitem{banks2014experimental}
{\sc H.~T. Banks and K.~L. Rehm}, {\em Experimental design for vector output
  systems}, Inverse Probl. Sci. Eng., 22 (2014), pp.~557--590.

\bibitem{bates2007nonlinear}
{\sc D.~M. Bates and D.~G. Watts}, {\em Nonlinear regression analysis and its
  applications}, Wiley Series in Probability and Mathematical Statistics:
  Applied Probability and Statistics, John Wiley \& Sons, Inc., New York, 1988.

\bibitem{bauer2000numerical}
{\sc I.~Bauer, H.~G. Bock, S.~K\"orkel, and J.~P. Schl{\"o}der}, {\em Numerical
  methods for optimum experimental design in {DAE} systems}, J. Comput. Appl.
  Math., 120 (2000), pp.~1--25.
\newblock SQP-based direct discretization methods for practical optimal control
  problems.

\bibitem{beal1960}
{\sc E.~M.~L. Beale}, {\em Confidence regions in non-linear estimation}, J.
  Roy. Statist. Soc. Ser. B, 22 (1960), pp.~41--88.

\bibitem{becker2005parameter}
{\sc R.~Becker, M.~Braack, and B.~Vexler}, {\em Parameter identification for
  chemical models in combustion problems}, Appl. Numer. Math., 54 (2005),
  pp.~519--536.

\bibitem{bock1987randwertproblemmethoden}
{\sc H.~G. Bock}, {\em Randwertproblemmethoden zur {P}arameteridentifizierung
  in {S}ystemen nichtlinearer {D}ifferentialgleichungen}, vol.~183 of Bonner
  Mathematische Schriften [Bonn Mathematical Publications], Universit\"at Bonn,
  Mathematisches Institut, Bonn, 1987.
\newblock Dissertation, Rheinische Friedrich-Wilhelms-Universit\"at, Bonn,
  1985.

\bibitem{bonnans2000perturbation}
{\sc J.~F. Bonnans and A.~Shapiro}, {\em Perturbation analysis of optimization
  problems}, Springer Series in Operations Research, Springer-Verlag, New York,
  2000.

\bibitem{boyd2015alternating}
{\sc N.~Boyd, G.~Schiebinger, and B.~Recht}, {\em The alternating descent
  conditional gradient method for sparse inverse problems}, SIAM J. Optim., 27
  (2017), pp.~616--639.

\bibitem{bredies2009generalized}
{\sc K.~Bredies, D.~A. Lorenz, and P.~Maass}, {\em A generalized conditional
  gradient method and its connection to an iterative shrinkage method}, Comput.
  Optim. Appl., 42 (2009), pp.~173--193.

\bibitem{bredies2013inverse}
{\sc K.~Bredies and H.~K. Pikkarainen}, {\em Inverse problems in spaces of
  measures}, ESAIM Control Optim. Calc. Var., 19 (2013), pp.~190--218.

\bibitem{Brezis:2010}
{\sc H.~Brezis}, {\em Functional analysis, {S}obolev spaces and partial
  differential equations}, Universitext, Springer, New York, 2011.

\bibitem{casas2012}
{\sc E.~Casas, C.~Clason, and K.~Kunisch}, {\em Approximation of elliptic
  control problems in measure spaces with sparse solutions}, SIAM J. Control
  Optim., 50 (2012), pp.~1735--1752.

\bibitem{chung2012experimental}
{\sc M.~Chung and E.~Haber}, {\em Experimental design for biological systems},
  SIAM J. Control Optim., 50 (2012), pp.~471--489.

\bibitem{dunn1979rates}
{\sc J.~C. Dunn}, {\em Rates of convergence for conditional gradient algorithms
  near singular and nonsingular extremals}, SIAM J. Control Optim., 17 (1979),
  pp.~187--211.

\bibitem{dunn1980convergence}
\leavevmode\vrule height 2pt depth -1.6pt width 23pt, {\em Convergence rates
  for conditional gradient sequences generated by implicit step length rules},
  SIAM J. Control Optim., 18 (1980), pp.~473--487.

\bibitem{elstrodt2013mass}
{\sc J.~Elstrodt}, {\em Ma{\ss}- und Integrationstheorie}, Springer-Lehrbuch,
  Springer Berlin Heidelberg, 2013.

\bibitem{fedorov1972theory}
{\sc V.~V. Fedorov}, {\em Theory of optimal experiments}, Academic Press, New
  York-London, 1972.
\newblock Translated from the Russian and edited by W. J. Studden and E. M.
  Klimko, Probability and Mathematical Statistics, No. 12.

\bibitem{fedorov2012model}
{\sc V.~V. Fedorov and P.~Hackl}, {\em Model-oriented design of experiments},
  vol.~125 of Lecture Notes in Statistics, Springer-Verlag, New York, 1997.

\bibitem{fedorov2013optimal}
{\sc V.~V. Fedorov and S.~L. Leonov}, {\em Optimal design for nonlinear
  response models}, Chapman \& Hall/CRC Biostatistics Series, CRC Press, Boca
  Raton, FL, 2014.

\bibitem{frank1956algorithm}
{\sc M.~Frank and P.~Wolfe}, {\em An algorithm for quadratic programming},
  Naval Res. Logist. Quart., 3 (1956), pp.~95--110.

\bibitem{haber2008numerical}
{\sc E.~Haber, L.~Horesh, and L.~Tenorio}, {\em Numerical methods for
  experimental design of large-scale linear ill-posed inverse problems},
  Inverse Problems, 24 (2008), pp.~055012, 17.

\bibitem{herzog2015sequentially}
{\sc R.~Herzog and I.~Riedel}, {\em Sequentially optimal sensor placement in
  thermoelastic models for real time applications}, Optim. Eng., 16 (2015),
  pp.~737--766.

\bibitem{jaggi2013revisiting}
{\sc M.~Jaggi}, {\em Revisiting frank-wolfe: Projection-free sparse convex
  optimization}, in Proceedings of the 30th International Conference on
  International Conference on Machine Learning - Volume 28, ICML'13, JMLR.org,
  2013, pp.~I--427--I--435.

\bibitem{kiefer1974general}
{\sc J.~Kiefer}, {\em General equivalence theory for optimum designs
  (approximate theory)}, Ann. Statist., 2 (1974), pp.~849--879.

\bibitem{kiefer1959optimum}
{\sc J.~Kiefer and J.~Wolfowitz}, {\em Optimum designs in regression problems},
  Ann. Math. Statist., 30 (1959), pp.~271--294.

\bibitem{korkel1999sequential}
{\sc S.~K{\"o}rkel, I.~Bauer, H.~G. Bock, and J.~Schl{\"o}der}, {\em A
  sequential approach for nonlinear optimum experimental design in {DAE}
  systems}, Scientific Computing in Chemical Engineering II, 2 (1999),
  pp.~338--345.

\bibitem{milzarekfilter}
{\sc A.~Milzarek and M.~Ulbrich}, {\em A semismooth {N}ewton method with
  multidimensional filter globalization for {$l_1$}-optimization}, SIAM J.
  Optim., 24 (2014), pp.~298--333.

\bibitem{pazman1986foundations}
{\sc A.~P{\'a}zman}, {\em Foundations of optimum experimental design}, vol.~14
  of Mathematics and its Applications (East European Series), D. Reidel
  Publishing Co., Dordrecht, 1986.
\newblock Translated from the Czech.

\bibitem{walter2017Helmholtz}
{\sc K.~Pieper, B.~Q. Tang, P.~Trautmann, and D.~Walter}, {\em Inverse point
  source location for the {H}elmholtz equation}, submitted.

\bibitem{pronzato2003removing}
{\sc L.~Pronzato}, {\em Removing non-optimal support points in {$D$}-optimum
  design algorithms}, Statist. Probab. Lett., 63 (2003), pp.~223--228.

\bibitem{pukelsheim1993optimal}
{\sc F.~Pukelsheim}, {\em Optimal design of experiments}, Wiley Series in
  Probability and Mathematical Statistics: Probability and Mathematical
  Statistics, John Wiley \& Sons, Inc., New York, 1993.
\newblock A Wiley-Interscience Publication.

\bibitem{rakotomamonjy2015generalized}
{\sc A.~{Rakotomamonjy}, R.~{Flamary}, and N.~{Courty}}, {\em {Generalized
  conditional gradient: analysis of convergence and applications}}, ArXiv
  e-prints,  (2015).

\bibitem{Rudin}
{\sc W.~Rudin}, {\em Real and complex analysis}, McGraw-Hill Book Co., New
  York, third~ed., 1987.

\bibitem{john1975review}
{\sc R.~C. St.~John and N.~R. Draper}, {\em {$D$}-optimality for regression
  designs: a review}, Technometrics, 17 (1975), pp.~15--23.

\bibitem{tarantola2005inverse}
{\sc A.~Tarantola}, {\em Inverse problem theory and methods for model parameter
  estimation}, Society for Industrial and Applied Mathematics (SIAM),
  Philadelphia, PA, 2005.

\bibitem{troeltzsch2010optimale}
{\sc F.~Tr{\"o}ltzsch}, {\em Optimal control of partial differential
  equations}, vol.~112 of Graduate Studies in Mathematics, American
  Mathematical Society, Providence, RI, 2010.
\newblock Theory, methods and applications, Translated from the 2005 German
  original by J\"urgen Sprekels.

\bibitem{ucinski2004optimal}
{\sc D.~Uci\'nski}, {\em Optimal measurement methods for distributed parameter
  system identification}, Systems and Control Series, CRC Press, Boca Raton,
  FL, 2005.

\bibitem{ulbrich2002semismooth}
{\sc M.~Ulbrich}, {\em Semismooth {N}ewton methods for operator equations in
  function spaces}, SIAM J. Optim., 13 (2002), pp.~805--842 (2003).

\bibitem{Wolfe1970away}
{\sc P.~Wolfe}, {\em Convergence theory in nonlinear programming},
  North-Holland, Amsterdam, 1970.

\bibitem{wynn1970}
{\sc H.~P. Wynn}, {\em The sequential generation of {$D$}-optimum experimental
  designs}, Ann. Math. Statist., 41 (1970), pp.~1655--1664.

\bibitem{yu2011cocktail}
{\sc Y.~Yu}, {\em D-optimal designs via a cocktail algorithm}, Stat. Comput.,
  21 (2011), pp.~475--481.

\end{thebibliography}

\end{document}